\documentclass[12pt]{amsart}
\usepackage[margin=1 in]{geometry}
\usepackage{subcaption}
\usepackage{amssymb}
\usepackage{prettyref}
\usepackage{enumerate,multirow,multicol,longtable,floatrow,caption,mdwlist,mathrsfs,ytableau}
\usepackage{tikz}
\usetikzlibrary{calc}
\usepackage{verbatim}

\usepackage[toc,page]{appendix}

\usepackage{color}
\usepackage{hyperref}
 
\DeclareMathOperator\Ens{Ens}
\DeclareMathOperator\Chn{Chn}
\usepackage{stmaryrd}

\usepackage{mathtools}
\DeclarePairedDelimiter\ceil{\lceil}{\rceil}
\DeclarePairedDelimiter\floor{\lfloor}{\rfloor}

\newtheorem{theorem}{Theorem}[section]
\newtheorem{lemma}[theorem]{Lemma}
\newtheorem{proposition}[theorem]{Proposition}
\newtheorem{corollary}[theorem]{Corollary}

\theoremstyle{definition}
\newtheorem{definition}[theorem]{Definition}

\theoremstyle{remark}
\newtheorem{example}[theorem]{Example}
\newtheorem{remark}[theorem]{Remark}
\numberwithin{equation}{section}

\def \h {\mathfrak h}

\newcommand{\Exterior}{\mathchoice{{\textstyle\bigwedge}}%
    {{\bigwedge}}%
    {{\textstyle\wedge}}%
    {{\scriptstyle\wedge}}}
\newcommand{\slngeo}{\mathrm{I}(\mathfrak{sl}_n,V \oplus \Exterior^2)}
\newcommand{\slngeoForTitle}{\texorpdfstring{$\slngeo$}{\mathrm{I}(sl(n))} }
\newcommand{\glngeo}{\mathrm{I}(\mathfrak{gl}_n,V \oplus \Exterior^2)}
\newcommand{\glngeoForTitle}{\texorpdfstring{$\glngeo$}{\mathrm{I}(gl(n))} }
\newcommand{\slngeoSm}[1]{\mathrm{I}(\mathfrak{sl}_{#1},V \oplus \Exterior^2)}

\newcommand{\glngeoSm}[1]{\mathrm{I}(\mathfrak{gl}_{#1},V \oplus \Exterior^2)}

\newcommand{\rayFn}{\vec{v}}
\newcommand{\ray}[1]{\rayFn\left(#1 \right)}
\newcommand{\indFn}{\vec{v}^{-1}}
\newcommand{\ind}[1]{\indFn\left(#1 \right)}
\newcommand{\test}[2]{\text{\bf Test}_{#1}\left(#2 \right)}

\newcommand{\posetz}{\mathbb{E}^*}
\newcommand{\poset}{\mathbb{E}}
\newcommand{\poseti}{\widehat{\mathbb{E}}}

\newcommand{\weights}{\mathfrak{W}}
\DeclareMathOperator{\erays}{ExtRays}
\DeclareMathOperator{\vecspan}{{Span}_{\geq 0}}

\usepackage{xr}
\newcommand{\pref}{\prettyref}
\newrefformat{assert}{(\ref{#1})}
\newrefformat{prop}{Proposition \ref{#1}}
\newrefformat{cor}{Corollary \ref{#1}}
\newrefformat{def}{Definition \ref{#1}}
\newrefformat{rem}{Remark \ref{#1}}
\newrefformat{exmp}{Example \ref{#1}}
\newrefformat{app}{Appendix \ref{#1}}
\newrefformat{gln-thm}{\protect{\cite[Theorem \ref*{#1}]{gl}}}
\newrefformat{gln-def}{\protect{\cite[Definition \ref*{#1}]{gl}}}
\newrefformat{gln-sec}{\protect{\cite[Section \ref*{#1}]{gl}}}

\begin{document}
\title[Incidence Geometry in a   Weyl Chamber II: $SL_n$ ]
 {Incidence Geometry in a Weyl Chamber II: $SL_n$}
\author[Esole]{Mboyo Esole}
\author[Jackson]{Steven Glenn Jackson}
\author[Jagadeesan]{Ravi Jagadeesan}
\author[No\"el]{Alfred G. No\"el}

\address{Department of Mathematics\\
Harvard University\\ Cambridge, MA 02138, USA.}
\email{esole@math.harvard.edu}
\email{rjagadeesan@college.harvard.edu}
\address{Department of Mathematics\\
         University of Massachusetts\\
         Boston, MA 02125, USA.}
\email{jackson@math.umb.edu}
\email{anoel@math.umb.edu}

\subjclass[2010]{05E10, 52C35, 05A15, 17B10, 17B81}
\begin{abstract}
We study the polyhedral geometry of the hyperplanes orthogonal to the weights of the  first and the second fundamental representations of $\mathfrak{sl}_n$ inside the dual fundamental Weyl chamber.
We obtain generating functions that enumerate the flats and the faces of a fixed dimension.
In addition, we describe the extreme rays of the incidence geometry and classify simplicial faces.

From the perspective of supersymmetric gauge theories with 8 supercharges in five dimensional spacetime, the poset of flats is isomorphic to the network of mixed Coulomb-Higgs branches.
On the other hand, the poset of faces is conjectured to be isomorphic to the network of crepant partial resolutions of an elliptic fibration with gauge algebra $\mathfrak{sl}_n$ and ``matter representation" given by the sum of the first two fundamental representations.

\end{abstract}

\keywords{
Hyperplane arrangement, representation theory, Lie algebra, Weyl chambers, root systems, weights lattice, poset, extreme rays.}

\maketitle

\tableofcontents{}

\section{Introduction}
\label{sec:intro}

The theory of hyperplane arrangements has connections to many different areas of pure and applied mathematics \cite{MR1217488,MR2383131,MR0357135}.
A new type of hyperplane arrangement  \cite{gl} defined by a representation  ${\bf R}$ of a reductive Lie algebra $\mathfrak{g}$ has emerged from string geometry \cite{Intriligator:1997pq} and is relevant in the study of elliptic fibrations  \cite{box,ESY1,ESY2,ES1}.
Its hyperplanes are the kernels of the weights of ${\bf R}$. Moreover, the arrangement lives in a dual fundamental Weyl chamber of $\mathfrak{g}$ instead of a full Cartan subalgebra.  We will review the physical motivation for considering such hyperplane arrangements before discussing our results.

\subsection{Physical motivation}
In classical quantum field theory, interactions between fundamental particles are modeled by gauge theories, which are characterized by a reductive Lie group $G$ called the gauge group. 
The gauge theory is said to be non-abelian when $G$ is.
A non-abelian gauge theory is said to be in a {\em Coulomb phase} when its gauge group is broken to a Cartan subgroup.
In the Coulomb phase, the gauge theory is effectively an abelian gauge theory that can be thought as a generalization of electro-magnetic theory in which there are $r$ distinct electromagnetic fields.
Charged particles transform according to representations of the gauge group.
Fundamental particles correspond to irreducible representations of $G$.
In a Coulomb phase of such a theory, the possible charges of particles with respect to the Cartan subalgebra are the weights of the representation under which the particles transform. 

In string geometry, compactifications of M-theory on a Calabi-Yau threefold naturally give rise to gauge theories and particles called \emph{hypermultiplets}.
The resulting gauge theories are {\em minimal supersymmetric}  in five dimensional spacetime \cite{Intriligator:1997pq,Ferrara:1996wv}.

For minimal supersymmetric gauge theories in five dimensional spacetime, the Coulomb phase  
is characterized by a real function $\mathcal{F}: \mathfrak{h}\to  \mathbb{R}$ called the Intriligator-Morrison-Seiberg potential \cite{Intriligator:1997pq}.  The relevant term of $\mathcal{F}(\varphi),$ namely
$$
{\mathcal F}(\varphi) = \cdots 
 +\frac{1}{12}\Big(\sum_{\alpha} | \alpha\cdot \varphi|^3 + \sum_{\lambda} |\lambda\cdot \varphi|^3\Big),
$$
depends on the Higgs field $\varphi$, the positive roots $\alpha$ of the gauge algebra and the weights $\lambda$ of the matter representation.
The potential is  singular along the hyperplanes that are kernels of the weights of the representation $\mathbf R$.
A phase of the Coulomb branch of the theory is a connected region in which the quantities $\alpha\cdot \varphi$ and $\lambda\cdot \varphi$ each take fixed signs. 
We fix the sign of $\alpha\cdot \varphi$ by requiring $\varphi$ to be in the dual fundamental Weyl chamber.  This motivates the following definition.

\begin{definition}[\pref{gln-def:incGeoGeneral}]
Let $\mathfrak{g}$ be a reductive Lie algebra over $\mathbb{C}$, let $\mathfrak{h}$ be a split, real form of a Cartan subalgebra of $\mathfrak{g},$ and let ${\bf R}$ be a representation of $\mathfrak{g}$.
We denote by $\mathrm{I}(\mathfrak{g},{\bf R})$ the real hyperplane arrangement consisting of the kernels of the weights of ${\bf R}$ restricted to a dual fundamental Weyl chamber in $\mathfrak{h}$.
\end{definition}

\begin{remark}
Because all Cartan subalgebras are conjugate and all fundamental Weyl chambers canonically related by the Weyl group action, the incidence geometry $\mathrm{I}(\mathfrak{g},{\bf R})$ is independent of the choice of $\mathfrak{h}$ and of a dual fundamental Weyl chamber in $\mathfrak{h}$.
\end{remark}
\begin{remark}
 In the rest of the paper, we call the ``dual  fundamental Weyl chamber" the {\em Weyl chamber}. 
\end{remark}

The condition that $\lambda\cdot \varphi\neq 0$ for all $\lambda\in \mathbf{R}$ in a given phase implies a correspondence between Coulomb phases and chambers of $\mathrm{I}(\mathfrak{g}, \mathbf{R})$.
Flats of $\mathrm{I}(\mathfrak{g},{\bf R})$ correspond to to the mixed \emph{Coulomb-Higgs branches} of the gauge theory.
A general framework to discuss the chamber structure of the Coulomb branch appeared in \cite{box} (see \cite{hls} for the chambers of $\slngeoSm{5}$), and the case relevant to this paper is described in \pref{prop:signFlows}.  The idea was to use box diagrams to represent the sign vectors for the chambers of $\mathrm{I}(\mathfrak{g},\mathbf{R})$.  Allowing zero entries lets the box diagrams describe arbitrary faces~\cite{ESY1,ESY2}.

When the Calabi-Yau variety is elliptically fibered, it has been conjectured that the gauge algebra $\mathfrak{g}$ and the representation $\mathbf{R}$ can be determined from the singular fibers of the elliptic fibration over points of codimensions one and two in the base of the fibration, respectively 
\cite{Intriligator:1997pq,grassi2003group,Morrison:2011mb,Grassi:2011hq,Marsano:2011hv,Krause:2011xj,Anderson:2015cqy}. 
Chambers are conjectured to correspond to crepant resolutions of the Weierstrass model of the elliptic fibration, and the resolutions corresponding to chambers that share a facet are conjectured to be related by an (extremal) flop 
\cite{box, ESY1, ESY2, ES1, EY,MR1259932, hls}. 

An important example arises when the generic fiber degenerates to a Kodaira fiber of type $\mathrm{I}_n^s$ over a divisor on the base of the fibration.
Since the dual graph of a Kodaira fiber of type $\mathrm{I}_n^s$ is the affine Dynkin diagram of type $\tilde{A}_{n-1}$, the corresponding Lie algebra is $\mathfrak{su}_n,$ which is the compact real form of $\mathfrak{sl}_n$.
If in addition, the elliptic fibration has a Mordell-Weil group of rank one, the gauge algebra is $\mathfrak{u}_n,$ which is the compact real form of $\mathfrak{gl}_n$.
 
The representation ${\bf R} = V\oplus\Exterior^2$ given by the sum of the defining representation of $\mathfrak{gl}_n$ and its second exterior power are conjectured to appear as the matter representations of the $SU(N)$ and $U(N)$ gauge theories
engineered from many Weierstrass models with codimension 1 singular fibers of type $\mathrm{I}_n^s$. 
This motivates the study of the incidence geometries $\glngeo$ and $\slngeo$.

\subsection{Connection to the literature}

In a recent paper \cite{gl}, we studied the incidence geometry $\glngeo$.  
A formula for the number of chambers appeared already in \cite{box} where the techniques of box graphs were also introduced. The chambers of the very important special case $\slngeoSm{5}$ were discussed in \cite{hls}. 
In \cite{gl}, we proved that all the  chambers of $\glngeo$  (and therefore faces) are simplicial and its extreme rays have the structure of a finite partial ordered set (poset) isomorphic to $\posetz_n=\{ (x,y) \in \mathbb{N}^2  \mid   0< x+y\leq n\}$   equipped with the Cartesian order induced from the usual order in $\mathbb{N}$  \pref{gln-thm:rayStruture}.
Under the identification of the set of extreme rays with $\posetz_n$, the $k$-faces of $\glngeo$ are in bijection with the $k$-chains of the poset $\posetz_n,$ while the $k$-flats are identified with certain unions of intervals of $\posetz_n$ called  {\em  $k$-ensembles} \pref{gln-def:Ensemble}. 
These identifications \cite[Theorems~\ref*{gln-thm:Characterization.Faces} and~\ref*{gln-thm:Characterization.Flats}]{gl} reduce the enumeration of the $k$-faces and $k$-flats of $\glngeo$ to the study of the poset $\posetz_n$.
It was proved that the generating functions of the numbers of $k$-faces and $k$-flats are simple rational functions \cite[Theorems~\ref*{gln-thm:faceCount} and~\ref*{gln-thm:flatCount}]{gl}.

\subsection{Results of this paper}

The aim of this  paper is to study the combinatorial properties of the family of hyperplane arrangements 
$\slngeo$,  
building on the results of \cite{gl} regarding the incidence geometry of the family $\glngeo$.
Figures~\ref{fig:YT345},\ref{fig:SU2_Phases},\ref{fig:SU3_Phases},\ref{fig:SU4_Phases}, and~\ref{fig:SU5_Phases} on pages~\pageref{fig:YT345} and~\pageref{fig:SU5_Phases} show the geometry of $\slngeo$ for $n \le 5$ \cite{ESY1,ESY2}.
We study the chambers, extreme rays, faces, flats of $\slngeo$ by relating $\slngeo$ to $\glngeo$. This allows us to exploit the partial order on the extreme rays of $\glngeo$.

We prove closed-form expressions for the generating functions of the number of $k$-faces and $k$-flats in Theorems~\ref{thm:slnFaces} and~\ref{thm:slnFlats}. 
In contrast to the case of $\glngeo$, the generating functions are not rational:  the generating function for the number of $k$-flats (resp.\ $k$-faces)  of $\slngeo$ generates a quadratic (resp.\ quartic) extension of the field of rational polynomials in two variables.  As a result, we can easily obtain counts of the numbers of $k$-faces and $k$-flats in $\slngeo$ for small $n$ and $k$.

\begin{example}
After a Taylor expansion from \pref{thm:slnFaces},  we can read off the number $b(n,k)$ of $k$-faces in $\slngeo$ as the coefficient of $x^ny^k$ in the generating function $B(x,y)$.  This yields the values
$$
\begin{array}{c|c|c|c|c|c|c|c|c|c|c}
b(n,k) & 0 & 1 & 2 & 3 & 4 & 5 & 6 & 7 & 8 & 9\\ \hline
0 & 1\\ \hline
1 & 1\\ \hline
2 & 1 & 1\\ \hline
3 & 1 & 3 & 2\\ \hline
4 & 1 & 6 & 9 & 4\\ \hline
5 & 1 & 14 & 37 & 36 & 12\\ \hline
6 & 1 & 23 & 87 & 133 & 92 & 24\\ \hline
7 & 1 & 43 & 219 & 467 & 502 & 270 & 58\\ \hline
8 & 1 & 64 & 414 & 1152 & 1713 & 1428 & 632 & 116\\ \hline
9 & 1 & 104 & 826 & 2864 & 5501 & 6300 & 4300 & 1620 & 260\\ \hline
10 & 1 & 145 & 1382 & 5814 & 13865 & 20461 & 19140 & 11092 & 3644 & 520
\end{array}.
$$
\end{example}

\begin{example}
After a Taylor expansion from \pref{thm:slnFlats},  we can read off the number $f(n,k)$ of $k$-flats in $\slngeo$ as the coefficient of $x^ny^k$ in the generating function $F(x,y)$.  This yields the values
$$
\begin{array}{c|c|c|c|c|c|c|c|c|c|c}
f(n,k) & 0 & 1 & 2 & 3 & 4 & 5 & 6 & 7 & 8 & 9\\ \hline
0 & 1\\ \hline
1 & 1\\ \hline
2 & 1 & 1\\ \hline
3 & 1 & 1 & 1\\ \hline
4 & 1 & 2 & 3 & 1\\ \hline
5 & 1 & 4 & 9 & 9 & 1\\ \hline
6 & 1 & 7 & 15 & 22& 13 & 1\\ \hline
7 & 1 & 13 & 41 & 70 & 58 & 20 & 1\\ \hline
8 & 1 & 20 & 72 & 137 & 161 & 99 & 26 & 1\\ \hline
9 & 1 & 34 & 156 & 357 & 489 & 396 & 178 & 35 & 1\\ \hline
10 & 1 & 49 & 258 & 669 & 1072 & 1126 & 734 & 271 & 43 & 1
\end{array}.
$$
\end{example}

We also study the chambers in $\slngeo$.  Figures~\ref{fig:YT_SU6},~\ref{fig:SU6_Phases} and~\ref{fig:SU7_Phases} on pages~\pageref{fig:YT_SU6} and~\pageref{fig:SU7_Phases} describe the chambers of $\slngeoSm{6}$ and $\slngeoSm{7}$, including when two chambers share a common facet.
We classify the extreme rays and simplicial chambers of $\slngeo$ in \pref{thm:extremeRaysOfSln}.
In contrast to the case of $\glngeo$, most chambers of $\slngeo$ are not simplicial for $n \gg 0$. \pref{fig:SimplicialSU6} describes the simplicial chambers of $\slngeoSm{6}$.

\begin{figure}
{\begin{centering}

\begin{minipage}{\linewidth}
 \centering{\begin{tabular}{c  c }
$\{3\}$ & $\{2,1\}$ 
 \\
 \includegraphics[scale=.7]{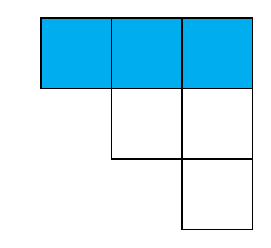}& 
    \includegraphics[scale=.7]{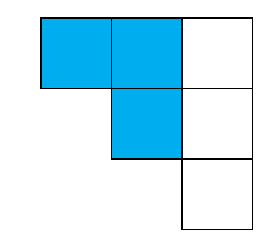}
   \end{tabular}
   }
\end{minipage}
   \subcaption{{\bf The 2 chambers of}  $\slngeoSm{3}$. 
\label{fig:YT_SU3}
 }
\vspace{1.5cm}

\begin{minipage}{\linewidth}
\centering
{ \begin{tabular}{c  c  c  c  }
$\{4\}$ & $\{4,1\}$ &  $\{3,2\}$ &  $\{3,2,1\}$ 
 \\
 \includegraphics[scale=.7]{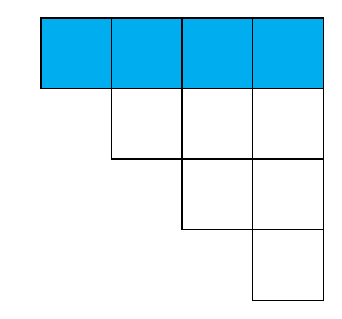}& 
    \includegraphics[scale=.7]{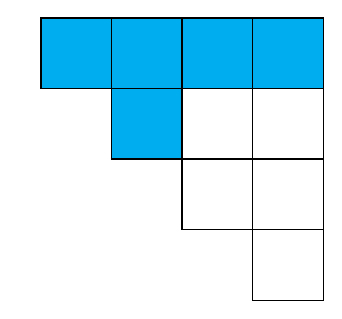}&
        \includegraphics[scale=.7]{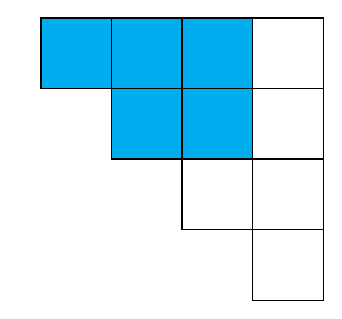}&
 \includegraphics[scale=.7]{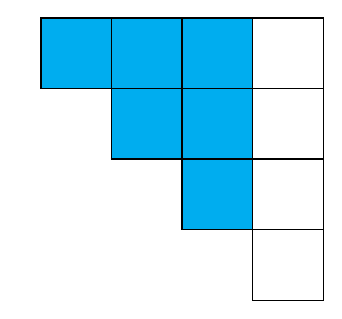}
   \end{tabular}
  }
 \end{minipage}
 \subcaption{{\bf The 4 chambers of}  $\slngeoSm{4}$. 
 \label{fig:YT_SU4}
}
\vspace{1.5cm}

\begin{minipage}{\linewidth}
\centering{
 \begin{tabular}{c  c  c  c  c  c  }
$\{5\}$ & $\{5,1\}$ &  $\{5,2\}$ &  $\{5,3\}$ & {$\{5,4\}$} &  $\{5,2,1\}$   
 \\
 \includegraphics[scale=.5]{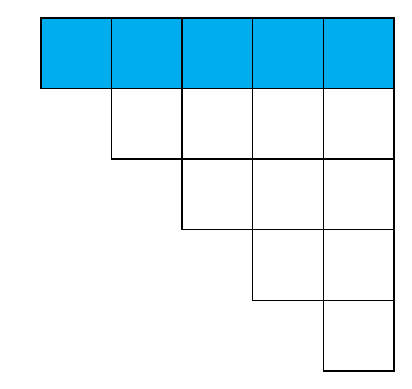}& 
    \includegraphics[scale=.5]{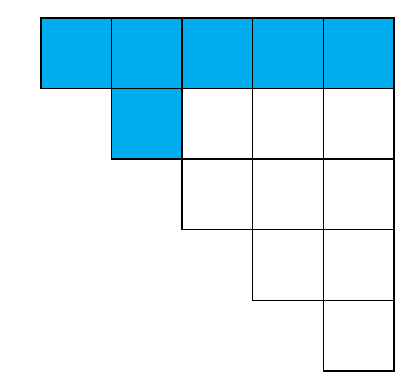}&
        \includegraphics[scale=.5]{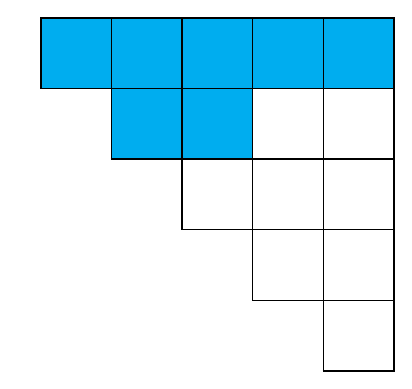}&
 \includegraphics[scale=.5]{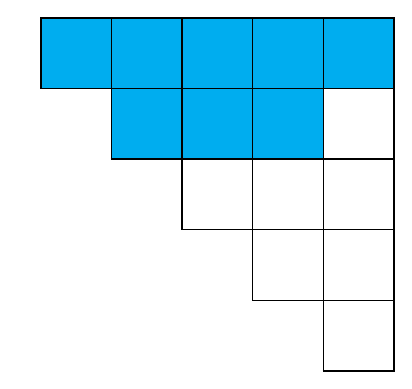}&
  \includegraphics[scale=.5]{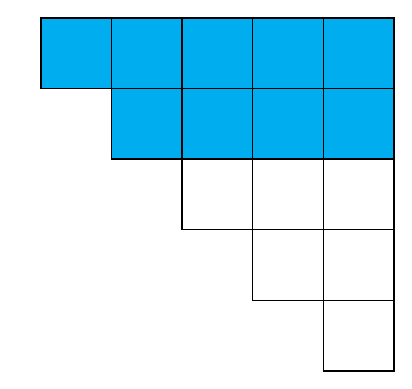}&
   \includegraphics[scale=.5]{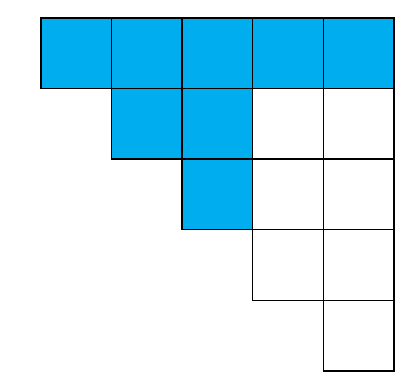}
 \\
 $\{4,3,2,1\}$ &  $\{4,3,2\}$ &  $\{4,3,1\}$ & $\{4,2,1\}$ & $\{3,2,1\}$ &  $\{4,3,2\}$\\
 \includegraphics[scale=.5]{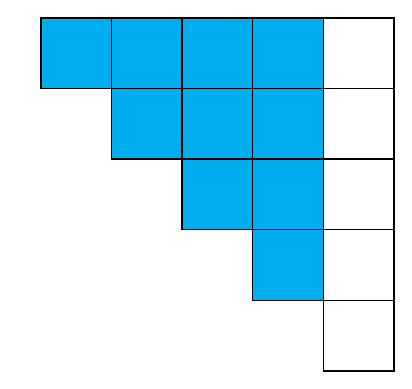}&      \includegraphics[scale=.5]{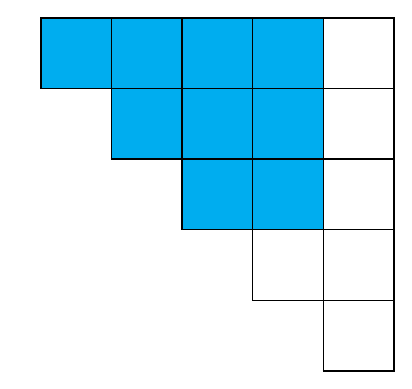}& \includegraphics[scale=.5]{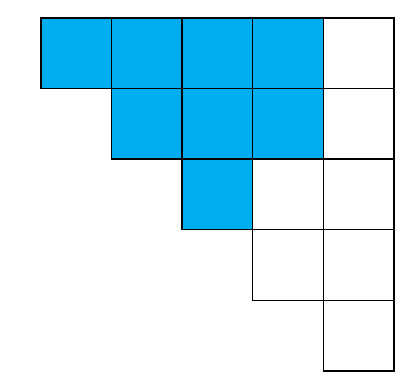}& \includegraphics[scale=.5]{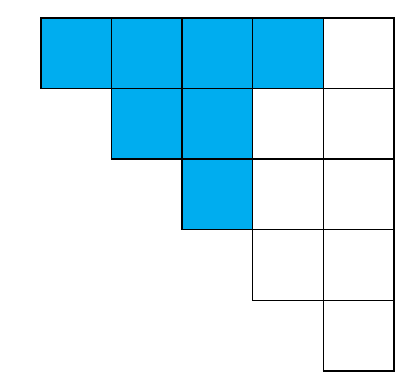}&  \includegraphics[scale=.5]{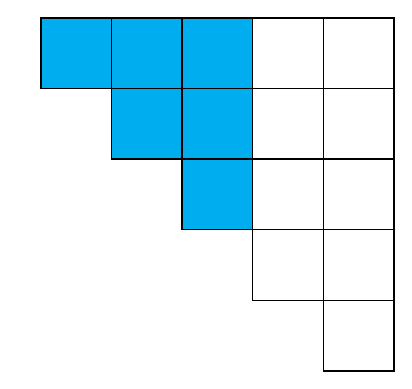}& \includegraphics[scale=.5]{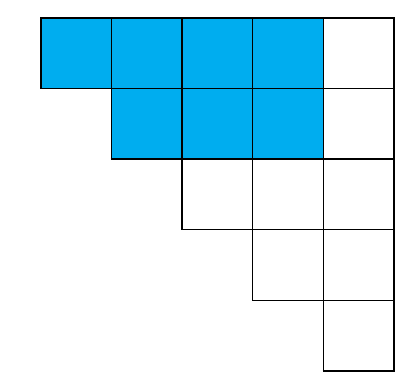}
   \end{tabular}}
   \end{minipage}
   \subcaption{{\bf The 12 chambers of}  $\slngeoSm{5}$. See \cite[Figure 9]{box}. 
 \label{fig:YT_SU5}    }
  \end{centering}
 }
 
 \vspace{1cm}
 \caption{ {\bf Chambers of $\slngeoSm{3}$, $\slngeoSm{4}$, $\slngeoSm{5}$}. 
Each chamber is labeled by a subset of $[n]=\{1,2,\ldots, n\}$. 
Elements of the labeling subset counts the number of boxes with $+$ signs in the rows of the diagram, and sign flows (see \pref{prop:signFlows}) ensure that the numbers of $+$ signs are pairwise distinct.
There is a $\mathbb{Z}_2$ symmetry connecting a chamber defined by a subset $S$ to a chamber defined by the complementary subset $S^C$.
The adjacency graph of the chambers is shown in  \pref{fig:SU5_Phases}. }
\label{fig:YT345}
\end{figure}

\begin{figure}
\includegraphics[scale=1.4]{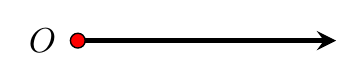}
\caption{    $\slngeoSm{2}$ consists of a half-line. See \cite[Figure 8]{ESY1}. \label{fig:SU2_Phases}}

\end{figure}

\begin{figure}

\begin{floatrow}
\ffigbox{
  \includegraphics[scale=1.5]{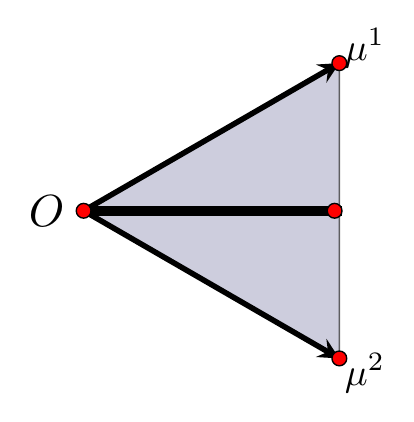}
}
{\caption{  $\slngeoSm{3}$ consists of two chambers separated by a half-line.  
See \cite[Figure 1]{ESY1}.
 \label{fig:SU3_Phases}}}

\ffigbox{
  \includegraphics[scale=.89]{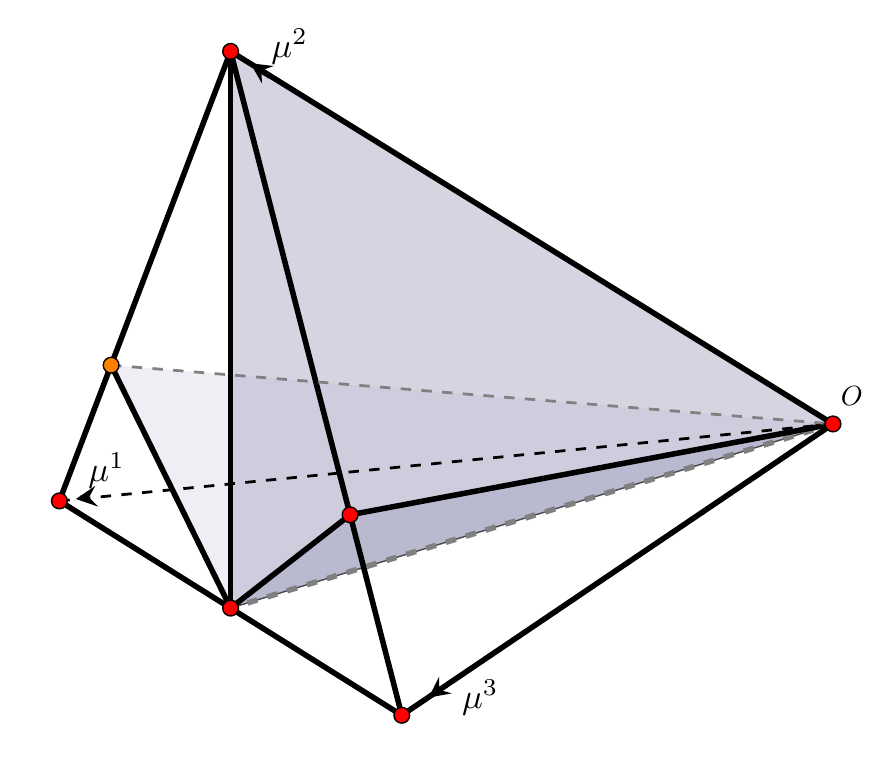}
}
{\caption{  $\slngeoSm{4}$ consists of four chambers.  The adjacency graph of these chambers is a linear chain. See \cite[Figure 2]{ESY1}. \label{fig:SU4_Phases} }   }
\end{floatrow}
\end{figure}

\begin{figure}
  \includegraphics[scale=1]{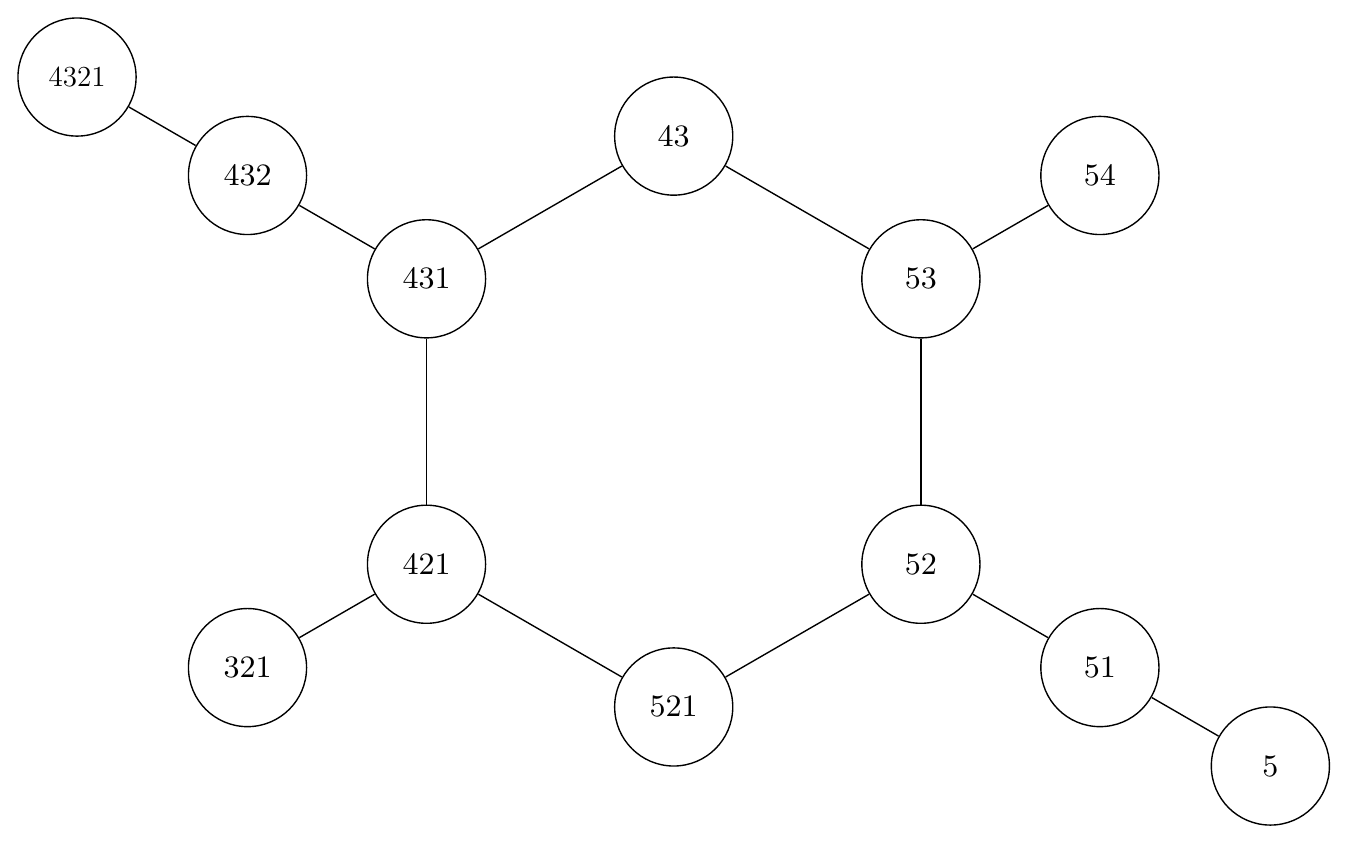}
\caption{Adjacency graph of the chambers of  $\slngeoSm{5}$. See \cite[Figure 2]{hls} and  \cite[Figure 17]{EY} for the central hexagon. 
Each node represents a chamber characterized by a subset  $S=\{a_1>a_2>\cdots >a_s\}$ of $\{1,2,3,4,5\}$. 
All the chambers are simplicial. The adjacency graph of the chambers of  $\slngeoSm{5}$ is realized by explicit resolutions of singularities of a $SU(5)$ Weierstrass model in \cite{ESY2}. 
 \label{fig:SU5_Phases}
}
\end{figure}

\begin{figure}
\scalebox{.8}{
   \begin{tabular}{c  c  c  c  c  c  }
$\{6\}$ & $\{6,1\}$ &  $\{6,2\}$ &  $\{6,3\}$ & {$\star \{6,4\}$} &  $\star  \{6,5\}$   
 \\
 \includegraphics[scale=.4]{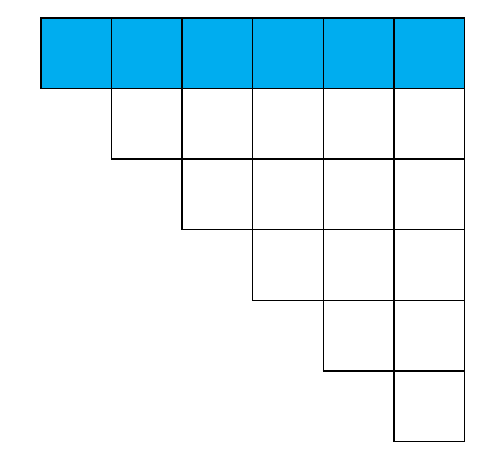}& 
    \includegraphics[scale=.4]{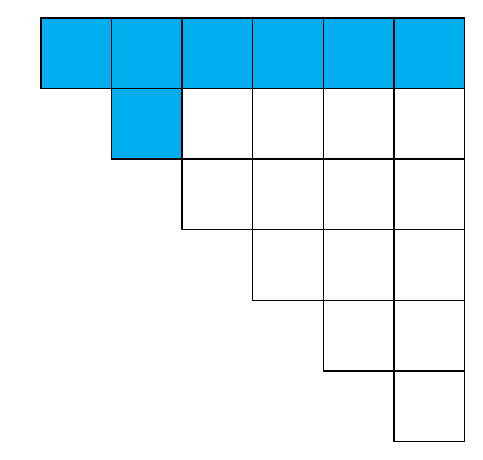}&
        \includegraphics[scale=.4]{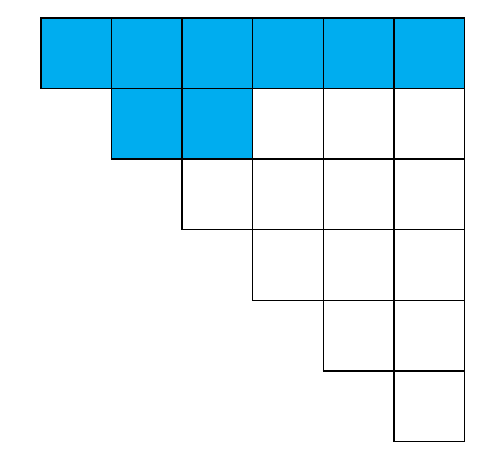}&
 \includegraphics[scale=.4]{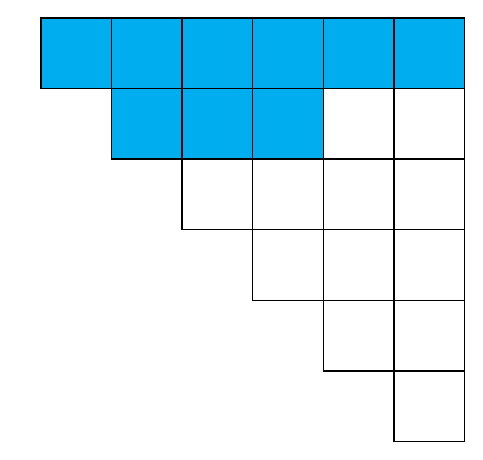}&
  \includegraphics[scale=.4]{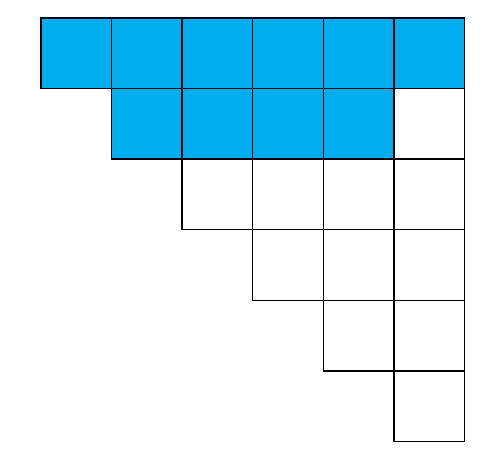}&
   \includegraphics[scale=.4]{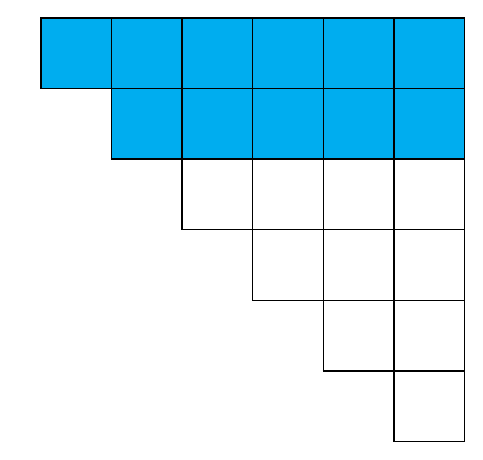}
 \\

 $\{6,2,1\}$ &  $\{6,3,1\}$ &  $\{6,4,1\}$ & $\{6,3,2\}$ & $\{6,3,2,1\}$ &  $\{6,5,1\}$\\
 \includegraphics[scale=.4]{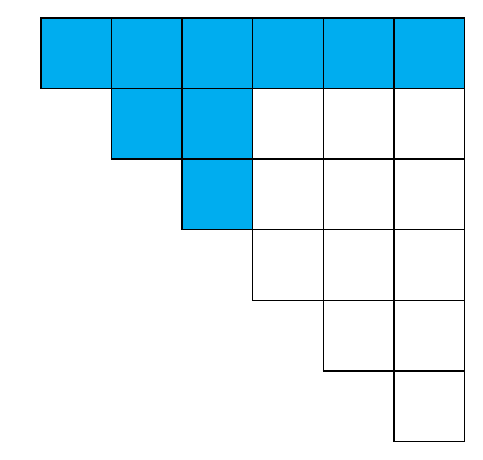}&      \includegraphics[scale=.4]{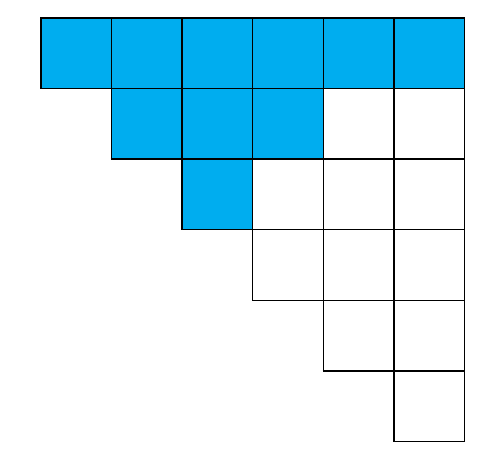}& \includegraphics[scale=.4]{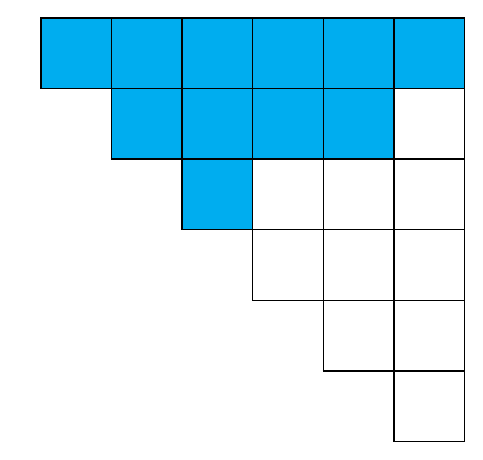}& \includegraphics[scale=.4]{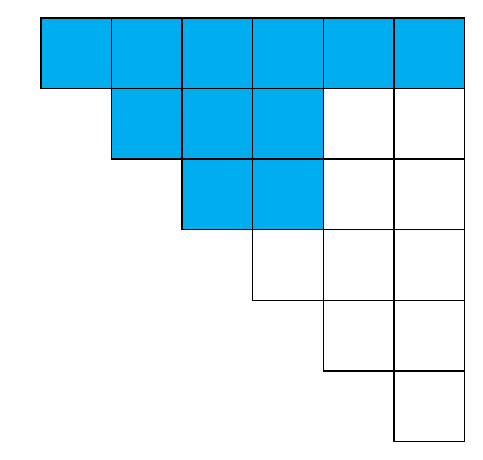}&  \includegraphics[scale=.4]{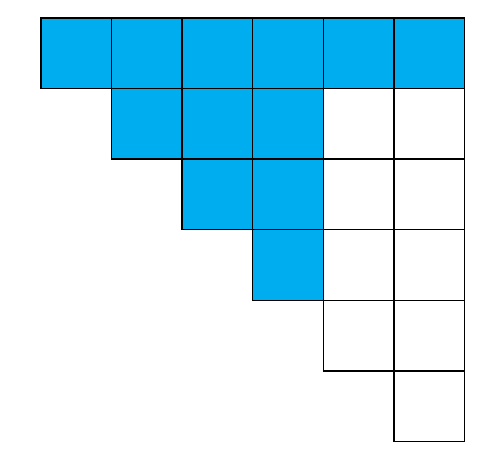}& \includegraphics[scale=.4]{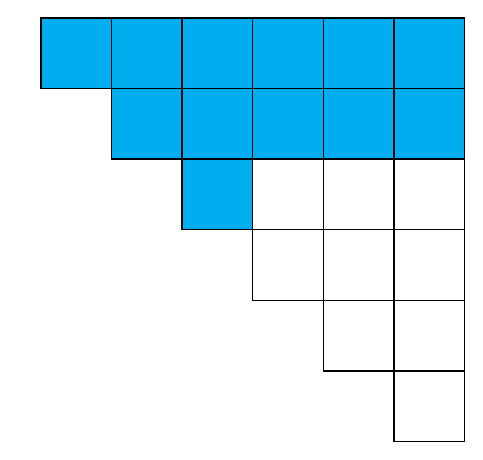}
\\  
$\{5,4,3\}$& $\{5,4,2\}$& $\{5,3,2\}$& $\{5,4,1\}$& $\{5,4\}$ &  $\{4,3,2\}$ \\
  \includegraphics[scale=.4]{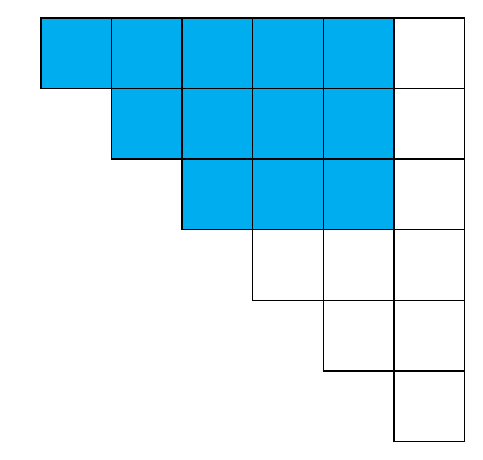}&
    \includegraphics[scale=.4]{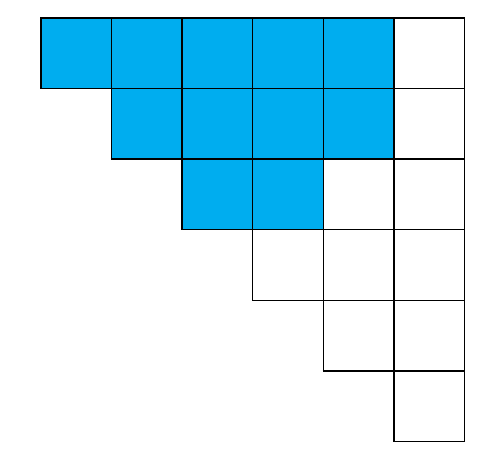}&
        \includegraphics[scale=.4]{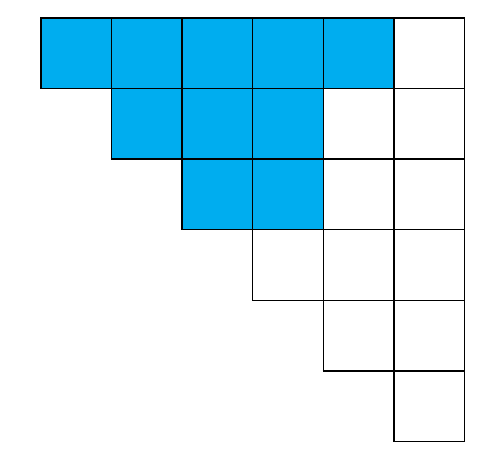} &
 \includegraphics[scale=.4]{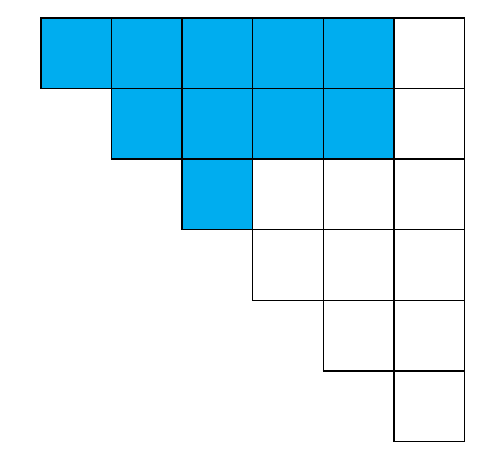}&
  \includegraphics[scale=.4]{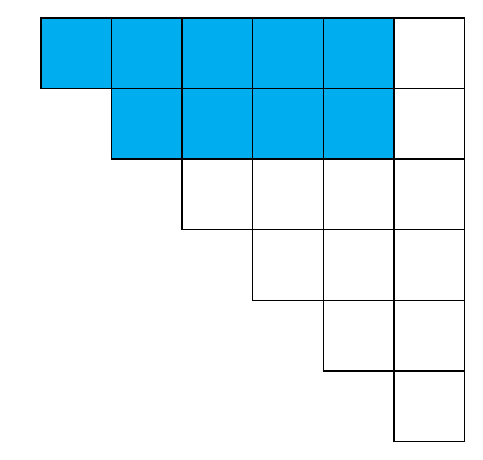} &
   \includegraphics[scale=.4]{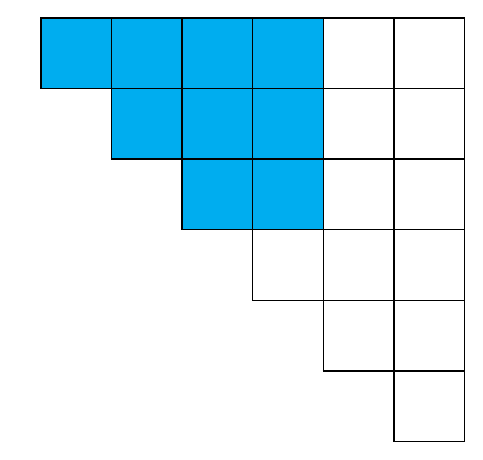}\\

$\{5,4,3,2,1\}$  &  $\{5,4,3,2\}$  & $\{5,4,3,1\}$& $\{5,4,2,1\}$& {$\star \{5,3,2,1\}$}& {$\star \{4,3,2,1\}$}\\
 \includegraphics[scale=.4]{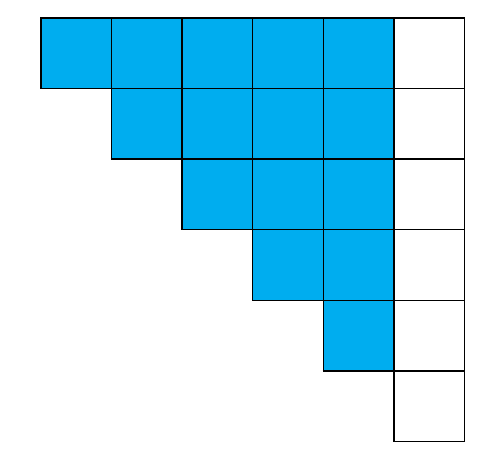}&
    \includegraphics[scale=.4]{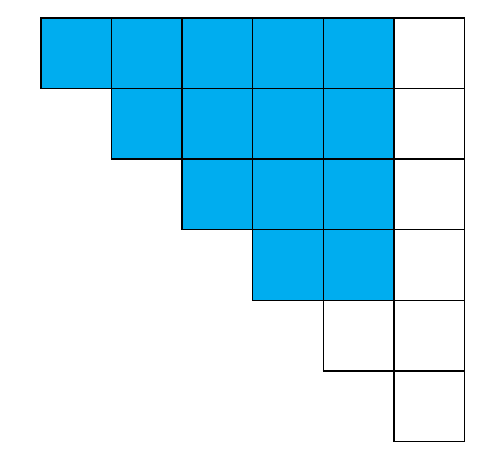}&
        \includegraphics[scale=.4]{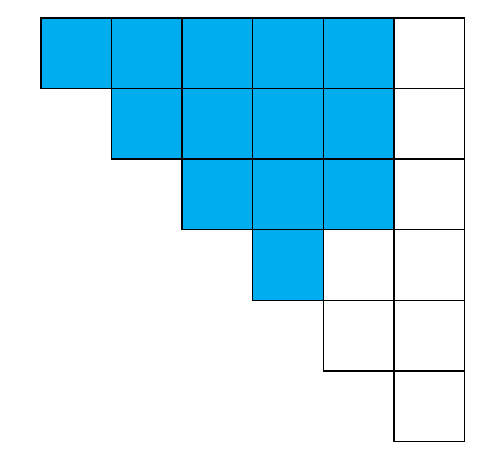}&
 \includegraphics[scale=.4]{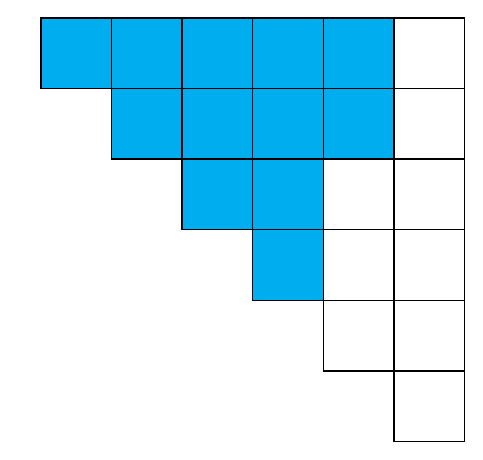}&
  \includegraphics[scale=.4]{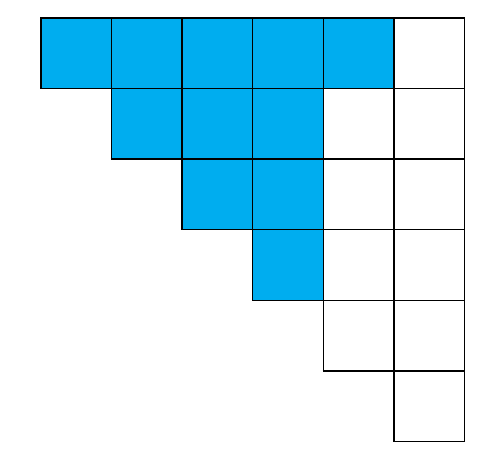}&
   \includegraphics[scale=.4]{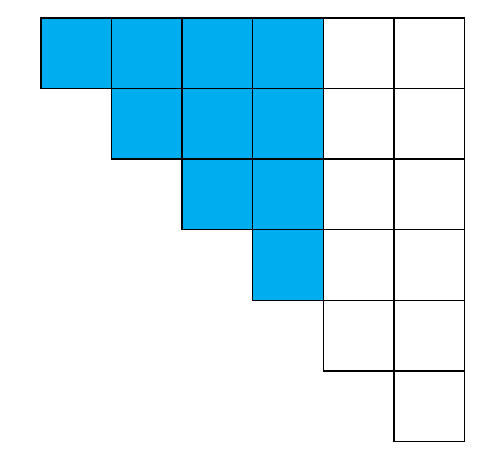}
\end{tabular}

  } 
\caption{{\bf The 24 chambers of}  $\slngeoSm{6}$.
 The four 
non-simplicial chambers of $\slngeoSm{6}$  are marked  with a  star  (see \pref{fig:SimplicialSU6}).
 }\label{fig:YT_SU6} 
\end{figure}

\begin{figure}
  \includegraphics[scale=3.2]{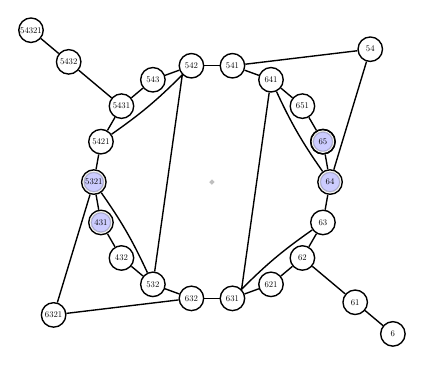}
\caption{{\bf Adjacency graph of the 24 chambers of }$\slngeoSm{6}$.
The four colored nodes are the  non-simplicial chambers. 
\label{fig:SU6_Phases}
}
\end{figure}

\begin{figure}
  \includegraphics[scale=.89]{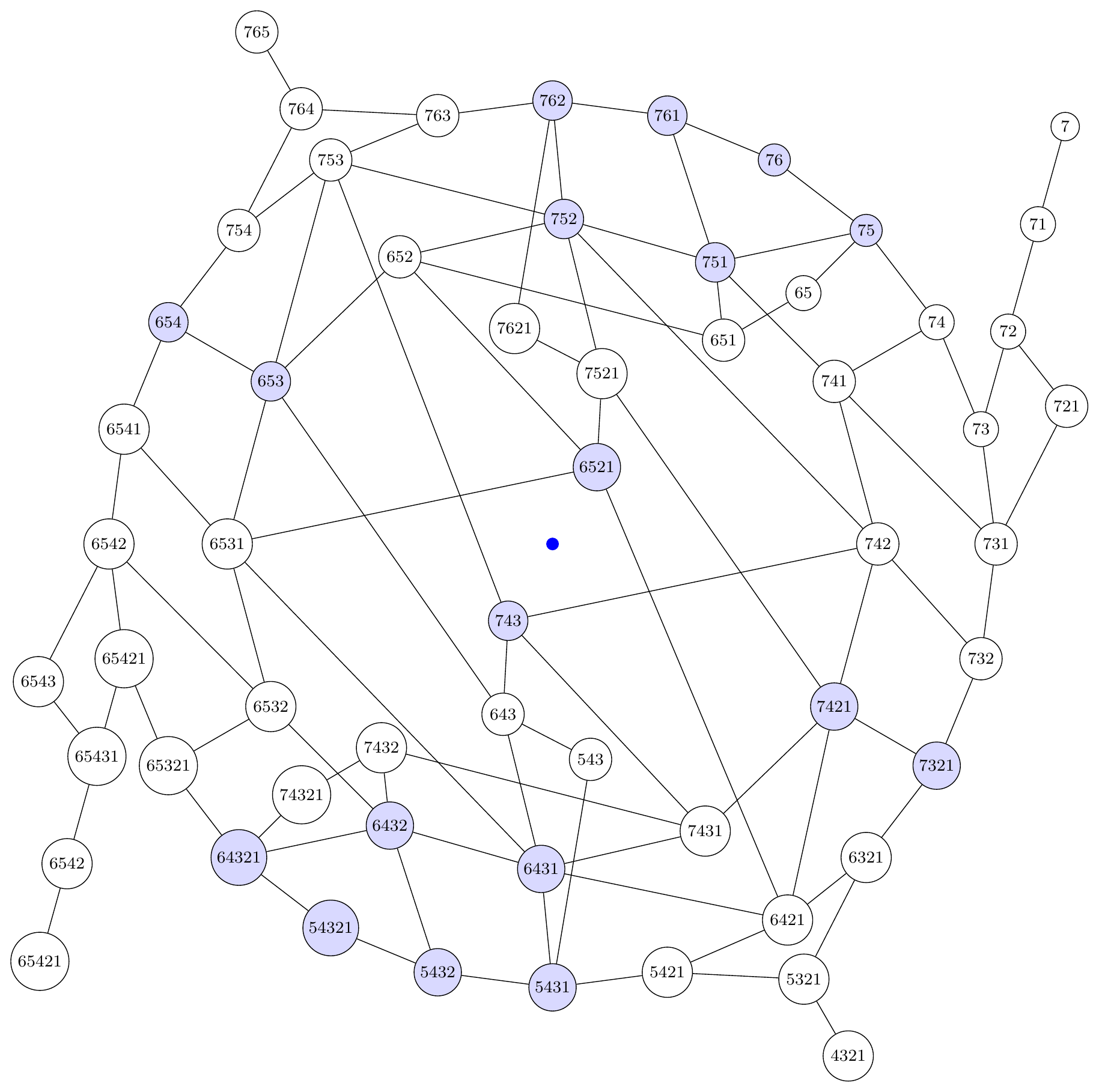}
\caption{{\bf Adjacency graph for the 58  chambers of }$\slngeoSm{7}$.
The 18  colored nodes are the  non-simplicial chambers. 
\label{fig:SU7_Phases}
}
\end{figure}

\begin{figure}
\begin{center}
\begin{minipage}{ \linewidth}
\centering{
\begin{tabular}{cccc}
$\{6,5\}$ &$\{6,4\}$ & $\{4,3,2,1\}$ & $\{5,3,2,1\}$\\
\ \   \includegraphics[scale=.6]{YT_6_65}   &
 \includegraphics[scale=.6]{YT_6_64}&
  \includegraphics[scale=.6]{YT_6_4321}&
  \includegraphics[scale=.6]{YT_6_5321}\\
 \includegraphics[scale=.7]{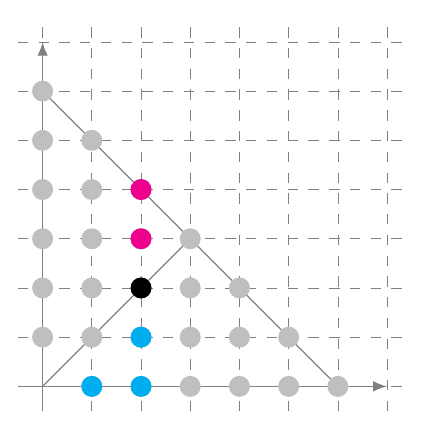}&
 \includegraphics[scale=.7]{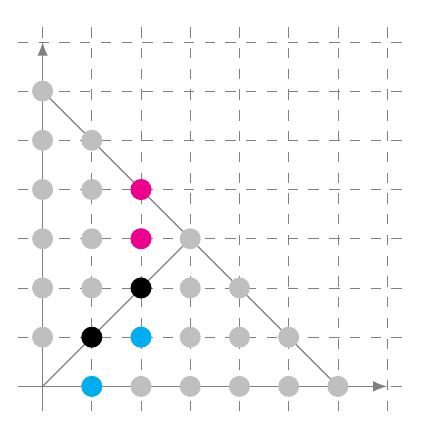}&
  \includegraphics[scale=.7]{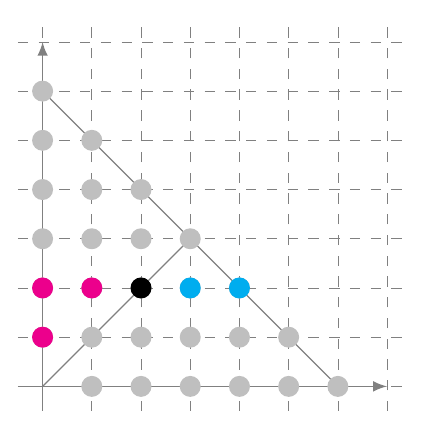}&
    \includegraphics[scale=.7]{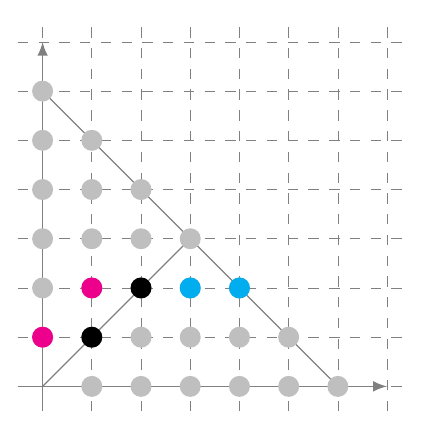}
    \end{tabular}
    }
\end{minipage}
\subcaption{ The four  non-simplicial chambers of $\slngeoSm{6}$. 
For each chamber, we denote by $\ell_+$, $\ell_-$, $\ell_0$, the number of extreme rays above, below, and on the diagonal $y=x$ respectively.  
 These are chambers of $\slngeoSm{6}$ since they cross the diagonal line. However, these chambers  are  {\bf not simplicial} since $\ell_\pm >1$. 
 The number of extreme rays of the chamber is $\ell_+\cdot \ell_-+\ell_0$. 
\label{fig:su6_non_simplicial}   }

\begin{minipage}{ \linewidth}
\centering{
\begin{tabular}{cccc}
$\{6\}$ &$\{6,3\}$ & $\{5,4,3,2,1\}$ & $\{5,4,2,1\}$\\
 \includegraphics[scale=.6]{YT_6_6}&
  \includegraphics[scale=.6]{YT_6_63}&
  \includegraphics[scale=.6]{YT_6_54321}&
  \includegraphics[scale=.6]{YT_6_5421}\\
 \includegraphics[scale=.7]{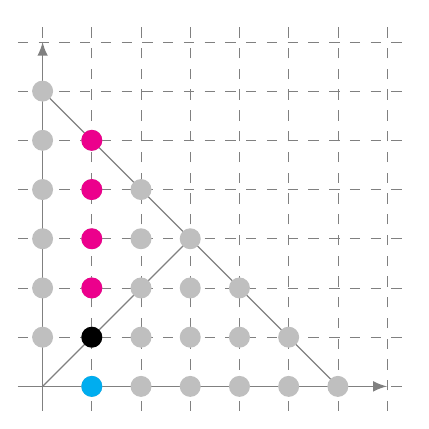}&
  \includegraphics[scale=.7]{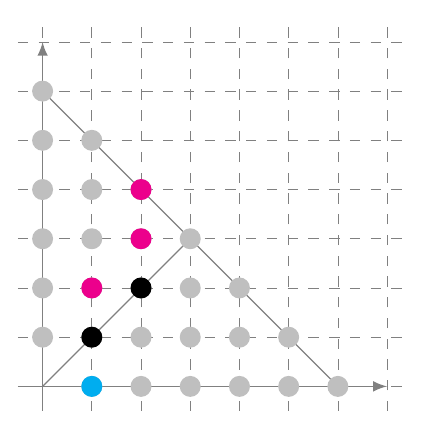}&
  \includegraphics[scale=.7]{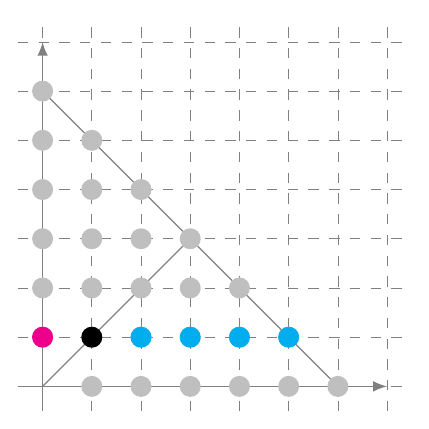}&
   \includegraphics[scale=.7]{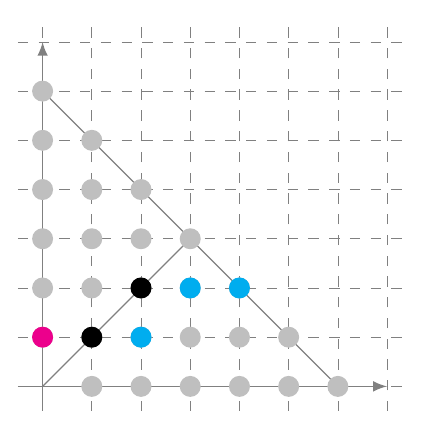}
  \end{tabular}}
\end{minipage}\label{fig:su6_simplicials}  
\subcaption{\footnotesize  
Simplicial chambers are such that  $min\{\ell_+, \ell_-\}=1$.  
 }

\end{center}
\caption{ {\bf Example of  simplicial and non-simplical chambers} of $\slngeoSm{6}$. 
A chamber of $\glngeoSm{6}$ gives a chamber of $\slngeoSm{6}$ if and only if ($\ell_+\cdot \ell_-\neq 0$).
The chamber is simplicial if and only if $min\{\ell_+, \ell_-\}=1$. The number of extreme rays of the chamber is $\ell_+\cdot \ell_-+\ell_0$. 
\label{fig:SimplicialSU6}
}  
\end{figure}

\clearpage
\subsection{Outline of the paper}

In \pref{sec:prelim}, we review some of the results of \cite{gl} and fix notation.
In \pref{sec:formalStatements}, we formally state our main results.
In \pref{sec:chambers}, we prove \pref{thm:chamberCount}.
In \pref{sec:relateSlnFacesAndFlatsToGln}, we prove bijections to connect flats and faces of $\slngeo$ to $\glngeo$, which we use in \pref{app:proofsOfSlnFacesAndFlats} to complete the proof of \pref{thm:slnFaces} and \pref{thm:slnFlats} modulo two technical combinatorial results (Propositions~\ref{prop:A0} and~\ref{prop:C0}), which are proved in the appendices.

\subsection{Acknowledgments}
M.E. is grateful to David Morrison and  Sakura Sch\"afer-Nameki for explaining different aspects of \cite{box}. M.E. would also like to thank  Shu-Heng Shao, and Shing-Tung Yau for helpful discussions. 
M.E. and A.N. thank William Massey and all the organizers of  the 20th Conference for African American Researchers in Mathematical Sciences (CAARMS 20) where this  work started. 
M.E. is supported in part by the National Science Foundation (NSF) grant DMS-1406925  ``Elliptic Fibrations and String
Theory". R.J. was supported by the Harvard College Research Program.

\section{Preliminaries}
\label{sec:prelim}

Basic notions of hyperplane arrangements are reviewed in section \pref{sec:IntroHyperArrang} and details regarding the Lie algebras $\mathfrak{gl}_n$ and $\mathfrak{sl}_n$ are reviewed in \pref{sec:typeArootSys}.
In \pref{sec:glnStuff}, we recall the combinatorial descriptions of faces, flats, and chambers in terms of extreme rays from \cite{gl}. In \pref{sec:notation}, we fix further notation for this paper.

\subsection{Introduction to hyperplane arrangements: chambers, flats, and faces}
\label{sec:IntroHyperArrang}
In this subsection, we give several basic notions related to hyperplane arrangements. 
For more details on hyperplane arrangements, we refer to  \cite{MR1217488,MR2383131,MR0357135}.

\subsubsection{Essential and central arrangements}
A hyperplane of  $\mathbb{R}^d$ is a $(d-1)$-dimensional affine subspace defined as  the vanishing locus of a polynomial equation of  degree one. 
A real hyperplane arrangement ${\mathcal A}$ is a finite set of linear hyperplanes $\lambda_i^\bot$ cut out by polynomials of degree one  $\lambda_i$ in the real affine space $\mathbb{R}^d$. 
A hyperplane arrangement is called {\em essential} if the normal vectors of its hyperplanes span the ambient space $\mathbb{R}^d$. 
A hyperplane arrangement is said to be {\em central} if the intersection of all its hyperplanes is non-empty. 
For a central and essential arrangement, the intersection of all the hyperplanes is a single point that we take to be the  origin of $\mathbb{R}^d$.

\subsubsection{Flats, chambers, and faces}
A {\em flat} of an arrangement $\mathcal A$ is an intersection  $\bigcap I$ of all the  elements of a subset $I$ of $\mathcal A$.
The ambient space itself is a flat  corresponding to the intersection of the empty family of hyperplanes of $\mathcal A$. 
The set of all flats form a semi-lattice $\mathfrak{L}(\mathcal{A})$ ordered by  reverse inclusion. 
The smallest element is the full ambient space. The  hyperplanes are the atoms of $\mathfrak{L}(\mathcal{A})$. 

The connected components of $\mathbb{R}^d\setminus \bigcup_i \lambda_i^\bot$ are the {\em chambers} of the arrangement.  
Each chamber is an open convex  polyhedron. For central arrangements, the faces are open convex polyhedral cones.
A (closed) {\em face} of the hyperplane arrangement $\mathcal A$ is by definition the closure of a chamber or its intersection with a flat of the arrangement $\mathcal A$. 
We denote by $F^0$ the relative interior of a face  $F$ in a hyperplane arrangement.  
The set of all faces of a hyperplane arrangement $\mathcal A$ form  a poset with $F_1 \leq F_2$ if the face $F_1$ lies in the closure of the face $F_2$. 

The number of chambers and the number of $k$-faces of a hyperplane arrangement can be computed using an elegant theorem of Zaslasky using Tutte-Grotendieck 
invariants such as  the M\"obius function and characteristic polynomial on the intersection semi-lattice  \cite{MR0357135}. 

The dimension of a face or a flat is defined as the dimension of  its linear span. 
A $k$-face  (resp.\ $k$-flat) is a face  (resp.\ a flat) of dimension $k$. 
A $0$-face is a {\em vertex}, a $1$-face is an {\em edge}.
A half-line that is a face is called an {\em extreme ray}.
For a central arrangement, the only vertex is the origin and the only edges are lines and extreme rays.
 
 In an essential central hyperplane arrangement in $\mathbb{R}^d$, a chamber is said to be {\em simplicial} if it is the positive span of $d$ independent vectors. 
 A hyperplane arrangement is said to be {\em simplicial} if all its chambers are simplicial.

\subsubsection{Sign vectors}
For a hyperplane arrangement, we choose  linear forms $\lambda$ such that each  hyperplane of $\mathcal{A}$   is the kernel $\lambda^\bot$ of a form $\lambda$. 
Geometrically, this is a choice of a normal direction for the hyperplane $\lambda^\bot$. 
We denote by $\lambda^+$ (resp.\ $\lambda^-$) the half-space on which $\lambda$ is 
 non-negative (resp.\ non-positive). 
For any point $p$ in the ambient space $\mathfrak{h}$, we attach a {\em sign vector } of dimension  $|\mathcal{A}|$  whose components are parametrized by the elements of $\mathcal{A}$.  The $\lambda$ component of the sign vector is the sign $sign(\lambda(p))\in \{-1, 0, +1\}$  of  $\lambda\in \mathcal{A}$ evaluated at $p$. We often write $\pm$ for $\pm 1$. 
The sign vector is an element of $\{ -, 0, +\}^{|\mathcal{A}|}$ uniquely determined by  the relative position of $p$ with respect to all the hyperplanes of the arrangement.  
The set of points having the same sign vector determines uniquely a face of the  hyperplane arrangement.

\subsection{Weyl chambers, roots and weights for \texorpdfstring{$\mathfrak{gl}_n$}{gl} and \texorpdfstring{$\mathfrak{sl}_n$}{sl}}
\label{sec:typeArootSys}

Let $\mathfrak{h} \subseteq \mathfrak{gl}_n$ denote the vector space of real, diagonal $n \times n$ matrices, which is the split, real form of a Cartan subalgebra of $\mathfrak{gl}_n$.  Let $x_i \in \mathfrak{h}^*$ denote the function that returns the $(i,i)$th entry of an element of $\mathfrak{gl}_n$. 
The {\em signature } of a vector  $(x_1, \ldots, x_n)$ is denoted by $$\sigma(x_1, \ldots, x_n)=x_1+\ldots +x_n.$$
The hyperplane $\mathfrak{h}_s$ of elements of $\mathfrak{h}$ of signature zero is the split, real form of a Cartan subalgebra of $\mathfrak{sl}_n$. 
The space $\mathfrak{h}_s \subset \mathfrak{gl}_n$ is the set of  $n\times n$ real diagonal matrices of trace zero.
The inclusion $\mathfrak{h}_s \to \mathfrak{h}$ gives rise to a surjection $\mathfrak{h}^* \to \mathfrak{h}_s^*$ by which we can obtain functionals on $\mathfrak{h}_s$ from functionals on $\mathfrak{h}$.  We abuse notation and identify functionals on $\mathfrak{h}$ with their restrictions to $\mathfrak{h}_s$.

A set of positive roots of $\mathfrak{gl}_n$ and of $\mathfrak{sl}_n$ is the set of differences $x_i-x_j$  for $1\leq i<j\leq n$. The open fundamental  Weyl chamber $W^0$ of $\mathfrak{gl}_n$ for this set of positive roots is the locus of points of $\mathfrak{h}$ whose coordinates  $(x_1, \ldots, x_n)$ form  a non-increasing sequence
$$W^0=\{\ (x_1,\ldots, x_n)\in \mathfrak{h}  \mid  x_1>x_2>\cdots> x_n\}.$$
Its closure is the dual fundamental Weyl chamber denoted by $W$. 
The open fundamental Weyl chamber of $\mathfrak{sl}_n$ and its closure are respectively denoted by $W^0_s$ and  $W_s$. 
They are the intersections of $W^0$ and $W_s$ with $\mathfrak{h}_s$.

The weights of $V$ are the $x_i$ ($i\in [n]$), while those of $\bigwedge^2$ are 
$x_i+x_j$  $(1\leq i < j\leq n)$.\footnote{Recall that $V$ restricts to the the first fundamental representation and $\bigwedge^2$ to the second fundamental representation of $\mathfrak{sl}_n$.} We denote by $\weights \subset \h^*$ the set of weights of $V \oplus \bigwedge^2$.  We can naturally regard $\weights$ as a subset of $\h_s^*$ for $n > 2$.

\subsection{Structure of chambers, faces, flats, and extreme rays of \glngeoForTitle}
\label{sec:glnStuff}
For more details regarding the results recalled in this section, see \cite{gl}.

\begin{figure}
\begin{multicols}{2}
\includegraphics[scale=1.2]{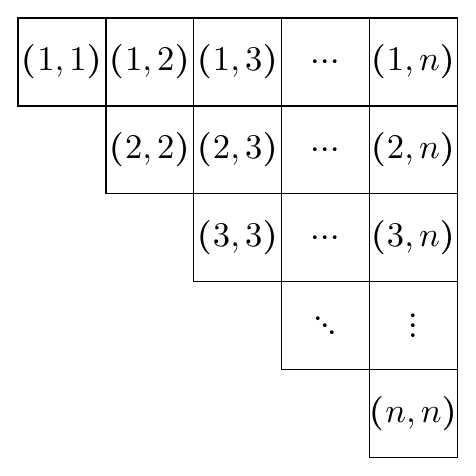}

\includegraphics[scale=1]{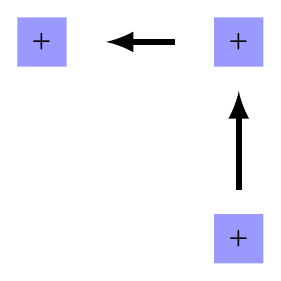}
 
\includegraphics[scale=1]{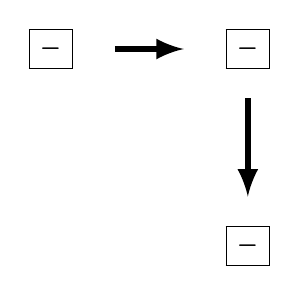}
\end{multicols}

\caption{{\bf Sign vectors of $\glngeo$.} 
The sign vectors of $\glngeo$ can be organized in a right-justified tableau  \protect{\cite[Section 2.4]{box}}. 
The box labeled by $(i,j)$ corresponds to the sign of the  weight $x_i+x_j$  (resp.\ $x_i$) when $i\neq j$  (resp.\ $i=j$).  
A sequence of signs corresponds to a chamber of $\glngeo$ if and only if it satisfies the sign flow condition of \pref{prop:signFlows}: positive (resp.\ negative) signs flow North and West (resp.\ South and East). 
   }
\label{fig:signTableau}
\end{figure}

A face is  determined by an assignment of signs to $x_i+x_j$ for all values of $(i,j)$ such that  $i,j\in \{1,2,\cdots, n\}$. 
These weights can be organized in a  right-justified Young tableau as in Figure \ref{fig:signTableau}.
It is important to realize that not all sign patterns are allowed.
Since we are in the Weyl chamber, we have
\[
i<j\iff x_i-x_j>0.
\]
Using the identities
\begin{align*}
x_i+x_{j+1} &=x_i+x_j-(x_{j}-x_{j+1})\\
x_{i+1}+x_{j} &=x_i+x_j- (x_{i}-x_{i+1}),\\
\end{align*}
one can prove the following simple sign rules 
 called {\em sign flows} in the physics literature  \cite{box}:
\begin{proposition}[\protect{\cite[Section 2.4]{box}}]
\label{prop:signFlows}
The signs in a tableau corresponding to a chamber satisfy the following conditions:
\begin{itemize}
\item All the boxes above or on the left of a box with positive entries are also positive. 
\item All the boxes below or on the right of a box with negative  entries are also  negative. 
\end{itemize}
\end{proposition}

A tableau satisfying these two rules corresponds to a unique chamber  
of $I(\mathfrak{gl}_n, V\oplus \Exterior^2)$.

\begin{remark}
A different notation for chambers will be more convenient for dealing with extreme rays.
The sign rules imply that the positive entries of each row are next to each other and start on the left border of the table.  It is therefore efficient to denote a given chamber by the numbers of positive entries on each row. 
The sign rules are then automatically satisfied if these numbers form a decreasing sequence 
$\{a_1>a_2> \cdots> a_k\}$, 
where we denote by $a_i$ the number of  positive entries on the $i$th row of the tableau.  
We don't count the rows that do not have any positive entries. 
An entry of the tableau located on the $i$th row and the $j$th column  (with $i\leq j$) is positive if and only if $i\leq k$ and $j\leq a_i$. It is negative otherwise. This justifies the following definition. 
\end{remark}

\begin{definition}[\protect{\pref{gln-def:subsetToChamber}}]
\label{def:subsetToChamber}
For $S = \{a_1 > \cdots > a_k\} \subseteq [n]$, define a  chamber  $C(S)$ of $I(\mathfrak{gl}_n, V\oplus \bigwedge{}^2)$ as the subset of $\mathfrak{h}$ on which  for  $1\leq i\leq j\leq n$.
\begin{align*}
x_i + x_j \ge 0 &\qquad \text{if } i \le k \text{ and } j \le a_i\\
x_i + x_j \le 0 &\qquad \text{otherwise}.
\end{align*}
We can equivalently write a subset $S \subseteq [n]$ as a \emph{characteristic vector} $s \in \{-1,1\}^n$, where $s =  (s_1,\ldots,s_n)$ and $s_i = 2\chi_S(n) - 1$.
  A characteristic vector $s = (s_1,\ldots,s_n)$ defines a subset $S\subset [n]$ such that  $i\in S$  (resp. $i\notin S$) if  and only if  $s_i=1$  (resp. $s_i=-1$). 
\end{definition}

\begin{table}
\begin{center}
\begin{tabular}{l | c|   c  |     l}
Tableau & Subset of $[n]$ & Characteristic vector & Interior points of the  chamber  \\
\hline
\rule{0pt}{6ex} 
\begin{ytableau}- & - & - \\ \none & - & -  \\ \none & \none & -    \end{ytableau} &   $\emptyset$ & $ (-1,-1,-1)$  &   $(-a_1, -a_1-a_2,-a_1-a_2-a_3)$\\
\rule{0pt}{6ex} 
\begin{ytableau} *(cyan)  +& - & - \\ \none & - & -  \\ \none & \none & -    \end{ytableau} &   $\{1\}$ & $(\ \ 1,-1,-1)$ & $(a_1,-a_1-a_2 ,-a_1-a_2-a_3)$\\
\rule{0pt}{6ex} 
\begin{ytableau} *(cyan)+ &*(cyan) + & - \\ \none & - & -  \\ \none & \none & -    \end{ytableau}  & $\{2\}$ &  $(-1,\ \  1,-1)$ &  $(a_1+a_2,-a_1 ,-a_1-a_2-a_3)$\\
\rule{0pt}{6ex} 
\begin{ytableau} *(cyan)+ & *(cyan)+  & *(cyan)+ \\ \none & - & -  \\ \none & \none & -    \end{ytableau}   &$\{3\}$ &  $(-1,-1,\ \ 1)$  &  $(a_1+a_2+a_3,-a_1-a_2 ,-a_1-a_2-a_3)$\\
\rule{0pt}{6ex} 
\begin{ytableau} *(cyan)+ & *(cyan)+ & *(cyan)+ \\ \none & *(cyan)+ & -  \\ \none & \none & -    \end{ytableau}   &$\{3,1\}$ & $(\ \  1,-1,\ \ 1)$ &   $(a_1+a_2+a_3,a_1 ,-a_1-a_2)$ \\
\rule{0pt}{6ex} 
\begin{ytableau} *(cyan)+ & *(cyan)+ & - \\ \none & *(cyan)+& -  \\ \none & \none & -    \end{ytableau}   &$\{2,1\}$ &  $(\ \ 1,\ \ 1,-1)$ &  $(a_1+a_2,a_1 ,-a_1-a_2-a_3)$\\
\rule{0pt}{6ex} 
\begin{ytableau} *(cyan)+ & *(cyan)+ & *(cyan)+ \\ \none & *(cyan)+ & *(cyan)+  \\ \none & \none & -    \end{ytableau}   &$\{3,2\}$  &  $(-1,\ \  1,\ \ 1)$ &  $(a_1+a_2+a_3,a_1+a_2 ,-a_1)$ \\
\rule{0pt}{6ex} 
\begin{ytableau} *(cyan)+& *(cyan)+ & *(cyan)+ \\ \none & *(cyan)+ & *(cyan)+  \\ \none & \none & *(cyan)+    \end{ytableau}  & $\{1,2,3\}$ & $(\ \ 1,\  \  1,\ \  1)$ &  $(a_1+a_2+a_3,a_1+a_2 ,a_1)$
 \end{tabular}
\vspace{.5cm}
\end{center}
\caption{{\bf Chambers of $\glngeoSm{3}$.}  See \cite{gl}. We show several different notations for the chambers.  In describing the sets of interior points, the variables $a_1,a_2,a_3$ denote positive real numbers.
In the rest of the paper, we use a color code in which positive entries are in color while negative entries are left empty. 
} 
\label{tab:example.gl3}
\end{table}

See  \pref{tab:example.gl3} for explicit examples of this notation in the case of $\glngeoSm{3}$. 
The notation of \pref{def:subsetToChamber} allows us to state the following classification of chambers and extreme rays.

\begin{theorem}[\protect{\pref{gln-thm:chamberStructure}}]
\label{thm:chamberStructure}
The chambers and extreme rays of $\glngeo$ satisfy the following properties.
\begin{enumerate}[(a)]
\item \label{assert:chamberBiject} The map $S \mapsto C(S)$ defines a bijection from $2^{[n]}$ to the set of chambers of $I(\mathfrak{gl}_n, V\oplus \bigwedge{}^2)$.  In particular, $I(\mathfrak{gl}_n, V\oplus \bigwedge{}^2)$ has $2^n$ chambers.
\item \label{assert:chamberRaysPrelim} The extreme rays of $C(S)$ are generated by the vectors $e_1^S,\ldots,e_n^S,$
 where
\[e_\ell^S = (\underbrace{1,\ldots,1}_{\pi_\ell^S},\underbrace{0,\ldots,0}_{n-\ell},\underbrace{-1,\ldots,-1}_{\ell-\pi_\ell^S}),\]
 where $\pi_\ell^S = \left|S \cap [n-\ell+1,\infty)\right|$  counts the elements of $S$ that are greater or equal to $n-\ell+1$, and the vectors $e_1^S,\ldots,e_n^S$ non-negatively span $C(S)$. 
In particular, the geometry $\glngeo$ is simplicial.
\end{enumerate}
\end{theorem}

In light of \pref{thm:chamberStructure}, we can introduce additional structure on the set of extreme rays of $\glngeo$.
We will equip the set of extreme rays with a partial order and use the combinatorics of the resulting partially ordered set (poset) to study faces and flats in $\slngeo$.

The following definitions describe the posets that will be relevant to us.

\begin{definition}[\protect{\pref{gln-def:quarterPlanePoset}}]
\label{def:quarterPlanePoset}
The \emph{discrete quarter plane poset} is the set  $\mathbb{N}^2$ endowed with  the Cartesian order induced by the usual order of the set $\mathbb{N}$ of  non-negative integers: 
$$
 \quad (a,b)\leq (c,d)\iff (a\leq c) \quad and \quad (b\leq d), \quad a,b,c,d\in \mathbb{N}.
$$ 
This is a graded poset with grading function  
$\ell(a,b)=a+b$.  We call $\ell(a,b)$ the \emph{ level}  of $(a,b)$. 
\end{definition}

\begin{example}\label{Example:ER}
Consider  $\glngeoSm{6}$ and the chamber $C(S)$ with $S=\{4,3,2,1\}$  (see \pref{fig:ExampleER}).  
The corresponding tableaux  and extreme rays are presented in \pref{fig:ExampleER}. To find $e_1^S$, we count the number of elements of $S$ greater or equal to $6-1+1=6$. That gives $\pi^S_1=0$. Since $1-\pi^S_1=1$, we have  $e_1^S=\ray{0,1}=(0,0,0,0,0,-1)$. In the same way, $\pi_\ell^S=0$ for $\ell=2$ since there are no elements of $S$ greater or equal to $5$. That means that $e_2^S=\ray{0,2}=(0,0,0,0,-1,-1)$.  
The first non vanishing value of $\pi_\ell^S$ is for $\ell=3$. 
 For $e_3^S$, we have $\pi_3^S=1$ since there is a unique element of $S$ greater or equal to $6-3+1=4$. 
 \end{example}

\begin{figure}
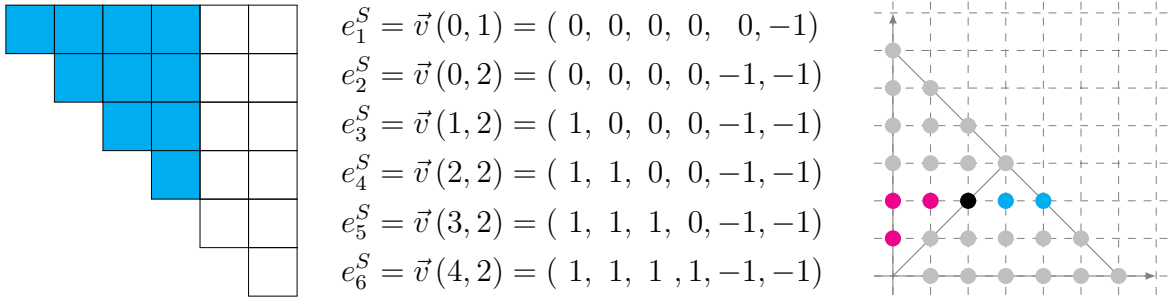

\begin{tabular}{ccc}
\raisebox{-2cm}{
 \includegraphics[scale=.9]{YT_6_4321}}&
$
\begin{aligned}
e_1^S=\ray{0,1} &=(\ 0,\  0,\  0,\  0,\ \   0,-1) \\
e_2^S=\ray{0,2} &=(\  0,\  0,\  0,\  0,-1,-1) \\
e_3^S=\ray{1,2}  &= (\  1,\  0,\  0,\   0,-1,-1)\\
e_4^S=\ray{2,2} &= (\  1,\  1,\ 0,\ 0,-1,-1)\\
e_5^S=\ray{3,2} &=(\  1,\  1,\  1,\ 0,-1,-1) \\
e_6^S=\ray{4,2}  &= (\  1,\  1,\  1\  ,1,-1,-1)
\end{aligned}
$
& 
 \raisebox{-2cm}{
  \includegraphics[scale=1]{ER_6_4321}}
\end{tabular}
\caption{ {\bf The chamber $C(S)$ of  $I(\mathfrak{gl}_6, V\oplus \bigwedge{}^2)$ with $S=\{4,3,2,1\}$ and its extreme rays}.
An extreme ray $\ray{a,b}$ is pictured by the colored dot $(a,b)$ in the poset $\poset_6$. The extreme rays of a chamber of  $I(\mathfrak{gl}_n, V\oplus \bigwedge{}^2)$  for a maximal chain of $\poset_n$.
The positive entries of the tableau are denoted in blue. The elements of $S$ are the number of positive entries in each row disregarding the rows with no positive entries. 
The determination of the extreme rays are discussed in Example \ref{Example:ER} on page \pageref{Example:ER}. 
 \label{fig:ExampleER}
 }
\end{figure}

\begin{definition}[\protect{\pref{gln-def:PosetForExtRays}}]
\label{def:PosetForExtRays}
We denote by $\poset_n$ the subset of the discrete quarter plane
 that consists of points at  level less or equal to $n$. 
$\poset_n$ is a poset with the order induced  by the Cartesian order defined above. 
We denote by $\poseti_n$ the poset $\poset_n\sqcup \{\infty\}$ where $\infty$ is greater than all the elements of $\poset_n$.    
We denote by $\posetz_n$ the poset $\poset_n$ with the origin removed: 
$$
\begin{aligned}
\poset_n &= \{ (a,b)\in \mathbb{N}^2 \quad | \quad 0\leq a + b\leq n \},\quad
 \quad
\poseti_n  = \poset_n \sqcup \{\infty\}, \quad \posetz_n =\poset_n \setminus \{ (0,0) \}.
\end{aligned}
$$
\end{definition}

We are now ready to relate $\posetz_n$ to the set of extreme rays of $\glngeo$.

\begin{definition}[\protect{\pref{gln-def:vFn}}]
\label{def:vFn}
Define a function $\vec{v}:\posetz_n \to W$ by
\[\ray{a,b}=(\underbrace{1,\cdots ,1}_{a}, \underbrace{0,\cdots ,0}_{n-a-b} , \underbrace{-1,\cdots ,-1}_{b}).\]
\end{definition}

The following theorem follows directly from \pref{thm:chamberStructure}.

\begin{theorem}[\protect{\pref{gln-thm:rayStruture}}]
\label{thm:rayStruture}
The extreme rays of $I(\mathfrak{gl}_n, V\oplus \bigwedge{}^2)$ satisfy the following properties.
\begin{enumerate}[(a)]
\item The function $\vec{v}$ defines a bijection from $\posetz_n$ to the set of extreme rays of $I(\mathfrak{gl}_n, V\oplus \bigwedge{}^2)$.
\item The set of extreme rays that lie in a chamber $C(S)$ are the extreme rays $\ray{\pi_\ell^S,\ell-\pi_\ell^S},$ where $\pi_\ell^S = \left|S \cap [n-\ell+1,\infty) \right|$. 
\end{enumerate}
\end{theorem}

We will characterize faces and flats by  the extreme rays they contain. For that reason we define a function which returns the extreme rays lying in a given subset of the Weyl chamber.

\begin{definition}[\protect{\pref{gln-def:raysOperator}}]
\label{def:raysOperator}
Define a function $\erays$ from $2^W$ to the power set of the set of extreme rays of $\glngeo$ as follows.  For $S \subseteq W,$ let $\erays(S)$ be the set of extreme rays of $\glngeo$ that lie in $S$.
\end{definition}

The following theorem relates faces in $\glngeo$ to the combinatorics of $\posetz_n$.

\begin{theorem}[\protect{\pref{gln-thm:Characterization.Faces}}]
\label{thm:Characterization.Faces}
For all $k$, the function $\vec{v}^{-1} \circ \erays$ induces a bijection from the set of $k$-faces of $\glngeo$ to the set of $k$-chains in $\posetz_n$.
\end{theorem}
\begin{remark}
Here, we say that $S \subseteq \posetz_n$ is a $k$-chain if $S$ is a chain and $|S| = k$.
\end{remark}
  
To state an analogue of \pref{thm:Characterization.Faces} for flats, we will need to define structures called \emph{ensembles} that will play the role of chains.
  
\begin{definition}[\protect{\pref{gln-def:Ensemble}}]
\label{def:Ensemble}
An \emph{ensemble} is the restriction to $\posetz_n$ of  a union
\[
\left(\bigcup_{0\leq i\leq k}  [A_i,B_i] \right)\cap \posetz_n
\]
of intervals $[A_i,B_i]$ of $\poseti_n$ satisfying the following four conditions:
\begin{enumerate}[(1)]
\item $A_0=0;$
\item $A_i\leq B_i$ for $0 \le i \le k$;
\item $B_i +(1,1) \leq A_{i+1}$ for $0 \le i \le k-1$; and
\item $\ell(B_k)< n$ or $B_k=\infty$.  
\end{enumerate}
We say that $E$ is a \emph{$k$-ensemble} if $k = \left|\ell(E)\right|$, so that $k$ counts the number of distinct levels of elements of $E$.
\end{definition}

\begin{theorem}[\protect{\pref{gln-thm:Characterization.Flats}}]
\label{thm:Characterization.Flats}
For all $k$, the function $\vec{v}^{-1} \circ \erays$ induces a bijection from the set of $k$-flats of $\glngeo$ to the set of $k$-ensembles of $\posetz_n$.
\end{theorem}

\subsection{Notation}
\label{sec:notation}
See \pref{tab:notations} for an exhaustive list.

We can naturally regard $\weights$ as a multiset of elements of $\h_s^*$ as well.
Given $\lambda \in \h^*$, define
\begin{align*}
\lambda^\perp &= \{v \in \h \mid \lambda(v) = 0\}\\
\lambda^+ &= \{v \in \h \mid \lambda(v) \ge 0\}\\
\lambda^- &= \{v \in \h \mid \lambda(v) \le 0\}.
\end{align*}
Denote by $\mathfrak{F}$ (resp.\ $\mathfrak{F}_s$) the set of faces and by $\mathfrak{L}$ (resp.\ $\mathfrak{L}_s$) the set of flats of $\glngeo$ (resp.\ $\slngeo$).

Given a subset $A \subset \h$, define following subsets of $\weights$:
\begin{align*}
\weights_0(A) &= \{\lambda \in \weights \mid \lambda(A) = 0\}\\
\weights_+(A) &= \{\lambda \in \weights \mid \lambda(A) \ge 0\} \setminus \weights_0(A) \\
\weights_-(A) &= \{\lambda \in \weights \mid \lambda(A) \le 0\} \setminus \weights_0(A).
\end{align*}
For all faces $F$ of $\slngeo$ or $\glngeo$, it follows from the definition of a face that
$$
\label{eq:faceSignPartition}
\weights = \weights_0(F) \sqcup \weights_+(F) \sqcup \weights_-(F)
$$
and that
$$
\label{eq:faceSignGeneralPoint}
\weights_\epsilon(F) = \weights_\epsilon(\{u\})
$$
for all $u \in F$ and $\epsilon \in \{0,+,-\}$.

Given a set $E$ of extreme rays of $\glngeo$ or a subset $E \subset \posetz_n$, denote by $E_0$ (resp.\ $E_+,$ $E_-$) the set of elements of $E$ with zero (resp.\ positive, negative) signature.

\section{Formal statements of results}
\label{sec:formalStatements}

In \pref{sec:faceFlatStatements}, we present generating functions for the numbers of chambers, $k$-faces and $k$-flats of $\slngeo$.
In \pref{sec:extremeRayStatements}, we discuss the structure of the set of extreme rays upon which the theorems of \ref{sec:faceFlatStatements} depend; as a corollary, we classify simplicial chambers in $\slngeo$.

\subsection{Enumerations of chambers, faces, and flats of \slngeoForTitle}
\label{sec:faceFlatStatements}

First, we recall the notation for formal power series.
Given a ring $A$ and a set $X$, we denote by $A\llbracket X \rrbracket$ the ring of formal power series with coefficients in $A$ in the alphabet $X$.  Let $R = \mathbb{Q}\llbracket x,y\rrbracket$.  We denote by $\sqrt{F}$ the principal square root of $F$ for $F \in R$ such that $F(0,0)=1$ (resp.\ $F \in R\llbracket z \rrbracket$ such that $F(0,0,0)=1$), whose existence and uniqueness are guaranteed by Hensel's Lemma.\footnote{Recall the statement of Hensel's Lemma: suppose that $A$ is a complete local ring with maximal ideal $\mathfrak{m}$ and residue field $k$ and that $p \in A[z]$.  Let $r \in k$ be such that $p(r) = 0$ (as an element of $k$).  If $p'(r) \not= 0$ (as an element of $k$), then there exists a unique $\tilde{r} \in A$ reducing to $r$ modulo $\mathfrak{m}$ such that $p(\tilde{r}) = 0$ (as an element of $A$).}

We denote by $b(n,k)$ the number of $k$-faces of $\slngeo$ and by
$$
B(x,y) = \sum_{n=0}^\infty \sum_{k=0}^\infty b(n,k)x^ny^k
$$
the corresponding generating function.

\begin{theorem}
\label{thm:slnFaces}
The generating function for the number of $k$-faces of $\mathrm{I}(\mathfrak{sl}_n, V\oplus \bigwedge{}^2)$ is
$$
\begin{aligned}
B(x,y) =& -\frac{4}{y\left(1 - 2 x + x^2 - 2 x y + 
 x^2 y + \sqrt{\alpha}\right)}\\
&+\frac{2 - 2 x - 2 x^2 + 2 x^3 + y - 3 x y - 3 x^2 y + 3 x^3 y - 2 x y^2 - 
 x^2 y^2 + x^3 y^2}{y (1 - 2 x + 2 x^3 - x^4 - 2 x y + 4 x^3 y - 2 x^4 y + 2 x^3 y^2 - x^4 y^2)}
 \end{aligned}
$$
in the fraction field of $R$, where $\alpha \in R$ is the polynomial defined as
$$
\alpha = 1 - 2 x^2 + x^4 - 6 x^2 y + 2 x^4 y - 4 x^2 y^2 + x^4 y^2
$$
\end{theorem}
\pref{thm:slnFaces} is proved in \pref{app:proofsOfSlnFacesAndFlats}.

\begin{definition}
Suppose that $C$ is a chamber of $\glngeo$ with characteristic vector $s = (s_1,\ldots,s_n) \in \{-1,1\}^n$.
We say that an index $1 \le i \le n$ is a \emph{null} (resp.\ \emph{positive}, \emph{negative}) \emph{level for } $C$ if $s_{n-i+1}+\cdots+s_n$ is zero (resp.\ positive, negative).
\end{definition}

In \pref{prop:signOfaLevelMeaning}, we will see that $k$ is a null (resp.\ positive, negative) level for $C$ if and only if  the $k$th extreme ray of $e_k$ of $C$ has signature zero (resp.\ positive signature, negative signature), where $e_k$ is the extreme ray of $C$ defined in \pref{thm:chamberStructure}\pref{assert:chamberRaysPrelim}.

\begin{theorem}
\label{thm:chamberCount}
The chambers of $\slngeo$ satisfy the following properties.
\begin{enumerate}[(a)]
\item \label{assert:chamberBiSignBij} For $n \ge 3$, every chamber of $\slngeo$ is uniquely expressible as $C \cap \h_s,$ where $C$ is a chamber of $\glngeo$.
\item \label{assert:chamberCount} The number of chambers in $\slngeo$ is 
$$
2^n - 2\binom{n}{\floor*{\frac{n}{2}}} + \delta_{n,0} + \delta_{n,1} + \delta_{n,2},
$$
which is the coefficient of $x^n$ in the generating function
$$
\frac{1}{1-2 x}-\frac{\sqrt{1-4 x^2}+ 2x-1}{x (1-2 x)}+1+x+x^2=\frac{2 x^4+x^3+x^2+\sqrt{1-4 x^2}-1}{x (2 x-1)}.
$$
\end{enumerate}
Suppose that $C$ is a chamber of $\glngeo$ and that
$$
S = \{a_1 > \cdots > a_k\} \subseteq [n]
$$
corresponds to $C$ under the bijection of \pref{def:subsetToChamber}. Assume that $n \ge 3$.  Then:
\begin{enumerate}[(a)]
\setcounter{enumi}{2}
\item \label{assert:chamberSubsetBisign} The space $C \cap \h_s$ is a chamber of $\slngeo$ if and only if there exist indices $1 \le i,j \le \ceil*{\frac{n}{2}}$ such that
$$
\begin{aligned}
a_i &\geq  n+2-2i
a_j &< n+2-2j.
\end{aligned}
$$
This occurs if and only if $C$ has a positive level and a negative level.
\end{enumerate}
\end{theorem}

\begin{figure}
\begin{multicols}{2}
\begin{minipage}{ \linewidth}
\centering{
{\raisebox{0cm}{  \includegraphics[scale=.58]{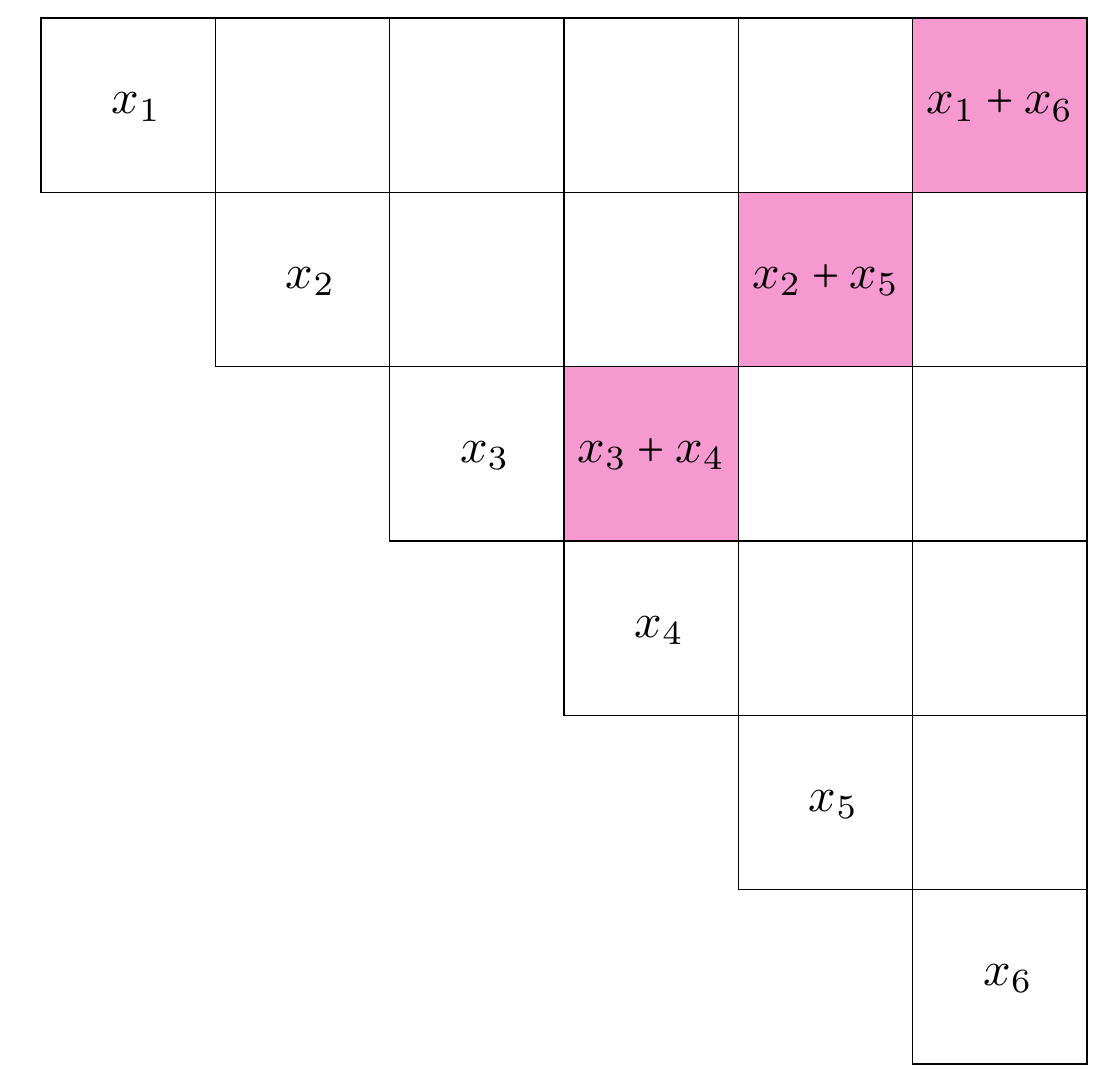}}}
}
\end{minipage}

\begin{minipage}{\linewidth}
 \centering{
  \includegraphics[scale=.55]{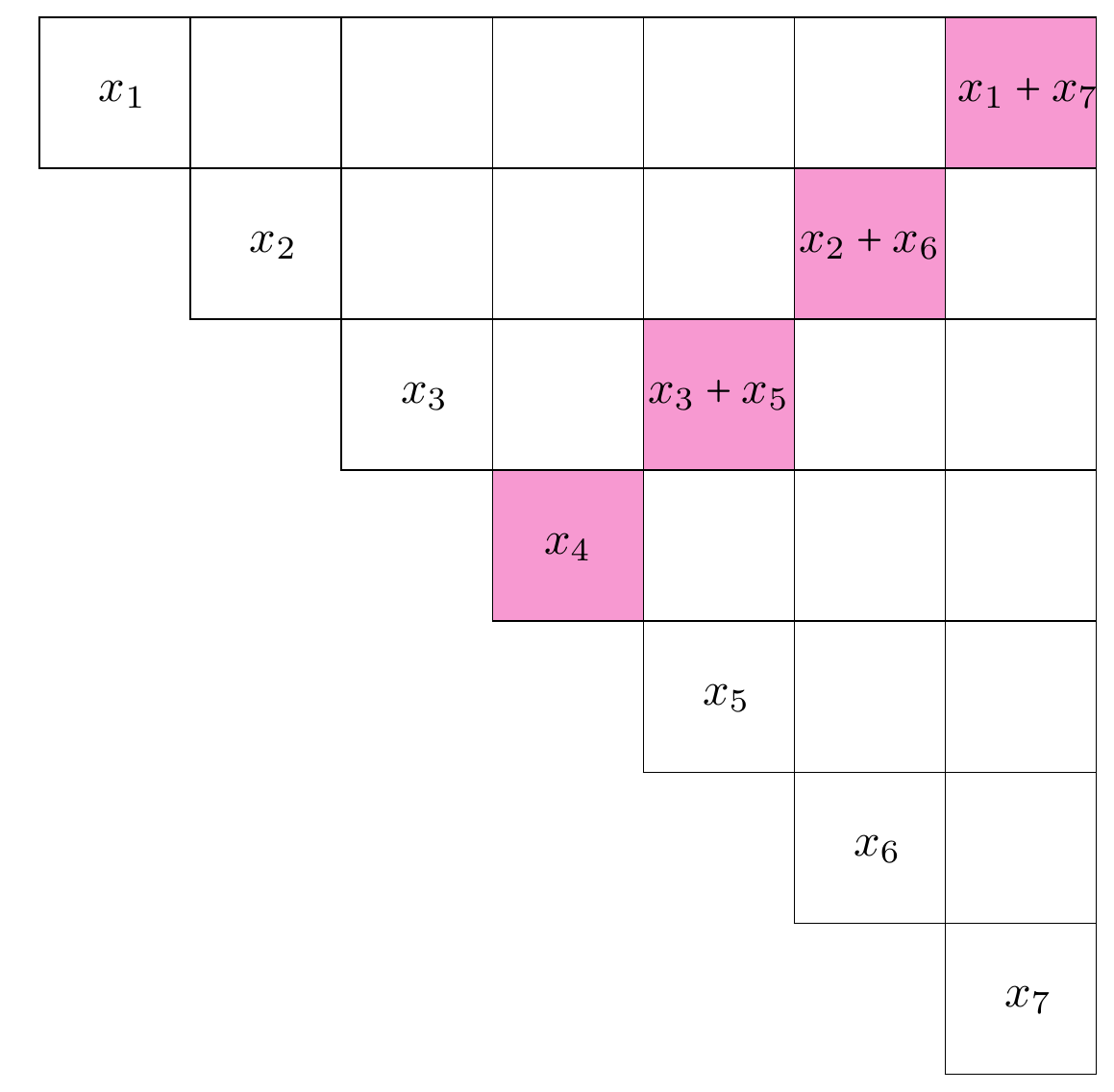}
  }
\end{minipage}
\end{multicols}
\caption{
{\bf Sign rules for chambers of $\slngeo$}.\pref{thm:chamberCount}\pref{assert:chamberSubsetBisign}.
For any $n\geq 3$, a tableau corresponds to a chamber of $\slngeo$ if and only if (1) it corresponds to a chamber of $\glngeo$ (it satisfies the sign rules discussed in Figure  \ref{fig:signTableau}); 
(2) The NE-SW diagonal (shaded in the figures) have entries of both positive and negative signs.
The second condition is clearly necessary since the  elements of the NE-SW diagonal add up to $x_1+\cdots+x_n$, which has to be zero. 
The condition is also sufficient as proven in \pref{thm:chamberCount}\pref{assert:chamberSubsetBisign}. 
The figure on the  left (resp.\ right) illustrate the NE-SW diagonal for  $n$ even  (resp.\ $n$ odd).  
 See for example \pref{fig:YT345}, \pref{fig:YT_SU6}, and \pref{fig:SimplicialSU6}. 
\label{fig:YToddEven}  }
\end{figure}

See \pref{fig:YToddEven} for a geometric formulation of \pref{thm:chamberCount}\pref{assert:chamberSubsetBisign}.

Theorem \ref{thm:chamberCount} is proved in  \pref{sec:chambers}.
We present a proof of \pref{thm:chamberCount}\pref{assert:chamberCount} that is similar to that of \cite{box} in order to explain the structure of the more delicate proofs of Theorems~\ref{thm:slnFaces} and~\ref{thm:slnFlats}.

\begin{remark}
\pref{thm:chamberCount}\pref{assert:chamberSubsetBisign} characterizes the chambers of $\slngeo$. It can be reformulated as follows. 
For $n$ even (resp.\ odd),  we denote by $\vec{n}$ the $\ceil*{\frac{n}{2}}$-vector   $(n,n-2, n-4, \cdots,  2)$ (resp.\ $\vec{n}=(n,n-2, \cdots, 1)$) whose components are  the first $\ceil*{\frac{n}{2}}$   even (resp.\ odd) positive integers rearranged in  decreasing order.   
We denote by   $\test{n}{S}$ the decreasing vector  whose components are  the first  $\ceil*{\frac{n}{2}}$ elements of $S$ rearranged in decreasing order and  padded with zeros at the end if $S$ has less than $\ceil*{\frac{n}{2}}$ elements.  
For $n\geq 3$, a subset $S\in 2^{[n]}$ corresponds to a chamber or $\slngeo$ if and only if 
  $$\test{n}{S} \not< \vec{n}\quad \text{ and } \quad \test{n}{S}\not\geq \vec{n},$$
componentwise.
\end{remark}

\begin{example}
If $n=5$, we identify the subsets that correspond to elements of $\slngeo$ by comparing  the vector formed by their  first $3$ bigger elements (padded with zeros at the end if necessary) with  
the vector  $\vec{5}=(5,3,1)$. The subset $\{5,4,2\}$ does not correspond to a chamber of  $\slngeoSm{5}$ since componentwise $(5,4,2)\geq (5,3,1)$. 
The subset   $\{5\}$ corresponds to a chamber of  $\slngeoSm{5}$ since $\test{5}{\{5\}}=(5,0,0)$ and $5\geq 5$ while $0<3$. 
The subset  $\{4,3,2,1\}$ corresponds to a chamber of  $\slngeoSm{5}$ since $\test{5}{\{4,3,2,1\}}=(4,3,2)$ and  $4\leq 5$ but $2>1$.
${}$
\end{example}

\begin{example} For  $\slngeoSm{6}$,  $\vec{6}=(6,4,2)$. 
The subset  $\{5,4,2\}$ does not correspond to a chamber of  $\slngeoSm{6}$ since $(5,4,2)\leq (6,4,2)$. 
 $\{5\}$ does not correspond to a chamber of  $\slngeoSm{6}$ since $(5,0,0)<(6,4,2)$.
The subset  $\{6,5\}$ does correspond to a chamber of  $\slngeoSm{6}$ since  $6\geq 6$ but $5>4$.
\end{example}

\begin{remark}\pref{thm:chamberCount}\pref{assert:chamberCount} can be derived directly from \pref{thm:slnFaces}. Recall that for $n \ge 1$, a chamber of $\slngeo$ is defined to be a $(n-1)$-face of $\slngeo$.  In particular, there are $b(n,n-1)$ chambers of $\slngeo$.  Using residues to compute an integral,  we obtain that
$$
\sum_{n=1}^\infty b(n,n-1)x^n = 
\underset{z=0}{Res}  
\frac{ x B(z,x z^{-1})}{z^2} =\frac{2 x^4+x^3+x^2+\sqrt{1-4 x^2}-1}{x (2 x-1)},
$$
and therefore recover \pref{thm:chamberCount}\pref{assert:chamberCount}.
\end{remark}

We denote by $f(n,k)$ the number of $k$-flats of $\slngeo$ and by
$$
F(x,y) = \sum_{n=0}^\infty \sum_{k=0}^\infty f(n,k)x^ny^k
$$
the corresponding generating function.  We will prove the following formula for $f(n,k)$.

\begin{theorem}
\label{thm:slnFlats}
The generating function counting flats of $\mathrm{I}(\mathfrak{sl}_n, V\oplus \bigwedge{}^2)$ by dimension is given by
$$
\begin{aligned}
F(x,y) = & \frac{2}{y(1-x)} + \frac{1 - x + 2 x^3 - 2 x^4 - x^5 + x^6 - 2 x^2 y + x^3 y + x^4 y - x^5 y}{(1-x)\left(1 - x^2 (1 + y)\right)^2}\\
&+ \frac{x - x^3 + x^4 - x^2 y + 2 x^3 y - x^4 y}{(1-x)(1 - 2 x + x^2 - 2 x y + 3 x^2 y - 2 x^3 y + x^2 y^2 - 2 x^3 y^2 + 
 x^4 y^2)}\\
&-\frac{8 \left(1 + x^2 - x y - x^3 y + x^4 y + x^2 y^2 + x^3 y^2 - x^3 y^3 + 
 x^4 y^3 - x^5 y^3 + x^6 y^3\right)}{y\left(\beta+x\sqrt{\gamma}+\sqrt{\zeta+\eta\sqrt{\gamma}}\right)\left(\beta-x\sqrt{\gamma}+\sqrt{\zeta-\eta\sqrt{\gamma}}\right)}
\end{aligned}
$$
in the fraction field of $R$, where $\beta,\gamma,\zeta,\eta \in R$ are polynomials defined as
$$
\begin{aligned}
\beta &= 1 - x + x^2 - x y - x^2 y - x^3 y + x^2 y^2 - x^3 y^2 + x^4 y^2\\
\gamma &= 1 - 2 y - 2 x^2 y + y^2 + 8 x^2 y^2 + x^4 y^2 - 2 x^2 y^3 - 2 x^4 y^3 + x^4 y^4\\
\zeta &= 1 + x^4 - 2 x^2 y - 2 x^4 y - 7 x^4 y^2 - 2 x^4 y^3 - 2 x^6 y^3 + x^4 y^4 + x^8 y^4\\
\eta &= %2(x^2 + x^2 y + x^4 y + x^4 y^2)=
x^2 (1+y)(1+x^2 y)
\end{aligned}
$$
\end{theorem}
Theorem \ref{thm:slnFlats}  is proved in \pref{app:proofsOfSlnFacesAndFlats}.

\subsection{Extreme rays and simplicial chambers of \slngeoForTitle}
\label{sec:extremeRayStatements}

First, we generalize the function $\rayFn$ to the case of $\slngeo$.

\begin{definition}
Let $\mathfrak{B}$ be  the set of ordered pairs of elements of $\posetz_n$ with  distinct (non-vanishing) signature
$$
\mathfrak{B} = \{(p,q) \in \left(\posetz_n\right)^2 \mid p < q \text{ and } \sigma(p)\sigma(q) < 0\},
$$
and let $\mathfrak{N}$ denote the set of elements of $\posetz_n$ of signature zero.
We refer to elements of $\mathfrak{N}$ as \emph{null vectors} and to elements of $\mathfrak{B}$ as \emph{bisigned pairs}.
The function $\rayFn: \mathfrak{B} \to \h_s$ associates to each pair $(p,q)\in \mathfrak{B}$, the unique convex combination of $\ray{p}$ and $\ray{q}$ with signature zero. 
For $(p,q) \in \mathfrak{B}$, the vector $\ray{p,q}$ is defined as follows
\begin{align*}
\ray{p,q} = & |\sigma(q)| \ray{p}+|\sigma(p)|\ray{q}\\
=& (\underbrace{\mu,\ldots,\mu}_a,\underbrace{\nu,\ldots,\nu}_{c-a},\underbrace{0,\ldots,0}_{n-c-d},\underbrace{-\nu,\ldots,-\nu}_{d-b},\underbrace{-\mu,\ldots,-\mu}_b),
\end{align*}
where $p=(a,b)$, $q=(c,d)$, $\mu =|a-b|+|c-d|$ and $\nu=|a-b|$. 
\end{definition}

If $E$ is any subset of $\posetz_n$ and $(p,q)\in\mathfrak{B}$, we abuse notation and write $(p,q) \in E$ to denote that $p \in E$ and $q \in E$.  The analogue of \pref{thm:chamberStructure}\pref{assert:chamberRaysPrelim} for $\slngeo$ is the following result.

\begin{theorem}
\label{thm:extremeRaysOfSln}
The function $\rayFn$ induces a bijection from $\mathfrak{N} \sqcup \mathfrak{B}$ to the set of extreme rays of $\slngeo$.
\end{theorem}

Theorem \ref{thm:extremeRaysOfSln} is proved in \pref{app:extremeRaysAndConseqs}.
We will now apply   \pref{thm:extremeRaysOfSln} to obtain results regarding the extreme rays of a chamber of $\slngeo$.

\begin{theorem}
\label{thm:chamberRays}
The chambers of $\slngeo$ satisfy the following properties.
\begin{enumerate}[(a)]
\item \label{assert:countSimplicialChambers} 
The number of simplicial chambers in $\slngeo$ is
$$
2\binom{n-1}{\floor*{\frac{n-1}{2}}}-\delta_{n,1}-\delta_{n,2}-2\delta_{n,3}-2\delta_{n,4}
$$
It follows that the generating function for the number of simplicial chambers of $\slngeo$ is 
$$
\frac{4 x^5+2 x^4+3 x^2+\sqrt{1-4 x^2}-1}{1-2x}
$$
\end{enumerate}
Suppose that $C$ is a chamber of $\glngeo$ with characteristic vector $s = (s_1,\ldots,s_n) \in \{-1,1\}^n$.
 Assume that $n \ge 3$ and that $C \cap \h_s$ is a chamber of $\slngeo$.  Let $\ell_+$ (resp.\ $\ell_-$,$\ell_0$) denote the number of positive (resp.\ negative, null) levels for $C$. 
  Let $e_k$ denote the extreme ray of $C$ at level $k$. 
\begin{enumerate}[(a)]
\setcounter{enumi}{1}
\item \label{assert:nullRay} Let $1 \le i \le \floor*{\frac{n}{2}}$.  The ray $\ray{i,i}$ lies in $C \cap \h_s$ if and only if $\ray{i,i}$ is an extreme ray of $C$.  This occurs if and only if $2i$ is a null level for $C$.
\item \label{assert:nonNullRay} Let $(p_1,p_2) \in \mathfrak{B}$.  The ray $\ray{p_1,p_2}$ lies in $C \cap \h_s$ if and only if $\ray{p_1}$ and $\ray{p_2}$ are extreme rays of $C$.  This occurs if and only if the equality
$$
\sigma (e_{\ell_i})=%\sum_{k=n-\ell(p_i)+1}^n s_k =
 \sigma(p_i),
$$
holds for $i = 1,2$ and $\ell_i=\ell(p_i)$.
\item \label{assert:extremeRayCount} The number of extreme rays of $C$ is $\ell_+\cdot\ell_- + \ell_0$.  In particular, $C \cap \h_s$ is simplicial if and only if $\min\{\ell_+,\ell_-\} = 1$.
\end{enumerate}
\end{theorem}

Theorem  \ref{thm:chamberRays}(\ref{assert:countSimplicialChambers})  is proven in  \pref{app:simplicialChambers} and the rest of theorem is proven in \pref{app:extremeRaysOfFacesAndChambers}.

\begin{table}
\begin{center}
\scalebox{.9}{
\begin{tabular}{ c| c |c}
 Chamber $C$  &Extr. rays of  $C$ $\glngeoSm{6}$ & Extr. rays  of $C\cap\h_s$ in $\slngeoSm{6}$\\
\hline 
\raisebox{-1cm}{\includegraphics[scale=.6]{YT_6_65}}
&

\begin{tabular}{p{3.2cm}c}
\raisebox{-1.5cm}{ \includegraphics[scale=.8]{ER_6_65}}
&
$
\begin{aligned}
\ray{1,0}\\
\ray{2,0} \\
 \ray{2,1} \\
  \ray{2,2}\\ 
   \ray{2,3}\\
    \ray{2,4}
\end{aligned}
$
\end{tabular}
& 
$
\begin{aligned}
\ray{2,2} &= (1,1,\  \   0,\  \   0,-1,-1)\\
\ray{(1,0),(2,3)} &= (2,1,\  \   0,-1,-1,-1)\\ 
\ray{(1,0),(2,4)} &=(3,1,-1,-1,-1,-1)\\
\ray{(2,0),(2,3)} &=(3,3,\  \   0,-2,-2,-2)\\
\ray{(2,0),(2,4)} &=(2,2,-1,-1,-1,-1)\\
\ray{(2,1),(2,3)} &= (2,2,\  \   0,-1,-1,-2)\\ 
\ray{(2,1),(2,4)} &=(3,3,-1,-1,-1,-3)
\end{aligned}
$
\end{tabular}
}
\end{center}
\caption{  
{\bf Example of  a non-simplicial chamber of $\slngeoSm{6}$}.  
The chamber $C$ defined by the subset $\{6,5\}$ of $\{1,2,\ldots, 6\}$ is simplicial as all the chambers of  
$\glngeoSm{6}$. Its extreme rays are listed in the second column of the table. 
Its intersection of $C$ with the hyperplane $ \h_s:x_1+\cdots+x_6=0$ determines a chamber $C\cap \h_s$ of $\slngeoSm{6}$. 
The chamber $C\cap \h_s$ is {\bf not simplicial} as discussed in Example \ref{Example:ER2} on page  \pageref{Example:ER2}.
The extreme rays of $C\cap \h_s$ as a chamber of $\slngeoSm{6}$ are listed on the third column and their construction is discussed in Example \ref{Example:ER2}.
\label{Table:C65}}
\end{table}

\begin{remark}
For $n \le 5,$ all chambers of $\slngeo$ are simplicial, so that $\slngeo$ is simplicial.  Non-simplicial chambers exist for $n \ge 6$.  This can be verified directly for $n \le 2$, and follows from \pref{thm:chamberCount}\pref{assert:chamberCount} and \pref{thm:chamberRays}\pref{assert:countSimplicialChambers} for $n \ge 3$.
\end{remark}
\begin{example}\label{Example:ER2}
We give an explicit example of a non-simplicial chamber of $\slngeoSm{6}$. Consider the chamber $C$ defined by the subset $\{6,5\}$ of $\{1,2,\ldots, 6\}$ or equivalently by the characteristic vector 
 $(-,-,-,-,+,+)$. See Table \ref{Table:C65}. 
Note that $C$ has positive,  null, and negative levels as follows (see Table  \ref{Table:C65} on page \pageref{Table:C65})
\begin{align*}
\text{positive levels for } C: &\quad 1,2,3\\
\text{null level for } C: &\quad 4\\
\text{negative levels for } C: &\quad 5,6.
\end{align*}
\pref{thm:chamberCount}\pref{assert:chamberSubsetBisign} guarantees that $C \cap \h_s$ is a chamber of $\slngeo$.  \pref{thm:chamberRays}\pref{assert:extremeRayCount} implies that $C \cap \h_s$ is not simplicial, and furthermore that it has 7 extreme rays in $\slngeo$.
It follows from \pref{thm:chamberRays}\pref{assert:nullRay} and \pref{thm:chamberRays}\pref{assert:nonNullRay} that the extreme rays of $C \cap \h_s$ in $\slngeo$ are
$$
\ray{2,2} \text{ and } \left(\ray{p,q} \text{ for } p \in \{(1,0),(2,0),(2,1)\} \text{ and } q \in \{(2,3),(2,4)\}\right).
$$
 The resulting seven extreme rays are given in the third column of  Table \ref{Table:C65}.
\pref{thm:chamberCount}\pref{assert:chamberCount} and \pref{thm:chamberRays}\pref{assert:countSimplicialChambers} guarantee that $\slngeoSm{6}$ has 4 non-simplicial chambers.  A direct calculation similar to the one presented in the previous paragraph shows that they are the chambers $C \cap \h_s$ when $C$ is a chamber of $\glngeo$ defined by one of the subsets $\{6,5\},\{6,4\},\{4,3,2,1\}$, and $\{5,3,2,1\}$. 
\end{example}

\section{Chambers of \slngeoForTitle: proof of \pref{thm:chamberCount}}
\label{sec:chambers}

In \pref{sec:countChambers}, we prove Part \pref{assert:chamberBiSignBij} and \pref{assert:chamberCount}
of \pref{thm:chamberCount}. In \pref{sec:bisignedChamberSubset}, we derive \pref{thm:chamberCount}\pref{assert:chamberSubsetBisign} from the auxiliary results of \pref{sec:countChambers}.

\subsection{Counting chambers of \slngeoForTitle}
\label{sec:countChambers}

The goal of this section is to prove Part \pref{assert:chamberBiSignBij} and \pref{assert:chamberCount} \pref{thm:chamberCount} by relating $\slngeo$ to $\glngeo$.  The proof of \pref{thm:chamberCount} is the prototype for the proofs of Theorems~\ref{thm:slnFaces} and~\ref{thm:slnFlats}: the later proofs are more delicate, but follow a similar structure to the proof presented in this section.

\begin{definition}
\label{def:phiForChambers}
We say that a chamber $C$ of $\glngeo$ is \emph{bi-signed} if $C^0 \cap \h_s \not= \varnothing$.
(This terminology is justified by \pref{lem:classifyBisignedChains}.)
We denote by $\mathfrak C_{\text{bi}}(n)$ the set of bi-signed chambers of $\glngeo$ and by $\mathfrak{C}_s(n)$ the set of chambers of $\slngeo$.
\end{definition}

First, we relate chambers of $\slngeo$ to those of $\glngeo$.

\begin{lemma}\label{lem:bisignedGivesSlnChamber}
If $C$ is a bi-signed chamber of $\glngeo$, then $C\cap \h_s$ is a chamber of $\slngeo$.
\end{lemma}
\begin{proof}
We see immediately that $C \cap \h_s$ is a face of $\slngeo$.
Because $C$ is a chamber, the set $C^0$ is open in $\h$.  Therefore, the set $C^0 \cap \h_s = \left(C \cap \h_s\right)^0$ is open in $\h_s$.  It follows that $C \cap \h_s$ is a chamber of $\slngeo$.
\end{proof}

\begin{proposition}
\label{prop:chamberBijection}
The map $\phi: \mathfrak{C}_{\text{bi}}(n) \to \mathfrak{C}_s(n)$ defined by $\phi(C) = C\cap \h_s$ is a bijection for $n \ge 3$.
\end{proposition}

\begin{proof} The map $\phi$ is well-defined by \pref{lem:bisignedGivesSlnChamber}. To prove surjectivity assume that $D$ is in $\mathfrak C_s$.  Because $n \ge 3$, the subspace $\h_s \subset \h$ is not contained in any of the $\mathfrak{gl}_n$-walls.  Hence, there exists a $x \in D$ that does not lie on any of the $\mathfrak{gl}_n$-walls.  There exists a unique chamber $C \in  \mathfrak C$ of $\glngeo$ containing $x$.  Because $x$ lies in the interior of $C$, the chamber $C$ must contain elements lying on both sides of $\h_s$.  It follows that $C \in \mathfrak{C_{\text{bi}}}$ and $\phi(C) = D$. 

Suppose that $\phi(C_1)=\phi(C_2)=D$. Let $x_1 \in C^0_1\cap \h_s$ and $x_2  \in C^0_2\cap \h_s$ then by convexity  the segment $[x_,x_2]$ lie in $D$. Also it does not intersect any wall. It follows that $x_1$ and $x_2$ belongs to the same $\mathfrak{gl}_n$-chamber. In other words, we have $C_1 = C_2$, which proves the injectivity of $\phi$.
\end{proof}

Next, we characterize the elements of $\mathfrak{C_{\text{bi}}}$ in terms of extreme rays.
% to help us determine $\left|\mathfrak{C_{\text{bi}}}\right|$.

\begin{lemma}\label{lem:classifyBisignedChains}
A chamber of $\glngeo$ is bi-signed if and only if it has extreme vectors of positive and negative signatures.
\end{lemma}

\begin{proof}
Let $C$ be a chamber of $\glngeo$.  Consider the set $E = \erays(C)$ of extreme rays of $C$.  \pref{thm:chamberStructure}\pref{assert:chamberRaysPrelim} guarantees that $\left|E\right| = n$.
We see immediately that $\left|E_0\right| \le \floor*{\frac{n}{2}} < n$.  It follows that $\left|E_+\right| + \left|E_-\right| > 0$.

Suppose that $\left|E_-\right| = 0$, so that $\left|E_+\right| > 0$. A direct calculation from \pref{thm:chamberStructure}\pref{assert:chamberRaysPrelim} shows that $\sigma(C) = \mathbb{R}_{\ge 0}$.  It follows that $C \cap \h_s$ lies on the boundary of $C$, so that $C$ is not bi-signed.

By symmetry, we can also conclude that $\left|E_+\right| > 0$ if $C$ is bi-signed.  The lemma follows.
\end{proof}

We are now ready to prove \pref{thm:chamberCount}\pref{assert:chamberBiSignBij}.

\begin{proposition}
\label{prop:bisignedEquivToGetChamber}
Suppose that $n \ge 3$.  A chamber $C$ of $\glngeo$ is bi-signed if and only if $C \cap \h_s$ is a chamber of $\slngeo$.
\end{proposition}
\begin{proof}
\pref{lem:bisignedGivesSlnChamber} proved the ``only if direction."  We now prove the ``if" direction.

Suppose that $C$ is a chamber of $\glngeo$ that is not bi-signed and let $E = \erays(C)$.  Without loss of generality, in light of \pref{lem:classifyBisignedChains}, assume that all the extreme rays of $C$ have non-negative signature.  A direct calculation from \pref{thm:chamberStructure}\pref{assert:chamberRaysPrelim} shows that $C \cap \h_s$ is non-negatively spanned by $E_0$.  Note that $\left|E_0\right| \le \floor*{\frac{n}{2}}$ due to \pref{thm:chamberStructure}\pref{assert:chamberRaysPrelim}.  It follows that $C \cap \h_s$ has dimension at most $\floor*{\frac{n}{2}}.$  Because $\floor*{\frac{n}{2}} < n-1$ for $n \ge 3,$ it is not possible for $C \cap \h_s$ to be a chamber of $\slngeo$.
\end{proof}

\begin{proof}[Proof of \pref{thm:chamberCount}\pref{assert:chamberBiSignBij}]
\pref{prop:bisignedEquivToGetChamber} ensures that $C \cap \h_s$ is not a chamber of $\slngeo$ when $C$ is not bi-signed, and so we can restrict to the case when $C$ is a bi-signed chamber of $\glngeo$.
\pref{prop:chamberBijection} ensures that every chamber in $\slngeo$ is uniquely expressible in the form $C \cap \h_s,$ where $C$ is a bi-signed chamber of $\glngeo,$ for $n \ge 3$, which is what was to be proved.
\end{proof}

\pref{lem:classifyBisignedChains} motivates the following definition, which will be useful in later sections as well.

\begin{definition}
We say that a subset of $\posetz_n$ is \emph{mono-signed} if all of its elements have non-negative signature or all of its elements have non-positive signature. On the other hand, we say that a subset of $\posetz_n$ \emph{bi-signed} if it is not mono-signed: that is, if it has at least one element of negative signature and at least one element of positive signature.
\end{definition}

To complete the proof of \pref{thm:chamberCount}, we simply need to count bi-signed $n$-chains in $\posetz_n$ in light of the preceding results and \pref{thm:Characterization.Faces}.  

\begin{proof}[Proof of \pref{thm:chamberCount}\pref{assert:chamberCount}]
The theorem is obvious for $n \le 2$.  For the remainder of the proof, we assume that $n \ge 3$.

\pref{prop:chamberBijection} and \pref{lem:classifyBisignedChains} guarantee that there are the same number of chambers of $\slngeo$ as there are bi-signed $n$-chains in $\posetz_n$.
Theorems~\ref{thm:chamberStructure} and~\ref{thm:Characterization.Faces} guarantee that $2^n$ counts the number of $n$-chains in $\posetz_n$.  Therefore, it suffices to prove that there are $2\binom{n}{\floor*{\frac{n}{2}}}$ mono-signed $n$-chains.

First, we count $n$-chains in $\posetz_n$ all of whose elements have non-negative signature.  Following~\cite{gl}, we can regard $n$-chains in $\posetz_n$ as lattice paths from $(0,0)$ to the line $x+y=n$ with permissible steps $(1,0)$ and $(0,1)$ by regarding $\posetz_n$ as a subset of $\mathbb{Z}^2$.  An $n$-chain consists solely of elements of non-negative signature if and only if the corresponding lattice path lies in the half-plane $y \le x$.  \pref{prop:DyckWithNoRestrictedEndpt} therefore implies that  $\binom{n}{\floor*{\frac{n}{2}}}$ counts the number of $n$-chains in $\posetz_n$ consisting solely of elements of non-negative signature.
By symmetry, $\binom{n}{\floor*{\frac{n}{2}}}$ is also the number of $n$-chains in $\posetz_n$ consisting solely of elements of non-positive signature.

Because $n \ge 1$, any $n$-chain in $\posetz_n$ must contain one of $(1,0)$ and $(0,1)$.  Hence, it is not possible for a $n$-chain in $\posetz_n$ to consist solely of elements of signature 0.  Therefore, the set of $n$-chains in $\posetz_n$ consisting of elements of non-negative signature is disjoint from the set of $n$-chains in $\posetz_n$ consisting of elements of non-positive signature. It follows that there are $2\binom{n}{\floor*{\frac{n}{2}}}$ mono-signed $n$-chains in $\posetz_n$, as desired.  We have to introduce the correction terms $\delta_{n,0} + \delta_{n,1} + \delta_{n,2}$ in order to deal with the case of $n \le 2$.
\end{proof}

\subsection{Simple criterion for a chamber of \glngeoForTitle to be bi-signed}
\label{sec:bisignedChamberSubset}

To prove \pref{thm:chamberCount}\pref{assert:chamberSubsetBisign}, we translate the bi-signed condition of the previous subsection to a constraint on the subset of $[n]$ corresponding to a chamber of $\glngeo$ using the following proposition.

\begin{proposition}
\label{prop:signOfaLevelMeaning}
Suppose that $C$ is a chamber of $\glngeo$.  A level $k \in [n]$ is a null (resp.\ positive, negative) level for $C$ if and only if $\sigma(e_k) = 0$ (resp.\ $\sigma(e_k) > 0,$ $\sigma_{e_k} < 0$), where $e_k$ is the extreme ray of $C$ defined in \pref{thm:chamberStructure}\pref{assert:chamberRaysPrelim}.
\end{proposition}
\begin{proof}
It follows from the definition of $e_k$ and the definition of the signature that
$$\sigma(e_k) = s_{n-k+1}+\cdots+s_n,
$$
where $(s_1,\ldots,s_n) \in \{-1,1\}^n$ is the characteristic vector of $C$.
We obtain the proposition immediately from the definition of a null (resp.\ positive, negative) level for $C$.
\end{proof}

\begin{proof}[Proof of \pref{thm:chamberCount}\pref{assert:chamberSubsetBisign}]
Let $s = (s_1,\ldots,s_n) \in \{-1,1\}^n$ denote the characteristic vector of $C$, which is also by definition the characteristic vector of a subset $S=\{a_1 > a_2>\ldots> a_r \}$ of $[n]$.
Let $e_1,\ldots,e_n$ denote the extreme rays of $C$ defined in
\pref{thm:chamberStructure}\pref{assert:chamberRaysPrelim}.
For $n \ge 3$, \pref{lem:classifyBisignedChains} and \pref{prop:bisignedEquivToGetChamber} ensure that $C \cap \h_s$ is a chamber of $\slngeo$ if and only if there exist $i,j$ such that $\sigma(e_i) > 0 > \sigma(e_j)$.
It follows from \pref{prop:signOfaLevelMeaning} that $C \cap \h_s$ is a chamber of $\slngeo$ if and only if $C$ has a positive level and a negative level, provided that $n \ge 3$.

Since 
$
 \sigma(e_\ell)=2\pi_\ell -\ell
$,
the extreme ray $e_\ell$ has positive signature if and only if there are at least $\floor*{\ell/2 }+1$ elements of $S$ in $[n-\ell+1,n]$. 
As the elements  $a_i$ of $S$ are in decreasing order, we deduce  that 
\begin{align*}
\ell \text{ is a positive level for } C &\Leftrightarrow a_{\floor*{\frac{\ell}{2}}+1} \geq n-\ell+1, \\
\ell \text{ is a negative level for } C & \Leftrightarrow   a_{\ceil*{\frac{\ell}{2}}}\   \    < n-\ell+1.
\end{align*}
In particular, if an even level $\ell$  is positive (resp.\ negative), so is the odd level $\ell+1$ (resp.\ $\ell-1$). 
It is therefore enough to consider only the signature of odd levels. A direct calculation  using the odd level 
$\ell=2i-1$ 
shows that there exists a positive (resp.\ negative) level for $C$ if and only if there exists $i$ such that $a_i \leq n-2i+2$ (resp.\ $a_j < n-2i+2$). The theorem follows.
\end{proof}

\section{Relating faces and flats of \slngeoForTitle to faces and flats of \glngeoForTitle}
\label{sec:relateSlnFacesAndFlatsToGln}

We would like to prove analogues of \pref{prop:chamberBijection} for faces and flats.
We do so in Sections~\ref{sec:bijectForFaces} and~\ref{sec:bijectForFlats}, respectively.
The arguments of this section are more delicate that those of \pref{sec:countChambers} because the natural analogue of \pref{prop:bisignedEquivToGetChamber} does not hold any longer.
For this reason, we will construct mappings from faces (resp.\ flats) in $\slngeo$ to faces (resp.\ flats) in $\glngeo$ instead of dealing with maps in the other direction.
In the notation of \pref{def:phiForChambers}, we work with $\phi^{-1}$ instead of $\phi$.

\subsection{Bijection for faces}
\label{sec:bijectForFaces} We define a map $\psi$ from the set  $\mathfrak{F}_s$ of faces of $\slngeo$ to the set $\mathfrak{F}$ of faces of $\glngeo$, prove that it is injective, and determine its image.

\begin{definition}
\label{def:psiForFaces}
Define a function $\psi: \mathfrak{F}_s \to \mathfrak{F}$ by
$$
\psi(F_s) = W \cap \left(\bigcap_{\lambda \in \weights_0(F_s)} \lambda^\perp\right) \cap \left(\bigcap_{\lambda \in \weights_+(F_s)} \lambda^+\right) \cap \left(\bigcap_{\lambda \in \weights_-(F_s)} \lambda^-\right),
$$
so that $\psi(F_s)$ is the intersection of all half-spaces (bounded by a weight hyperplane) and all weight hyperplanes containing $F_s$.
\end{definition}
\begin{remark}
We could define $\psi$ more succinctly as
$$
\psi(F_s) = \bigcap_{\lambda \in \weights_0(F_s)} \lambda^+,
$$ but we will not need this definition.
\end{remark}

\begin{lemma} \label{lem:propertiesOfPsiForFaces}The correspondence $\psi$ is a well-defined function and satisfies the following properties.
\begin{enumerate}[(a)]
\item \label{assert:injPsi} $\psi$ is injective;
\item  \label{assert:weightPsi} $\weights_\epsilon(\psi(F_s)) = \psi(F_s)$ for all $\epsilon \in \{0,+,-\}$;
\item \label{assert:intersectPsi} $\psi(F_s) \cap \h_s = F_s$; and
\item \label{assert:minimalPsi} $\psi(F_s)$ is the smallest face of $\glngeo$ containing $F_s$;
\end{enumerate}
\smallskip
for all faces $F_s$ of $\slngeo$.
\end{lemma}
\begin{proof}
Suppose that $F$ is a face of $\slngeo$.  Recall that because $F$ is a face of $\slngeo$, \pref{eq:faceSignPartition} guarantees that the sets $\weights_0(F),\weights_+(F),\weights_-(F)$
form a partition of $\weights$. The definition of $\psi$ ensures that $\psi(F)$ is the subset of the Weyl chamber on which elements of $\weights_0(F)$ (resp.\ $\weights_+(F)$, $\weights_-(F)$) are zero (resp.\ non-negative, non-positive).  The fact $\weights =\weights_0(F)\sqcup\weights_+(F)\sqcup\weights_-(F)$ implies that each weight is assigned a unique sign, which ensures that $\psi(F)$ is a face of $\glngeo$.
Therefore, $\psi$ is well-defined.

A direct calculation from the definition of $\psi$ shows that Property \pref{assert:weightPsi} holds.  The definition of faces of $\slngeo$ and Property \pref{assert:weightPsi} imply that Property \pref{assert:intersectPsi} holds as well.  Property \pref{assert:injPsi} follows immediately from Property \pref{assert:intersectPsi}.

We now prove Property \pref{assert:minimalPsi}.  Suppose that $F_s$ is a face of $\slngeo$ and let $F = \psi(F_s)$.  Property \pref{assert:intersectPsi} implies that $F \supseteq F_s$.  Suppose that $G \supset F_s$ is a face of $\glngeo$ that contains $F_s$.  We see immediately that $\weights_\epsilon(G) \supseteq \weights_\epsilon(F_s)$ for $\epsilon \in \{+,-\}$ and $\weights_0(G) \subseteq \weights_0(F_s)$.  Property \pref{assert:weightPsi} yields that $\weights_\epsilon(G) \supseteq \weights_\epsilon(F)$ for $\epsilon \in \{+,-\}$ and $\weights_0(G) \subseteq \weights_0(F)$, so that $F$ is a sub-face of $G$.  In particular, we have $G \supseteq F,$ as desired.
\end{proof}

The following definition will help us characterize the extreme rays of elements of the image of $\psi$.

\begin{definition}
A subset of $\h$ is \emph{null} if all the extreme rays it contains have signature zero. It is said to be \emph{mono-signed} if the signatures of its extreme vectors are either all non-negative or all non-positive, and \emph{bi-signed} if it is not mono-signed.  Let $\mathfrak{F}_{\text{null}}(n,k)$ (resp.\ $\mathfrak{F}_{\text{bi}}(n,k)$) denote the set of null (resp.\ bi-signed) $k$-faces of $\glngeo$.  Let $\mathfrak{L}_{\text{null}}(n,k)$ (resp.\ $\mathfrak{L}_{\text{bi}}(n,k)$) denote the set of null (resp.\ bi-signed) $k$-flats of $\glngeo$.
\end{definition}

We are now ready to determine the image of $\psi$.

\begin{lemma}
\label{lem:imageOfPsiForFaces}
Let $F$ be a $k$-face of $\glngeo$.  Exactly one of the following holds.
\begin{enumerate}[(1)]
\item $F$ is mono-signed but neither null nor an element of the image of $\psi$.
\item $F$ is null and $\psi^{-1}(F)$ consists of a single $k$-face of $\slngeo$.
\item $F$ is bi-signed and $\psi^{-1}(F)$ consists of a single $(k-1)$-face of $\slngeo$.
\end{enumerate}
\end{lemma}

The idea of lemma is that condition (1) (resp.\ (2), (3)) indicates that $F$ lies on one side of $\h_s$ (resp.\ lies on $\h_s$, straddles both sides of $\h_s$).  The proof simply converts this intuitive dichotomy into a rigorous argument.

\begin{proof}
Suppose that $F$ is a $k$-face of $\glngeo$.  Denote by $E$ the set $\erays(F)$ of extreme rays of $F$. \pref{thm:chamberStructure}\pref{assert:chamberRaysPrelim} guarantees that $E$ non-negatively spans $F$.

The fact that all the hyperplanes in $\glngeo$ pass through the origin and the fact that $W$ is closed under multiplication by $\mathbb{R}_{\ge 0}$
ensure that $F$ is closed under multiplication by $\mathbb{R}_{\ge 0}$.  Because $\sigma$ is linear, it follows that $\sigma(F)$ must be closed under multiplication by $\mathbb{R}_{\ge 0}$.  This leaves very few possibilities for $\sigma(F)$, namely
$$
\label{eq:signaturesOfFace}
\sigma(F) \in \{\mathbb{R}_{\ge 0},\mathbb{R}_{\le 0},\{0\},\mathbb{R}\}.
$$

We divide into cases to account for all possibilities to prove the lemma.
\begin{itemize}
\item Case 1: $\sigma(F) = \mathbb{R}_{\ge 0}$ or $\sigma(F) = \mathbb{R}_{\le 0}$.  We prove that $F$ satisfies the conditions of (1).  Without loss of generality, assume that $\sigma(F^0) = \mathbb{R}_{\ge 0}.$ Because $F$ does not contain any element of negative signature, the set $E_-$ is empty.

Because $E$ non-negatively spans $F$ and all the extreme rays lie on one of the closed half-spaces determined by $\h_s$, the set $E_0 = E \cap \h_s$ non-negatively spans $F_s = F \cap \h_s$.  However, the non-negative span of $E_0$ is a sub-face of $F$ because $F$ is a cone over a simplex with extreme ray set $E$ by \pref{thm:chamberStructure}\pref{assert:chamberRaysPrelim}..  It follows that $F_s = F \cap \h_s$ is a sub-face of $F$ in $\glngeo$.  Because $F_s \subseteq \h_s$, we have $\sigma(F_s) = 0$.  In particular, the face $F_s$ is a strict sub-face of $F$ in $\glngeo$.

It follows from \pref{lem:propertiesOfPsiForFaces}\pref{assert:intersectPsi} that $\psi^{-1}(F) \subseteq \{F_s\}$.  However, the contrapositive of \pref{lem:propertiesOfPsiForFaces}\pref{assert:minimalPsi} implies that $F \not= \psi(F_s),$ because $F_s$ is a strict sub-face of $F$ (in $\glngeo$) that contains $F_s$.  Therefore, $F$ does not lie in the image of $\psi$, as claimed.

\item Case 2: $\sigma(F) = 0$.  We prove that $F$ satisfies the conditions of (2).  Because $\sigma(F) = 0$, we have $F \subseteq \h_s$.  In particular, every extreme ray of $F$ must have signature zero, so that $F$ is null. The fact that $F \subset \h_s$ implies that $F$ is a $k$-face of $\slngeo$.  \pref{lem:propertiesOfPsiForFaces}\pref{assert:weightPsi} ensures that $\psi(F) = F$ while \pref{lem:propertiesOfPsiForFaces}\pref{assert:injPsi} guarantees that $\psi^{-1}(F)$ consists of at most one element.  Therefore, we have $\psi^{-1}(F) = \{F\}$ with $F$ a $k$-face of $\slngeo$, as desired.

\item Case 3: $\sigma(F) = \mathbb{R}$.  We prove that $F$ satisfies the conditions of (3).  First, we prove that $F$ is bi-signed.  Because $E$ non-negatively spans $F$, the set $E$ must contain elements of both negative and positive signatures, which is precisely the fact that $F$ is bi-signed.

Let $p \in F^0 \cap \h_s$ be an interior point of $F$ with signature zero.  A direct calculation using the fact that $F^0$ intersects $\h_s$ shows that $F_s = F \cap \h_s$ is a $(k-1)$-face of $\slngeo$.  Note that $p \in F_s^0$.  \pref{eq:faceSignGeneralPoint} guarantees that
$$
\weights_\epsilon(F_s) = \weights_\epsilon(\{u\}) = \weights_\epsilon(F)
$$
for all $\epsilon \in \{0,+,-\}$, so that $F_s$ and $F$ have the same sign vector.  \pref{lem:propertiesOfPsiForFaces}\pref{assert:weightPsi} and the fact that a face is determined by its sign vector ensure that $\psi(F_s) = F.$  The fact that $\psi^{-1}(F) = \{F_s\}$ follows due to \pref{lem:propertiesOfPsiForFaces}\pref{assert:injPsi}.
\end{itemize}

\pref{eq:signaturesOfFace} ensures that the cases are exhaustive, which completes the proof of the lemma.
\end{proof}

We can simply apply \pref{thm:Characterization.Faces} to the statement of \pref{lem:imageOfPsiForFaces} into the language of the poset $\posetz_n$ in order to obtain the desired bijection.

\begin{definition}
We say that a subset of $\posetz_n$ is \emph{null} if it consists of elements of signature zero.  Let $\Chn_{\text{null}}(n,k)$ (resp.\ $\Chn_{\text{bi}}(n,k)$) denote the set null (resp.\ bi-signed) of $k$-chains in $\posetz_n$. Let $\Ens_{\text{null}}(n,k)$ (resp.\ $\Ens_{\text{bi}}(n,k)$) denote the set null (resp.\ bi-signed) $k$-ensembles in $\posetz_n$.
\end{definition}

\begin{proposition}
\label{prop:bijectForFaces}
The function $\psi$ satisfies following properties.
\begin{enumerate}[(a)]
\item \label{assert:psiIsoPosets} We have $\erays(\psi(F)) \subsetneq \erays(\psi(G))$ if and only if $F \subsetneq G$, where $F$ and $G$ are faces of $\slngeo$.
\item \label{assert:psiInteractWithFaces} If $F \in \mathfrak{F}_{\text{null}}(n,k) \cup \mathfrak{F}_{\text{bi}}(n,k+1)$, then $\psi(F \cap \h_s) = F$.
\item \label{assert:psiBijectionCombinatorial} The function $\indFn \circ \erays \circ \psi$ induces a bijection from the set of $k$-faces of $\slngeo$ to the set of chains $\Chn_{\text{null}}(n,k) \cup \Chn_{\text{bi}}(n,k+1)$.
\end{enumerate}
\end{proposition}
\begin{proof}
Suppose that $F \subsetneq G$ are faces of $\slngeo$.
\pref{lem:propertiesOfPsiForFaces}\pref{assert:minimalPsi} implies that $\psi(G)$ is a face of $\glngeo$ containing $G$, hence $F$.  Applying \pref{lem:propertiesOfPsiForFaces}\pref{assert:minimalPsi} yields that $\psi(G) \supseteq \psi(F)$.  Because $\psi$ is injective (as proven in \pref{lem:propertiesOfPsiForFaces}\pref{assert:injPsi}), it follows that $\psi(G) \supsetneq \psi(F)$.
Suppose on the other hand that $\psi(F) \subsetneq \psi(G)$.  Intersecting with $\h_s$ yields that $F \subseteq G$.  Because $\psi$ is well-defined (by \pref{lem:propertiesOfPsiForFaces}), it follows that $F \subsetneq G$.
We have proved that $\psi(F) \subsetneq \psi(G)$ if and only if $F \subsetneq G$.  \pref{thm:Characterization.Faces} therefore implies Property \pref{assert:psiIsoPosets}.

We now prove Property \pref{assert:psiBijectionCombinatorial}.
It follows from \pref{lem:imageOfPsiForFaces} that $\psi$ induces a bijection from the set of $k$-faces of $\slngeo$ to the $\mathfrak{F}_{\text{null}}(n,k) \cup \mathfrak{F}_{\text{bi}}(n,k+1)$. \pref{thm:Characterization.Faces} guarantees that $\indFn \circ \erays$ induces bijections from $\mathfrak{F}_{\text{null}}(n,k)$ to $\Chn_{\text{null}}(n,k)$ and from $\mathfrak{F}_{\text{bi}}(n,k+1)$ to $\Chn_{\text{bi}}(n,k+1)$, which implies Property \pref{assert:psiBijectionCombinatorial}.

It remains to prove Property \pref{assert:psiInteractWithFaces}.  We already saw that $\psi$ induces a bijection from the set of $k$-faces of $\slngeo$ to the $\mathfrak{F}_{\text{null}}(n,k) \cup \mathfrak{F}_{\text{bi}}(n,k+1)$.  In particular, there exists a unique face $F_s$ of $\slngeo$ such that $\psi(F_s) = F$.
\pref{lem:propertiesOfPsiForFaces}\pref{assert:intersectPsi} implies that $F_s = F \cap \h_s$, and Property \pref{assert:psiInteractWithFaces} follows.
\end{proof}

We immediately obtain the following corollary from \pref{prop:bijectForFaces}\pref{assert:psiBijectionCombinatorial}.

\begin{corollary}
\label{cor:numberOfSlnFacesBiject}
We have
$$
b(n,k) = \left|\Chn_{\text{null}}(n,k)\right| + \left|\Chn_{\text{bi}}(n,k+1)\right|.
$$
\end{corollary}

\subsection{Bijection for flats}
\label{sec:bijectForFlats}
In the previous subsection, we exploited the fact that faces of $\glngeo$ are simplicial in the proof of \pref{lem:imageOfPsiForFaces}.  The replacement for this condition in the case of flats is the following lemma.

\begin{lemma}
\label{lem:dim1flatIrreducibility}
For $i = 1,2$, let $L_i$ be a $k_i$-flat of $\slngeo$ or of $\glngeo$ (it is allowed that $L_1,L_2$ are flats of different incidence geometries).
If $L_1 \subsetneq L_2$ then $k_1 < k_2$.
\end{lemma}
\begin{proof}
The proof is elementary linear algebra.  Let $\Lambda_i$ denote the $\mathbb{R}$-span of $L_1$.  It follows from the definitions of flats of $\slngeo$ and $\glngeo$ that $L_i = \Lambda_i \cap W$ and that $k_i = \dim L_i$ (where $\dim$ denotes the dimension of a vector space over $\mathbb{R}$).  It follows that $\Lambda_1 \subsetneq \Lambda_2$ so that $\dim \Lambda_1 < \dim \Lambda_2$.  The lemma follows.
\end{proof}

We follow the structure of the previous subsection to establish an analogue of \pref{prop:bijectForFaces}\pref{assert:psiBijectionCombinatorial} for flats: we define a map $\tau$ from the set of flats of $\slngeo$ to the set of flats of $\glngeo$, prove that it is injective, and determine its image.

\begin{definition}
\label{def:etaForFlats}
Define a function $\tau: \mathfrak{L}_s \to \mathfrak{L}$ by
$$
\tau(L_s) = W \cap \left(\bigcap_{\lambda \in \mathcal{W_0}(L_s)} \lambda^\perp\right),
$$
so that $\tau(F_s)$ is the intersection of all the weight hyperplanes containing $L_s$.
\end{definition}

\begin{lemma} \label{lem:propertiesOfTauForFlats}The corrrespondence $\tau$ is a well-defined function and satisfies the following properties.
\begin{enumerate}[(a)]
\item $\tau$ is injective;
\item $\weights_0(\tau(L_s)) = \tau(L_s)$;
\item $\tau(L_s) \cap \h_s = L_s$; and
\item $\tau(L_s)$ is the smallest flat of $\glngeo$ containing $L_s$
\end{enumerate}
for all flats $L_s$ of $\slngeo$.
\end{lemma}
\begin{proof}
Suppose that $L$ is a flat of $\slngeo$.  The definition of $\tau$ ensures that $\tau(F)$ is the subset of the Weyl chamber on which elements of $\weights_0(F)$ vanish, so that $\tau(F)$ is a flat of $\glngeo$.  Therefore, $\tau$ is well-defined.

A direct calculation from the definition of $\tau$ shows that Property (b) holds.  The definition of flats of $\slngeo$ and Property (b) imply that Property (c) holds as well.  Property (a) follows immediately from Property (c).

We now prove Property (d).  Suppose that $L_s$ is a flat of $\slngeo$ and let $L = \psi(L_s)$.  Property (c) implies that $L \supseteq L_s$.  Suppose that $M \supset L_s$ is a flat of $\glngeo$ that contains $L_s$.  We see immediately that $\weights_0(M) \subseteq \weights_0(L_s)$.
Property (b) yields that $\weights_0(M) \subseteq \weights_0(L)$, so that $L$ is a subflat of $M$.  In particular, we have $M \supseteq L,$ as desired.
\end{proof}

The analysis of null flats (and hence of non-null mono-signed flats) is more complicated than the analysis of null faces because the non-negative span of a set of extreme rays of a flat is not necessarily a sub-flat.  In order to handle this difficulty, we need to use the explicit description of flats (\pref{thm:Characterization.Flats}).

\begin{lemma}
\label{lem:nullSetsAreEnsembles}
Suppose that $S$ is a set of extreme rays of $\glngeo$ of signature zero.  Then, the non-negative span of $S$ is a $|S|$-flat with extreme ray set $S$.
\end{lemma}
\begin{proof}
Recall that the extreme rays of signature zero are the elements $\ray{i,i}$ for $1 \le i \le \floor*{\frac{n}{2}}$.
The definition of an ensemble and the fact that an extreme ray of signature zero is uniquely determined by its level imply that $\ind{S}$ is a $|S|$-ensemble.  \pref{thm:Characterization.Flats} guarantees the existence of a unique flat $L$ in $\glngeo$ such that $\erays(L) = S$.  It follows from \pref{thm:chamberStructure}\pref{assert:chamberRaysPrelim} that $L$ is the non-negative span of $S$.
\end{proof}

The analysis of flats is more delicate than the analysis of faces partly due to the possibility that a bi-signed flat contains an element outside the image of $\tau$.  We need to introduce an additional condition in order to deal with this subtlety: the correct analogue of the condition of being bi-signed in the context of flats is the condition of being bi-signed and \emph{unbalanced}.

\begin{definition}
We say that a $k$-flat $L$ in $\glngeo$ is \emph{balanced} if $L$ has $k-1$ extreme rays of signature zero and \emph{unbalanced} otherwise.  We denote by $\mathfrak{L}_{\text{bi,unbal}}(n,k)$ the set of bi-signed unbalanced $k$-flats of $\glngeo$.
\end{definition}

We now prove a lemma to illustrate the geometric meaning of unbalancedness.

\begin{lemma}
\label{lem:geomMinimality}
Suppose that $L$ is an unbalanced $k$-flat in $\glngeo$ and that $L_s = L \cap \h_s$ is $k'$-flat in $\glngeo$.  Then, we have $k \not= k'+1$.
\end{lemma}
\begin{proof}
Let $E = \erays(L)$ denote the set of extreme rays of $L$.
The definition of $L_s$ ensures that $E_0 = \erays(L_s)$.   Because $L_s$ is assumed to be a flat in $\glngeo$, \pref{thm:chamberStructure}\pref{assert:chamberRaysPrelim} guarantees that $L_s$ is the non-nonegative span of $E_0$.  \pref{lem:nullSetsAreEnsembles} implies that $k' = |E_0|$.
The unbalancedness of $L$ ensures that $|E_0| \not= k-1$, and the lemma follows.
\end{proof}

We are now ready to determine the image of $\tau$.

\begin{lemma}
\label{lem:imageOfTauForFlats}
Let $L$ be a $k$-flat of $\glngeo$.  Exactly one of the following holds.
\begin{enumerate}[(1)]
\item $L$ is mono-signed but neither null nor an element of the image of $\tau$.
\item $L$ is null and $\tau^{-1}(L)$ consists of a single $k$-face of $\slngeo$.
\item $L$ is bi-signed and unbalanced and $\tau^{-1}(F)$ consists of a single $(k-1)$-face of $\slngeo$.
\item $L$ is bi-signed, balanced, and does not lie in the image of $\tau$.
\end{enumerate}
\end{lemma}

The statement of \pref{lem:imageOfTauForFlats} is more complicated than that of \pref{lem:imageOfPsiForFaces}
due to the additional sublety involving the minimality condition.  The intuition behind the lemma and the structure of the proof are similar to those of \pref{lem:imageOfPsiForFaces}.

\begin{proof}
We follow the proof of \pref{lem:imageOfPsiForFaces}.
Suppose that $L$ is a $k$-face of $\glngeo$.  Denote by $E$ the set $\erays(L)$ of extreme rays of $F$.  \pref{thm:chamberStructure}\pref{assert:chamberRaysPrelim} guarantees that $E$ non-negatively spans $L$.

The fact that $L$ is closed under multiplication by $\mathbb{R}_{\ge 0}$ ensures, as in the proof of \pref{lem:imageOfPsiForFaces}, that
\begin{equation}
\label{eq:signaturesOfFlats}
\sigma(L) \in \{\mathbb{R}_{\ge 0},\mathbb{R}_{\le 0},\{0\},\mathbb{R}\}.
\end{equation}

We divide into cases to account for all possibilities to prove the lemma.  
\begin{itemize}
\item Case 1: $\sigma(L) = \mathbb{R}_{\ge 0}$ or $\sigma(L) = \mathbb{R}_{\le 0}$.  We prove that $L$ satisfies the conditions of (1).  Without loss of generality, assume that $\sigma(L^0) = \mathbb{R}_{\ge 0}.$ Because $L$ does not contain any element of negative signature, the set $E_-$ is empty.

It follows from \pref{thm:chamberStructure}\pref{assert:chamberRaysPrelim} that $E_0$ non-negatively spans $L_s = L \cap \h_s$.  \pref{lem:nullSetsAreEnsembles} guarantees that $E_0$ non-negatively spans a flat of $\glngeo$.  It follows that $L_s = F \cap \h_s$ is a sub-flat of $F$ in $\glngeo$.  Because $L_s \subseteq \h_s$, we have $\sigma(L_s) = 0$.  In particular, the flat $L_s$ is a strict sub-flat of $L$ in $\glngeo$.

It follows from \pref{lem:propertiesOfTauForFlats}(c) that $\tau^{-1}(L) \subseteq \{L_s\}$.  However, the contrapositive of \pref{lem:propertiesOfTauForFlats}(d) implies that $L \not= \psi(L_s),$ because $L_s$ is a strict sub-flat of $L$ (in $\glngeo$) that contains $L_s$.  Therefore, $L$ does not lie in the image of $\tau$, as claimed.

\item Case 2: $\sigma(L) = 0$.  We prove that $L$ satisfies the conditions of (2).  Because $\sigma(L) = 0$, we have $L \subseteq \h_s$.  In particular, every extreme ray of $L$ must have signature zero, so that $L$ is null. The fact that $L \subset \h_s$ implies that $L$ is a $k$-flat of $\slngeo$.  \pref{lem:propertiesOfTauForFlats}(b) ensures that $\tau(L) = L$ while \pref{lem:propertiesOfTauForFlats}(a) guarantees that $\psi^{-1}(L)$ consists of at most one element.  Therefore, we have $\psi^{-1}(L) = \{L\}$ with $L$ a $k$-flat of $\slngeo$, as desired.

\item Case 3: $\sigma(L) = \mathbb{R}$.  We prove that $L$ satisfies the conditions of exactly one of (3), (4).  First, we prove that $F$ is bi-signed.  Because $E$ non-negatively spans $L$, the set $E$ must contain elements of both negative and positive signatures, which is precisely the fact that $L$ is bi-signed.  The definition of a flat of $\slngeo$ ensures that $L_s = L \cap \h_s$ is a flat of $\slngeo$.  It follows from the fact that $\sigma(L) = \mathbb{R}$ that $L_s$ has codimension 1 in $L$, so that $L_s$ is actually a $(k-1)$-flat of $\slngeo$.  We need to divide into cases based on whether $L_s$ is a flat for $\glngeo$.
\begin{itemize}
\item Subcase 3.1: $L_s$ is not a flat of $\glngeo$.  We prove that $L$ satisfies the conditions of (3).  Let $L'_s$ denote the non-negative span of $E_0$, which is a $|E_0|$-flat in $\glngeo$ by \pref{lem:nullSetsAreEnsembles}.  Note that $L'_s \subseteq L \cap \h_s = L_s$.  Because $L_s$ is not a flat in $\glngeo$, we obtain that $L'_s \subsetneq L_s$.  Because $L_s$ is a $(k-1)$-flat in $\slngeo$, \pref{lem:dim1flatIrreducibility} guarantees that $|E_0| < k-1$, which implies that $L$ is unbalanced.

Suppose that $\tau(L_s)$ is a $k'$-flat.  \pref{lem:propertiesOfTauForFlats}(d) guarantees that that $\tau(L_s)$ is the smallest flat of $\glngeo$ that contains $L_s$, so that $k' \le k$.
Because $L_s$ is not a flat of $\glngeo$, we know that $\tau(L_s) \supsetneq L_s$.  \pref{lem:dim1flatIrreducibility} guarantees that $k-1 < k'$ because $L_s$ is a $(k-1)$-flat of $\slngeo$.
It follows that $k = k'$.  The contrapositive of \pref{lem:dim1flatIrreducibility} and the fact that $\tau(L_s) \subseteq L$ imply that $\tau(L_s) = L$.  \pref{lem:propertiesOfTauForFlats}(a) then ensures that $\tau^{-1}(L) = \{L_s\}$, as desired.
\item Subcase 3.2: $L_s$ is a flat of $\glngeo$.  We prove that $L$ satisfies the conditions of (4).  Note that $L_s$ is a $(k-1)$-flat of $\glngeo$ because it is a $(k-1)$-flat of $\slngeo$ that is a flat of $\glngeo$.  The contrapositive of \pref{lem:geomMinimality} implies that $L$ is balanced.

\pref{lem:propertiesOfTauForFlats}(c) guarantees that $\tau^{-1}(L) \subseteq \{L_s\}.$  Because $L_s$ is a strict sub-flat of $L$ (in $\glngeo$) that contains $L_s$, \pref{lem:propertiesOfTauForFlats}(d) implies that $\tau(L_s) \not= L_s$.  Hence, we obtain that $L$ is not in the image of $L_s$, which completes the proof that $L$ satisfies the conditions of (4).
\end{itemize}
Subcases 3.1 and 3.2 clearly exhaust all the possibilities encompassed by Case 3.
\end{itemize}

\pref{eq:signaturesOfFlats} ensures that the cases are exhaustive, which completes the proof of the lemma.
\end{proof}

We can simply apply \pref{thm:Characterization.Flats} to the statement of \pref{lem:imageOfTauForFlats} into the language of the poset $\posetz_n$ in order to obtain the desired bijection.

\begin{definition}
We say that a $k$-ensemble $E$ in $\posetz_n$ is \emph{balanced} if $E$ has $k-1$ elements of signature zero and \emph{unbalanced} otherwise.  We denote by $\Ens_{\text{bi,unbal}}(n,k)$ the set of bi-signed unbalanced $k$-ensembles in $\posetz_n$.
\end{definition}

\begin{proposition}
\label{prop:bijectForFlats}
The function $\indFn \circ \erays \circ \tau$ induces a bijection from the set of $k$-flats of $\slngeo$ to the set of ensembles $\Ens_{\text{null}}(n,k) \cup \Ens_{\text{bi,unbal}}(n,k+1)$.
\end{proposition}
\begin{proof}
It follows from \pref{lem:imageOfTauForFlats} that the function $\tau$ induces a bijection from the set of $k$-flats of $\slngeo$ to $\mathfrak{L}_{\text{null}}(n,k) \cup \mathfrak{L}_{\text{bi,unbal}}(n,k+1)$.  \pref{thm:Characterization.Flats} guarantees that $\indFn \circ \erays$ induces bijections from $\mathfrak{L}_{\text{null}}(n,k)$ to $\Ens_{\text{null}}(n,k)$ and from $\mathfrak{L}_{\text{bi,unbal}}(n,k+1)$ to $\Ens_{\text{bi}}(n,k+1)$.  The proposition follows.
\end{proof}

We immediately obtain the following corollary.

\begin{corollary}
\label{cor:numberOfSlnFlatsBiject}
We have
$$
f(n,k) = \left|\Ens_{\text{null}}(n,k)\right| + \left|\Ens_{\text{bi,unbal}}(n,k+1)\right|.
$$
\end{corollary}

\appendix\addtocontents{toc}{\protect\setcounter{tocdepth}{1}}

\section{Proofs of Theorems~\ref{thm:slnFaces} and~\ref{thm:slnFlats}}
\label{app:proofsOfSlnFacesAndFlats}

In Appendices~\ref{app:proofOfSlnFaces} and~\ref{app:proofOfSlnFlats}, we complete the enumeration of faces and flats of $\slngeo$ modulo two technical results, Propositions~\ref{prop:A0} and~\ref{prop:C0}, whose proofs are deferred to \pref{app:functionalEqSolveA0C0}.

\subsection{Proof of \pref{thm:slnFaces}}
\label{app:proofOfSlnFaces}
In order to apply \pref{cor:numberOfSlnFacesBiject} to prove \pref{thm:slnFaces}, we need to enumerate null chains and bi-signed chains.  First, we enumerate null chains.

We denote by $d(n,k)$ the number of null $k$-chains in $\posetz_n$ and by
$$
D=D(x,y)=\sum_{n=0}^\infty \sum_{k=0}^\infty d(n,k)x^ny^k
$$
the corresponding generating function.

\begin{proposition}
\label{prop:nullCount}
For all $n,k \ge 0$,  the number of null $k$-chains in $\posetz_n$ is
$$
d(n,k) = \binom{\floor*{\frac{n}{2}}}{k},
$$
and therefore the corresponding generating function is
$$
D = \frac{1+x}{1-x^2(1+y)}.
$$
Furthermore, $d(n,k)$ counts the number of null ensembles.
\end{proposition}
\begin{proof}
Because the set $N$ of elements of signature zero in $\posetz_n$ is a chain, any subset of $N$ is a chain.  Any null chain is by definition a subset of $N$.  The  expression for $d(n,k)$  follows because $\left|N\right| = \floor*{\frac{n}{2}}$.

The expression for the generating function $D$ can be obtained from the formula for $d(n,k)$ by dividing into cases based on the parity of $n$: 
$$
d(2m,k) = d(2m+1,k) = \binom{m}{k}
$$
for all $m \ge 0$. Using the identity 
$$
\sum_{m=0}^\infty \sum_{k=0}^\infty \binom{m}{k} x^m y^k =\frac{1}{1-x(1+y)},
$$
we obtain that
$$
D(x,y)=\sum_{m=0}^\infty \sum_{k=0}^\infty \binom{m}{k} (x^2)^m y^k + x \sum_{m=0}^\infty \sum_{k=0}^\infty \binom{m}{k} (x^2)^m y^k 
 =\frac{1+x}{1-x^2(1+y)}.
$$

Because every null chain is an ensemble and every ensemble a chain, $d(n,k)$ also counts null $k$-ensembles in $\posetz_n$.
\end{proof}

It remains to enumerate bi-signed chains, which we do indirectly by enumerating mono-signed chains and using the enumerations of null chains and all chains.  First, we recall the enumeration of all chains from \cite{gl}.  Define $g(n,k)$ to be the number of $k$-chains in $\posetz_n$, and denote by
$$
G(x,y) = \sum_{n=0}^\infty \sum_{k=0}^\infty g(n,k)x^ny^k
$$
the corresponding generating function.
In light of \pref{thm:Characterization.Faces}, the following result was proven in \cite{gl}.

\begin{theorem}[\pref{gln-thm:faceCount}]
\label{thm:glnFaces}
The generating function for the number of $k$-chains in $\posetz_n$ is
$$G(x,y) = \frac{1-x}{1-2x+x^2-2xy+x^2y}.$$
\end{theorem}

To enumerate mono-signed chains, it suffices to enumerate chains of elements of non-positive signature.  In order to write a recursion to count such chains, we need to generalize the problem slightly.

We denote by $\Chn'(n,k,\sigma)$ the set of $k$-chains in $\posetz_n$ all of whose elements have signature at most $\sigma$ and by
$$
\Chn(n,k,\sigma) = \begin{cases}
\Chn'(n,k,\sigma) & \text{if } \sigma \ge 0\\
\varnothing & \text{otherwise}
\end{cases}
$$
the restriction of $\Chn'(n,k,\sigma)$ to $\sigma \ge 0$.
The size of $\Chn(n,k,\sigma)$ is denoted by
$$
a(n,k,\sigma) = \left|\Chn(n,k,\sigma)\right|
$$
In particular, $a(n,k,0)$ counts the number of $k$-chains in $\posetz_n$ consisting of elements of non-positive signature.

As usual, we pass to generating functions to perform the enumeration.
For $w \in \mathbb{Z}$, define
$$
A_\sigma  = A_\sigma (x,y) = \sum_{n=0}^\infty \sum_{k=0}^\infty a(n,k,\sigma) x^ny^k
$$
and let
$$
A = A(x,y,z) = \sum_{\sigma=0}^\infty A_\sigma (x,y) z^\sigma
$$
It is clear that $A_\sigma  \in \mathbb{Z}\llbracket x,y,z\rrbracket$ and $A \in \mathbb{Z}\llbracket x,y,z\rrbracket$ so that \emph{a fortiori} $A_\sigma \in R$ and $A \in R\llbracket z\rrbracket$.
The enumeration of chains consisting solely of elements of non-positive signature is achieved by the following result.

\begin{proposition}
\label{prop:A0}
The generating function  for the number of $k$-chains in $\posetz_n$ consisting of elements of non-positive signature is
$$
A_0 = \frac{2}{1 - 2 x + x^2 - 2 x y + 
 x^2 y + \sqrt{\alpha}}
$$
in $R$, where $\alpha \in R$ is the polynomial defined in \pref{thm:slnFaces}.
\end{proposition}

The proof of \pref{prop:A0} is technical and is therefore deferred to \pref{app:functionalEqSolveA0C0}.

\begin{proof}[Proof of \pref{thm:slnFaces} assuming \pref{prop:A0}]
Note that $a(n,k,0)$ counts the number of $k$-chains in $\posetz_n$ all of whose elements have non-negative (resp.\ non-positive signature).  A chain is null precisely if all of its elements simultaneously have non-negative and non-positive signature.  Therefore, there are
$$
2a(n,k,0)-d(n,k)
$$
monosigned $k$-chains in $\posetz_n$.
It follows that there are
$$
g(n,k)-2a(n,k,0)+d(n,k)
$$
bi-signed $k$-chains in $\posetz_n$.
\pref{cor:numberOfSlnFacesBiject} implies that
$$b(n,k) = g(n,k+1)-2a(n,k+1,0)+d(n,k+1)+d(n,k).$$  Passing to generating functions yields the identity
$$
B = \frac{G-2A_0+D}{y}+D.
$$
After some algebra, we obtain the theorem from
the formulae for the series $G,A_0,D$ given in \pref{thm:glnFaces} and Propositions~\ref{prop:A0} and~\ref{prop:nullCount}, respectively.
\end{proof}

\subsection{Proof of \pref{thm:slnFlats}}
\label{app:proofOfSlnFlats}

In order to apply \pref{cor:numberOfSlnFlatsBiject} to prove \pref{thm:slnFlats}, we need to enumerate null ensembles and bi-signed unbalanced ensembles.  We have already enumerated the null ensembles in \pref{prop:nullCount}.
In order to enumerate bi-signed unbalanced ensembles, we instead count bi-signed balanced ensembles and all bi-signed ensembles.  

First, we count bi-signed balanced ensembles. We denote by $v(n,k)$ the number of bi-signed balanced $k$-ensembles in $\posetz_n$ and by
$$
V=V(x,y)=\sum_{n=0}^\infty \sum_{k=0}^\infty v(n,k)x^ny^k
$$
the corresponding generating function.

\begin{proposition}
\label{prop:balCount}
For all $n,k \ge 0$, the number of bi-signed balanced ensembles is
\begin{equation}
v(n,k) = \left(k-1+\chi_{2\mathbb{Z}+1}(n)\right)\binom{\floor*{ \frac{n}{2}} - 1}{k-2} + \delta_{n,1}\delta_{k,1}
\end{equation}
and therefore the corresponding generating function is
\begin{equation}
V= \frac{x y \left(1-x^2\right) \left(1+x y-x^2\right)}{\left(1-x^2 (y+1)\right)^2}.
\end{equation}
\end{proposition}
In order to prove \pref{prop:balCount}, we exploit grading of $\posetz_n$ by level. Given $S \subset \posetz_n$, write
$$
S_{(k)} = \{p \in S \mid \ell(p)=k\}
$$
for the set of elements of $S$ lying in level $k$.  The following lemma relates $S_{(k)}$ to $S_{(k+1)}$ when $S$ is an ensemble.

\begin{lemma}
\label{lem:consecLevels}
Let $S$ be an ensemble in $\posetz_n$ and suppose that $1 \le k \le n-1$.  We have
$$
-1 \le |S_{(k+1)}|-|S_{(k)}| \le 1.
$$
\end{lemma}
\begin{proof}
The statement is clearly true if $S$ is an interval.  The corresponding assertion for ensembles follows because every ensemble can be (canonically) expressed as disjoint union of non-empty intervals $I_1,\ldots,I_m$ such that
$$
\max_{p \in I_j} \ell(p) + 1 < \min_{p \in I_{j+1}} \ell(p)
$$
for all $1 \le j \le m-1$.
\end{proof}

\begin{proof}[Proof of \pref{prop:balCount}]
One can easily verify the proposition for $n = 1,2$.  Assume henceforth that $n \ge 3$.
Denote by $\Ens_{\text{bi,bal}}(n,k)$ the set of bi-signed balanced $k$-ensembles in $\posetz_n$.
Recall that an $k$-ensemble $S$ is balanced if and only if the $k = |S_0| + 1$, and recall that by definition $k = \left|\ell(S)\right|$.
We see immediately $\left|\ell(S_0)\right| = |S_0|$ because a null extreme ray is determined by its level.
It follows that $\left|\ell(S) \setminus \ell(S_0)\right| = 1$ for all balanced ensembles $S$.
Define a function $\mu: \Ens_{\text{bi,bal}}(n,k) \to [1,n] \cap \mathbb{Z}$ by defining $\mu(S)$ to be the unique element of $\ell(S) \setminus \ell(S_0)$.  The preceding remarks guarantee that $\mu$ is a function.

For $S \subset \posetz_n$, let
$$
\nu(S) = \left(\ell(S_+) \cup \ell(S_-)\right) \cap \ell(S_0).
$$
We use the functions $\mu$ and $\nu$ to understand the structure of bi-signed unbalanced ensembles.

Suppose that $S$ is a bi-signed balanced ensemble in $\posetz_n$.
We claim that exactly one of the following holds:
\begin{itemize}
\item $\mu(S) = 2t+1$ is odd with $0 \le t \le \floor*{\frac{n-1}{2}} - \chi_{2\mathbb{Z}}(n)$, $\nu(S) = \varnothing,$ and
\begin{align*}
S_+ &= \{(t+1,t)\}\\
S_- &= \{(t,t+1)\}\\
S_0 &\supseteq \left(\{(t,t),(t+1,t+1)\} \cap \left[1,\frac{n}{2}\right]^2\right).
\end{align*}
In this case, we say that $S$ is of type I.
\item $n=2t$ is even, $\mu(S) = n-1$, $\nu(S) = \{n\}$, and
\begin{align*}
S_+ &= \{(t,t-1),(t+1,t-1)\}\\
S_- &= \{(t-1,t),(t-1,t+1)\}\\
S_0 &\supseteq \{(t-1,t-1),(t,t)\}.
\end{align*}
In this case, we say that $S$ is of type II.
\end{itemize}
We see immediately that $S$ cannot be simultaneously of type I and of type II. We prove that $S$ must be of type I or type II by dividing into cases based on whether $\nu(S)$ is empty.
\begin{itemize}
\item Case I: $\nu(S) = \varnothing$.  Suppose that $\mu(S) = j$.  Because
$$
\ell(S_-) \cup \ell(S_+) = \nu(S) \cup (\ell(S_-) \cup \ell(S_+) \setminus \ell(S_0)) = \{j\}
$$
and $S$ is bi-signed, $\ell(S_-)=\ell(S_+)=\{k\}$.
In particular, we have $|S_{(j)}| \ge 2$ and $|S_{(m)}| \le 1$ for all $m \not= j$ with equality only if $m$ is even.

If $j \ge 2$, then \pref{lem:consecLevels} implies that $1 \ge |S_{(j-1)}| \ge |S_{(j)}| - 1 \ge 2,$ so that $|S_{(j)} = 2$, $j-1$ is even, and $\left(\frac{j-1}{2},\frac{j-1}{2}\right) \in S$.  Similarly, if $j \le n-1$, it follows from \pref{lem:consecLevels} that $|S_{(j)}| = 2$, $j+1$ is even, and $\left(\frac{j+1}{2},\frac{j+1}{2}\right) \in S$.  Because $n \ge 3$, we must have $j \ge 2$ or $j \le n-1$, so that $j$ is odd and $|S_{(j)}| = 2$ unconditionally.  In particular, we have $|S_+| = |S_-| = 1$.  Define $j = 2t+1$.

It follows from the definition of an ensemble implies that the set of $x$-coordinates of $S_{(j)}$ forms an interval in $\mathbb{Z}$.  Because $S_{(j)}$ contains one element of positive signature and one element of negative signature, we obtain that $|S_{(j)}| = \{(t+1,t),(t,t+1)\},$ as claimed.

It only remains to prove that $0 \le t \le \floor*{ \frac{n-1}{2} }-\chi_{2\mathbb{Z}}(n).$ The fact that $1 \le j \le n$ implies that $0 \le t \le \floor*{ \frac{n-1}{2} }.$  Therefore, we may assume that $n$ is even and need only prove that $j < n-1$.  Because $j$ is odd and $n$ even, we know that $j \le n-1$.  Assume for sake of deriving a contradiction that $j = n-1$.  Because $\{(t,t),(t+1,t),(t,t+1),(t+1,t+1)\} \subset S$ and $2t+2 = n$, the definition of an ensemble implies that $(t+2,t) \in S$ (as no upper end-point of a constituent interval of a flat in $\posetz_n$ can have level $n$).  This contradicts the fact that $|S_+| = 1$.

\item Case II: $\nu(S) \not= \varnothing$.  Let $u \in \nu(S)$ be arbitrary.  In particular, $u \in \ell(S_0)$ so that $u$ is even and $u > 1$.  The definition of $g$ guarantees that $|S_{(u)}| \ge 2$, hence $|S_{(u-1)}| \ge 1$ by \pref{lem:consecLevels}.  The fact that $u-1$ is odd implies that $u-1 \notin \ell(S_0)$.  It follows that $\mu(S) = u-1$.

The fact that $u+1$ is odd implies that $u+1 \notin \ell(S_0)$, so that $S_{(u+1)} = 0$.  The contrapositive of \pref{lem:consecLevels} guarantees that $u = n$.  Since $u$ was an arbitrary element of $\nu(S)$, we obtain that $\nu(S) = \{u\}$.  It follows that $n = u \in \ell(S_0) \subset 2\mathbb{Z}$, so that $S$ is of type II.
\end{itemize}

We are now ready to determine the fibers $\mu^{-1}(e)$ by casework on $e$.
\begin{itemize}
\item Case 1: $e = 2t+1$ is odd and $e \in [1,n-2]$.  In this case, $\mu^{-1}(e)$ consists solely of elements of type I.  In particular, every element of $E$ is of the form
$$S^E_T = \{(t,t),(t+1,t),(t,t+1),(t+1,t+1)\} \cup \{(v,v) \mid v \in T\}$$
for a unique $T \subset \{1,\ldots,t-1\} \cup \left\{t+1,\ldots,\floor*{ \frac{n}{2}}\right\}$ with $|T|=k-3$.  We see immediately that $S^1_T$ is a bi-signed balanced $k$-ensemble in $\posetz_n$ with $\mu(S^E_T) = e$ for such sets $T$.  It follows that $|\mu^{-1}(e)| = \binom{\floor*{ \frac{n}{2}} - 2}{k-3}$.
\item Case 2: $e = 1$.  In this case, $\mu^{-1}(e)$ consists solely of elements of type I.  In particular, every element of $E$ is of the form
$$S^1_T = \{(1,0),(0,1),(1,1)\} \cup \{(v,v) \mid v \in T\}$$
for a unique $T \subset \left\{2,\ldots,\floor*{ \frac{n}{2}}\right\}$ with $|T|=k-2$.  We see immediately that $S^1_T$ is a bi-signed balanced $k$-ensemble in $\posetz_n$ with $\mu(S^1_T) = 1$ for such sets $T$.  It follows that $|\mu^{-1}(e)| = \binom{\floor*{ \frac{n}{2}} - 1}{k-2}$.
\item Case 3: $e = n$ and $n=2u+1$ is odd.  In this case, $\mu^{-1}(e)$ consists solely of elements of type I.  In particular, every element of $E$ is of the form
$$S^n_T = \{(u,u),(u+1,u),(u,u+1)\} \cup \{(v,v) \mid v \in T\}$$
for a unique $T \subset \{1,\ldots,u\}$ with $|T|=k-2$.  A direct calculation shows that $S^n_T$ is a bi-signed balanced $k$-ensemble in $\posetz_n$ with $\mu(S^n_T) = 1$ for such sets $T$.  It follows that $|\mu^{-1}(e)| = \binom{\floor*{ \frac{n}{2}} - 1}{k-2}$.
\item Case 4: $e = n-1$ and $n=2u$ is even.  In this case, $\mu^{-1}(e)$ consists solely of elements of type II.  In particular, every element of $E$ is of the form
\begin{multline*}S^n_T = \{(u-1,u-1),(u,u-1),(u-1,u),(u+1,u-1),(u,u),(u-1,u+1)\}\\
\cup \{(v,v) \mid v \in T\}\end{multline*}
for a unique $T \subset \{1,\ldots,u-2\}$ with $|T|=k-3$. A direct calculation shows that $S^n_T$ is a bi-signed balanced $k$-ensemble in $\posetz_n$ with $\mu(S^n_T) = 1$ for such sets $T$.  It follows that $|\mu^{-1}(e)| = \binom{\floor*{ \frac{n}{2}} - 2}{k-3}$.
\item Case 5: $e$ is even.  The classification of bi-signed balanced ensembles by type implies that $\mu^{-1}(e) = \varnothing$.
\end{itemize}
Summing the results of the casework, we obtain that
\begin{align*}
\left|\Ens_{\text{bi,bal}}(n,k)\right| &= \left(\floor*{ \frac{n}{2}} - 1\right)\binom{\floor*{ \frac{n}{2}} - 2}{k-3} + \left(1+\chi_{2\mathbb{Z}+1}(n)\right)\binom{\floor*{ \frac{n}{2}} - 1}{k-2}\\
&=(k-2)\binom{\floor*{ \frac{n}{2}} - 1}{k-2} + \left(1+\chi_{2\mathbb{Z}+1}(n)\right)\binom{\floor*{ \frac{n}{2}} - 1}{k-2}\\
&=\left(k-1+\chi_{2\mathbb{Z}+1}(n)\right)\binom{\floor*{ \frac{n}{2}} - 1}{k-2}
\end{align*}
for $n \ge 3,$ as desired.  The dependency on the parity of $n$ comes from Case 3.  We have to introduce the correction term $\delta_{n,1}\delta_{k,1}$ in order to deal with the case of $n \le 2$.

The formula for the generating function $V$ can be obtained from the recurrence relation by dividing into cases based on the parity of $n$.  Indeed, note that
\begin{align}
\label{eq:balFormEven}
v(2m,k) &= (k-1)\binom{m-1}{k-2} & \text{for } m \ge 1\\
\label{eq:balFormOddBig}
v(2m+1,k) &= k\binom{m-1}{k-2} & \text{for } m \ge 1\\
\label{eq:balForm1}
v(1,k) &= \delta_{k,1}.
\end{align}
The formula for $V$ follows from \pref{eq:balFormEven}, \pref{eq:balFormOddBig}, and \pref{eq:balForm1} after some algebra.
\end{proof}

It remains to enumerate bi-signed ensembles, which we do indirectly by enumerating mono-signed ensembles and using the enumerations of null ensembles  and all ensembles.
First, we recall the enumeration of all ensembles from \cite{gl}.  Define $h(n,k)$ to be the number of $k$-ensembles in $\posetz_n$, and denote by
$$
H(x,y) = \sum_{n=0}^\infty \sum_{k=0}^\infty h(n,k)x^ny^k
$$
the corresponding generating function.
In light of \pref{thm:Characterization.Flats}, the following result was proven in \cite{gl}.
\begin{theorem}[\pref{gln-thm:flatCount}]
\label{thm:glnFlats}
The generating function for the number of $k$-ensemble in $\posetz_n$ is
$$
H(x,y) = \frac{1 - x - x y + 2 x^2 y - x^3 y}{1 - 2 x + x^2 - 2 x y + 3 x^2 y - 2 x^3 y + x^2 y^2 - 2 x^3 y^2 + x^4 y^2}
$$
in $R$.
\end{theorem}

To enumerate mono-signed ensembles, it suffices to enumerate ensembles all of whose elements have non-positive signature.  In order to write a recursion to count such chains, we need to generalize the problem slightly.

We denote by $\Ens'(n,k,\sigma)$  the set of $k$-ensembles in $\posetz_n$ all of whose elements have signature at most $\sigma$ and by
$$
\Ens(n,k,\sigma) = \begin{cases}
\Ens'(n,k,\sigma) & \text{if } \sigma \ge 0\\
\varnothing & \text{otherwise}.
\end{cases}
$$
the restriction of $\Ens'(n,k,\sigma)$ to $\sigma \ge 0$.
The size of $\Ens(n,k,\sigma)$ is denoted by
$$
c(n,k,\sigma) = \left|\Ens(n,k,\sigma)\right|
$$
In particular, there are $c(n,k,0)$ ensembles in $\posetz_n$ consisting of elements of non-positive signature.

As usual, we pass to generating functions to perform the enumeration.
For $w \in \mathbb{Z}$, define
$$
C_\sigma  = C_\sigma (x,y) = \sum_{n=0}^\infty \sum_{k=0}^\infty c(n,k,\sigma) x^ny^k
$$
and let
$$
C = C(x,y,z) = \sum_{\sigma=0}^\infty C_\sigma (x,y) z^\sigma
$$
It is clear that $C_\sigma  \in \mathbb{Z}\llbracket x,y,z\rrbracket$ and $C \in \mathbb{Z}\llbracket x,y,z\rrbracket$ so that \emph{a fortiori} $C_\sigma \in R$ and $C \in R\llbracket z\rrbracket$.
The enumeration of chains of elements of non-positive signature is achieved by the following result.

\begin{proposition}
\label{prop:C0}
The generating function for the number $c(n,k,0)$ of $k$-ensembles in $\posetz_n$ consisting of elements of non-positive signature is
$$ 
C_0 = \frac{4 \left(1 + x^2 - x y - x^3 y + x^4 y + x^2 y^2 + x^3 y^2 - x^3 y^3 + 
 x^4 y^3 - x^5 y^3 + x^6 y^3\right)}{\left(\beta+x\sqrt{\gamma}+\sqrt{\zeta+\eta\sqrt{\gamma}}\right)\left(\beta-x\sqrt{\gamma}+\sqrt{\zeta-\eta\sqrt{\gamma}}\right)},
$$
in $R$, where $\beta,\gamma,\zeta,\eta \in R$ are the polynomials defined in \pref{thm:slnFlats}.
\end{proposition}
The proof of \pref{prop:C0} is technical and therefore deferred to  \pref{app:functionalEqSolveA0C0}.

\begin{proof}[Proof of \pref{thm:slnFlats} assuming \pref{prop:C0}]
Note that $c(n,k,0)$ counts the number of $k$-ensembles in $\posetz_n$ all of whose elements have non-negative (resp.\ non-positive signature).  An ensemble is null precisely if all of its elements simultaneously have non-negative and non-positive signature.  \pref{prop:nullCount} ensures that there are
$$
2a(n,k,0)-d(n,k)
$$
monosigned $k$-ensembles in $\posetz_n$.
It follows that there are
$$
h(n,k)-2c(n,k,0)+d(n,k)-v(n,k)
$$
bi-signed unbalanced $k$-ensembles in $\posetz_n$.
\pref{cor:numberOfSlnFlatsBiject} implies that
$$
f(n,k) = h(n,k+1)-2c(n,k+1,0)+d(n,k+1)-v(n,k+1)+d(n,k).$$ Passing to generating functions yields the identity
$$
F = \frac{H-2C_0+D-V}{y}+D.
$$
We obtain the theorem from
the formulae for the series $H,C_0,D,V$ given in \pref{thm:glnFlats} and Propositions~\ref{prop:C0}, \ref{prop:nullCount} and~\ref{prop:balCount}, respectively.
\end{proof}

\section{Extreme rays and simplicial chambers of \slngeoForTitle}
\label{app:extremeRaysAndConseqs}

In \pref{app:classifyExtremeRays}, we prove \pref{thm:extremeRaysOfSln}, and in \pref{app:extremeRaysOfFacesAndChambers}, we prove the last three parts of \pref{thm:chamberRays} and analogues for faces.
In \pref{app:simplicialChambers}, we prove \pref{thm:chamberRays}\pref{assert:countSimplicialChambers}.

\subsection{Classification of extreme rays in \slngeoForTitle}
\label{app:classifyExtremeRays}

Most of the work in proving \pref{thm:extremeRaysOfSln} has already been done: we deduce \pref{thm:extremeRaysOfSln} by specializing \pref{prop:bijectForFaces} to the case of 1-faces.

Work in the notation of \pref{sec:bijectForFaces}.  Because extreme rays are 1-faces, we will need to deal with the set of chains $\Chn_{\text{null}}(n,1) \cup \Chn_{\text{bi}}(n,2)$ in order to apply \pref{prop:bijectForFaces}\pref{assert:psiBijectionCombinatorial}.
Note that $\Chn_{\text{null}}(n,1)$ contains exactly the chains that consist of one element of $\posetz_n$ of signature 0, while $\Chn_{\text{bi}}(n,2)$ contains exactly the chains that consist of two comparable elements of $\posetz_n$ that have signatures of opposite signs.
In particular, the maps $\omega: \mathfrak{N} \to \Chn_{\text{null}}(n,1)$ and $\omega: \mathfrak{B} \to \Chn_{\text{bi}}(n,2)$ defined as
\begin{align*}
\omega(p) &= \{p\}\\
\omega(p,q) &= \{p,q\}
\end{align*}
are bijective.  Therefore, $\omega$ induces a bijection from $\mathfrak{N} \cup \mathfrak{B}$ to $\Chn_{\text{null}}(n,1) \cup \Chn_{\text{bi}}(n,2)$.
Note that $u \in C$ if and only if $\omega{u} \subseteq C$ for all $C \subseteq \posetz_n$ and $u \in \mathfrak{N} \cup \mathfrak{B}$.

We will use the following lemma repeatedly when discussing extreme rays of $\slngeo$ in order to relate the function $\rayFn : \mathfrak{N} \cup \mathfrak{B} \to \h_s$ to the function $\psi$.

\begin{lemma}
\label{lem:explainDefOfRayForSln}
For all $u \in \mathfrak{N} \cup \mathfrak{B},$ the ray $\ray{u}$ is an extreme ray of $\slngeo$ and satisfies $\ind{\erays(\psi(\ray{u}))} = \omega(u)$.
\end{lemma}
\begin{proof}
Recall that extreme rays of an incidence geometry are 1-faces.

First, suppose that $u \in \mathfrak{N}$.
\pref{thm:chamberStructure}\pref{assert:chamberRaysPrelim} guarantees that $\ray{u}$ is an extreme ray, hence of 1-face, of $\glngeo$.  Because $\sigma(u) = 0$, we see immediately that $\sigma(\ray{u}) = 0$ so that $\ray{u}$ lies in $\h_s$.  In particular, it follows that $\ray{u}$ defines a 1-face, hence an extreme ray, of $\slngeo$.
Because $\ray{u}$ is null, \pref{prop:bijectForFaces}\pref{assert:psiInteractWithFaces} guarantees that $\ray{u} = \psi(\ray{u} \cap \h_s) = \psi(\ray{u})$. In particular, we obtain that
$$
\erays(\psi(\ray{u})) = \erays(\ray{u}) = \{u\} = \omega(u),
$$
as desired.

Next, suppose that $u \in \mathfrak{B}$, and suppose that $u = (p,q)$.  Let $F_s$ denote the ray $\ray{u}$ and let $F$ denote the unique face of $\glngeo$ such that $\indFn \circ \erays(F) = \{p,q\},$ whose existence and uniqueness are guaranteed by \pref{thm:Characterization.Faces}.  \pref{thm:chamberStructure}\pref{assert:chamberRaysPrelim} implies that $F$ is non-negatively spanned by $\ray{p}$ and $\ray{q}$.  It follows from the definition of $\rayFn$ on elements of $\mathfrak{B}$ that $F_s = F \cap \h_s$.
The definition of $\mathfrak{B}$ implies that $F$ is bi-signed, and therefore \pref{prop:bijectForFaces}\pref{assert:psiInteractWithFaces} guarantees that $F = \psi(F_s)$.  In particular, we obtain that
$$
\erays(\psi(\ray{u})) = \erays(\psi(F_s)) = \erays(F) = \omega(u),
$$
as desired.
\end{proof}

\begin{proof}[Proof of \pref{thm:extremeRaysOfSln}]
It follows from \pref{lem:explainDefOfRayForSln} that $\rayFn$ defines an injection from $\mathfrak{N} \cup \mathfrak{B}$ to the set of extreme rays of $\slngeo$.

To finish the proof, we need to show that every extreme ray of $\slngeo$ is of the form $\ray{u}$ for some $u \in \mathfrak{N} \cup \mathfrak{B}$.
Recall that extreme rays of an incidence geometry are 1-faces.
Suppose that $r$ is an extreme ray of $\glngeo$.  Let $u = \omega^{-1}\left(\ind{\erays(\psi(r))}\right)$, which exists by \pref{prop:bijectForFaces}\pref{assert:psiBijectionCombinatorial}.  \pref{lem:explainDefOfRayForSln} guarantees that
$$
\ind{\erays(\psi(\ray{u}))} = \omega(u) = \ind{\erays(\psi(r))}.
$$
Because $\ray{u}$ and $r$ are 1-faces of $\slngeo$, \pref{prop:bijectForFaces}\pref{assert:psiBijectionCombinatorial} implies that $r = \ray{u}$.  The theorem follows. 
\end{proof}

\subsection{Extreme rays in chambers and faces of \slngeoForTitle}
\label{app:extremeRaysOfFacesAndChambers}

To classify the extreme rays of $\slngeo$ that lie in a given face, we will use \pref{lem:explainDefOfRayForSln}.  We will prove the following generalization of Parts \pref{assert:nullRay}, \pref{assert:nonNullRay}, and \pref{assert:extremeRayCount} of \pref{thm:chamberRays}.

\begin{theorem}
\label{thm:faceRays}
Let $F$ be a $k$-face of $\slngeo$ and let $C = \erays\left(\psi(F)\right)$, where $\psi$ denotes the bijection defined in \pref{def:psiForFaces}.
\begin{enumerate}[(a)]
\item \label{assert:extremeRayCriterion} An extreme ray $r$ of $\slngeo$ lies in $F$ if and only if $\ind{r} \in C$.
\item \label{assert:extremeRayFaceCount} $F$ has $|C_+|\cdot|C_-| + |C_0|$ extreme rays in $\slngeo$.  In particular, $F$ is simplicial if and only if $\min\{|C_+|,|C_-|\} \le 1$.
\end{enumerate}
\end{theorem}
\begin{proof}
Suppose that $r$ is an extreme ray of $\slngeo$.  \pref{lem:explainDefOfRayForSln} implies that
$$
\ind{\erays(\psi(r))} = \omega(\ind{r}).
$$
Because $r$ and $F$ are faces of $\slngeo$, \ref{prop:bijectForFaces}(\ref{assert:psiIsoPosets}) implies that
$$
r \subseteq F \Leftrightarrow \ind{\erays(\psi(r))} \subseteq C.
$$
It follows
$$
r \subseteq F \Leftrightarrow \omega(\ind{r}) \subseteq C \Leftrightarrow \ind{r} \in C,
$$
which is Part \pref{assert:extremeRayCriterion}.

It remains to prove Part \pref{assert:extremeRayFaceCount}.
Note that $C \cap \mathfrak{N} = \left|C_0\right|$.
\pref{prop:bijectForFaces}\pref{assert:psiBijectionCombinatorial} guarantees that $C$ is either a null $k$-chain or a bi-signed $(k+1)$-chain in $\posetz_n$.  Because $C$ is a chain, every pair of elements of $C$ is comparable.  We see immediately that $\left|\mathfrak{B} \cap C^2\right| = \left|C_+\right| \cdot \left|C_-\right|$.
It follows that $C$ contains $|C_+|\cdot|C_-| + |C_0|$.  Part \pref{assert:extremeRayCriterion} implies that $F$ contains $|C_+|\cdot|C_-| + |C_0|$ extreme rays of $\slngeo$.

A direct calculation shows that $F$ is the non-negative span of the extreme rays of $\slngeo$ that it contains.  It follows that $F$ is simplicial if and only if $|C_+|\cdot|C_-| + |C_0| = k$.  We divide into cases based on whether $C$ is null to prove that $|C_+|\cdot|C_-| + |C_0| = k$ if and only if $\min\{|C_+|,|C_-|\} \le 1$.
\begin{itemize}
\item Case 1: $C$ is a null $k$-chain.  Then, we have $C_+ = C_- = \varnothing$ by definition, so that $\min\{|C_+|,|C_-|\} \le 1$.
We also obtain that 
$$
|C_+|\cdot|C_-| + |C_0| = |C_0| = |C| = k
$$
so that $F$ is simplicial.
\item Case 2: $C$ is a bi-signed $(k+1)$-chain.   Note that
$$
|C_+| + |C_-| = |C| - |C_0| = k+1-|C_0|.
$$
A direct calculation shows that
$$
|C_+|\cdot|C_-| + |C_0| = k \Leftrightarrow (|C_+|-1)(|C_-|-1)=1.
$$
Therefore, $F$ is simplicial if and only if $1 \in \{|C_+|,|C_-|\}.$
The fact $C$ is bi-signed implies that $|C_+|,|C_-| \ge 1$, so that $F$ is simplicial if and only if $\min\{|C_+|,|C_-|\} \le 1$.
\end{itemize}
\pref{prop:bijectForFaces}\pref{assert:psiBijectionCombinatorial} guarantees that the cases exhaust all the possibilities for the chain $C$.  The theorem follows.
\end{proof}

In order to connect extreme rays of $\glngeo$ to subsets, we will need the following simple lemma.

\begin{lemma}
\label{lem:charactVectorAndExtRays}
Suppose that $C$ is a chamber of $\glngeo$ and that $p \in \posetz_n$.  Then, $p$ is an extreme ray of $C$ if and only if
\[\sigma(p) = \sum_{k=n-\ell(p)+1}^n s_k.\]
\end{lemma}
\begin{proof}
Immediate from \pref{thm:chamberStructure}\pref{assert:chamberRaysPrelim}.
\end{proof}

Interpreting \pref{thm:faceRays} in the language of subsets and characteristic vectors of a chamber yields Parts \pref{assert:nullRay}, \pref{assert:nonNullRay}, and \pref{assert:extremeRayCount} of \pref{thm:chamberRays}.

\begin{proof}[Proof of Parts \pref{assert:nullRay}, \pref{assert:nonNullRay}, and \pref{assert:extremeRayCount} of \pref{thm:chamberRays}]
Note that because $n \ge 3$, there are no null $(n-1)$-chains in $\posetz_n$.
It follows from \pref{prop:chamberBijection} and \pref{prop:bijectForFaces}\pref{assert:psiInteractWithFaces}
that $\phi$ and $\psi$ define inverse bijections between the set  $\mathfrak{C}_{\text{bi}}$ of bi-signed chambers of $\glngeo$ and the set $\mathfrak{C}_s$ of chambers of $\slngeo$, where $\phi$ was defined in \pref{def:phiForChambers} and $\psi$ was defined in \pref{def:psiForFaces}.

Suppose that $C$ is a bi-signed chamber of $\glngeo$ and let $C_s = \phi(C) = C \cap \h_s$ denote the corresponding chamber of $\slngeo$.  Let $M = \ind{\erays(C)}$ denote the $n$-chain in $\posetz_n$ corresponding to $C$.  \pref{thm:faceRays}\pref{assert:extremeRayCriterion} implies that an extreme ray $\ray{u}$ (with $u \in \mathfrak{N} \cup \mathfrak{B}$) of $\slngeo$ lies in $C_s$ if and only if $u \in M$.  In particular, we obtain that if $p \in \mathfrak{N}$ (resp.\ $(p_1,p_2) \in \mathfrak{B}$), then $\ray{p} \in C_s$ (resp.\ $\ray{p_1,p_2} \in C_s$) if and only if $p \in M$ (resp.\ $p_1,p_2 \in M$), which occurs if and only if $\ray{p} \in C$ (resp.\ $\ray{p_i} \in C$ for $i=1,2$).
A direct calculation from \pref{lem:charactVectorAndExtRays} yields Parts \pref{assert:nullRay} and \pref{assert:nonNullRay} of the theorem.

It remains to prove Part \pref{assert:extremeRayCount}.
Note that $\ell_0$ (resp.\ $\ell_+,\ell_-$) counts the number of elements of $M$ of zero (resp.\ positive, negative) signature.
Therefore, \pref{thm:faceRays}\pref{assert:extremeRayFaceCount} implies that $C_s$ has $\ell_+ \cdot \ell_- + \ell_0$ extreme rays in $\slngeo$ and that $C_s$ is simplicial if and only if $\min\{\ell_+,\ell_-\} \le 1$.  Recall that $M$ is bi-signed (because $n \ge 3$) so that $\ell_+,\ell_- \ge 1$.  It follows that $C_s$ is simplicial if and only if $\min\{\ell_+,\ell_-\} \le 1,$ and the theorem follows.
\end{proof}

\subsection{Simplicial chambers of \slngeoForTitle}
\label{app:simplicialChambers}

We now apply a bijective argument  to deduce \pref{thm:chamberRays}\pref{assert:countSimplicialChambers} from \pref{thm:faceRays}\pref{assert:extremeRayFaceCount}.

\begin{proof}[Proof of \pref{thm:chamberRays}\pref{assert:countSimplicialChambers}]
The theorem is obvious for $n \le 4.$  We assume for the remainder of the proof that $n \ge 5$.

As in the proof of Parts \pref{assert:nullRay}, \pref{assert:nonNullRay}, and \pref{assert:extremeRayCount} of \pref{thm:chamberRays}, the assumption that $n \ge 3$ ensures that $\rayFn \circ \erays \circ \psi$ defines a bijection from the set $\mathfrak{C}_s$ of chambers of $\slngeo$ to the set $\Chn_{\text{bi}}(n,n)$ of bi-signed $n$-chains in $\posetz_n$.
In light of \pref{thm:faceRays}\pref{assert:extremeRayFaceCount}, it suffices to count bi-signed $n$-chains $C \subset \posetz_n$ such that $\min\{|C_+|,|C_-|\} \le 1$.  We call such chains \emph{admissible} and let $a(n)$ denote the number of admissible chains in $\posetz_n$ so that $a(n)$ is the number of simplicial chambers in $\slngeo$.

For all bi-signed chains $C$, we have $|C_+|,|C_-| \ge 1$ so that $\min\{|C_+|,|C_-|\} = 1$. Note that every element of $\posetz_n$ of odd level has non-zero signature because $\ell(p) + \sigma(p)$ is even for all $p \in \posetz_n$.  It follows that if $C$ is an $n$-chain, then $|C_+|+|C_-| \ge \ceil*{\frac{n}{2}} \ge 3$ because $n \ge 5$ and $C$, being an $n$-chain, contains an element of every level in $[n]$.
This discussion motives the following definition.  Call an $n$- chain $C \subset \posetz_n$ \emph{negatively} (resp.\ \emph{positively}) \emph{admissible} if $|C_-| = 1$.  The discussion of this paragraph shows the set of admissible chains is the disjoint union of the set of positively admissible chains and the set of negatively admissible chains.  Let $a_+(n)$ denote the number of positively admissible chains in $\posetz_n$.  By symmetry, there are also $a_+(n)$ negatively admissible chains in $\posetz_n$, so that there are $2a_+(n)$ admissible chains in $\posetz_n$ in total.

\begin{figure}
\centering{
\raisebox{0cm}{  \includegraphics[scale=.7]{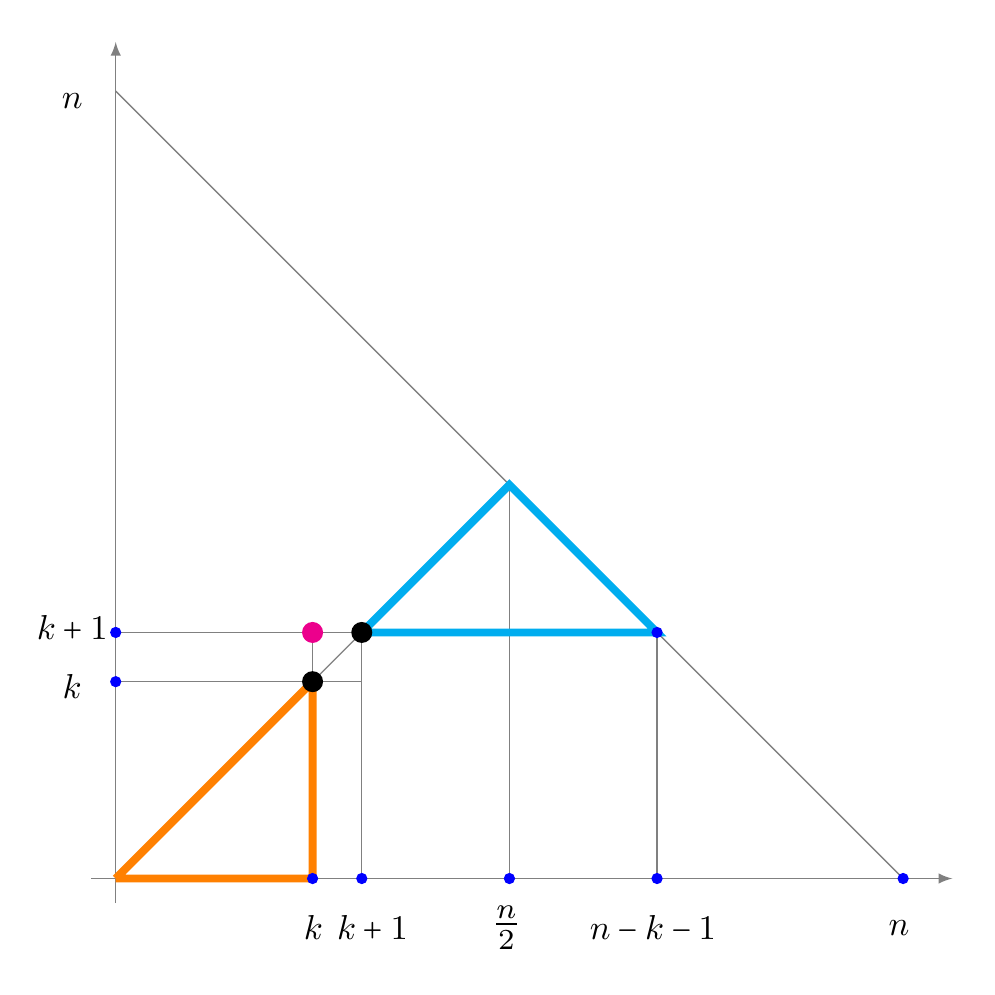} }
\raisebox{0cm}{  \includegraphics[scale=.7]{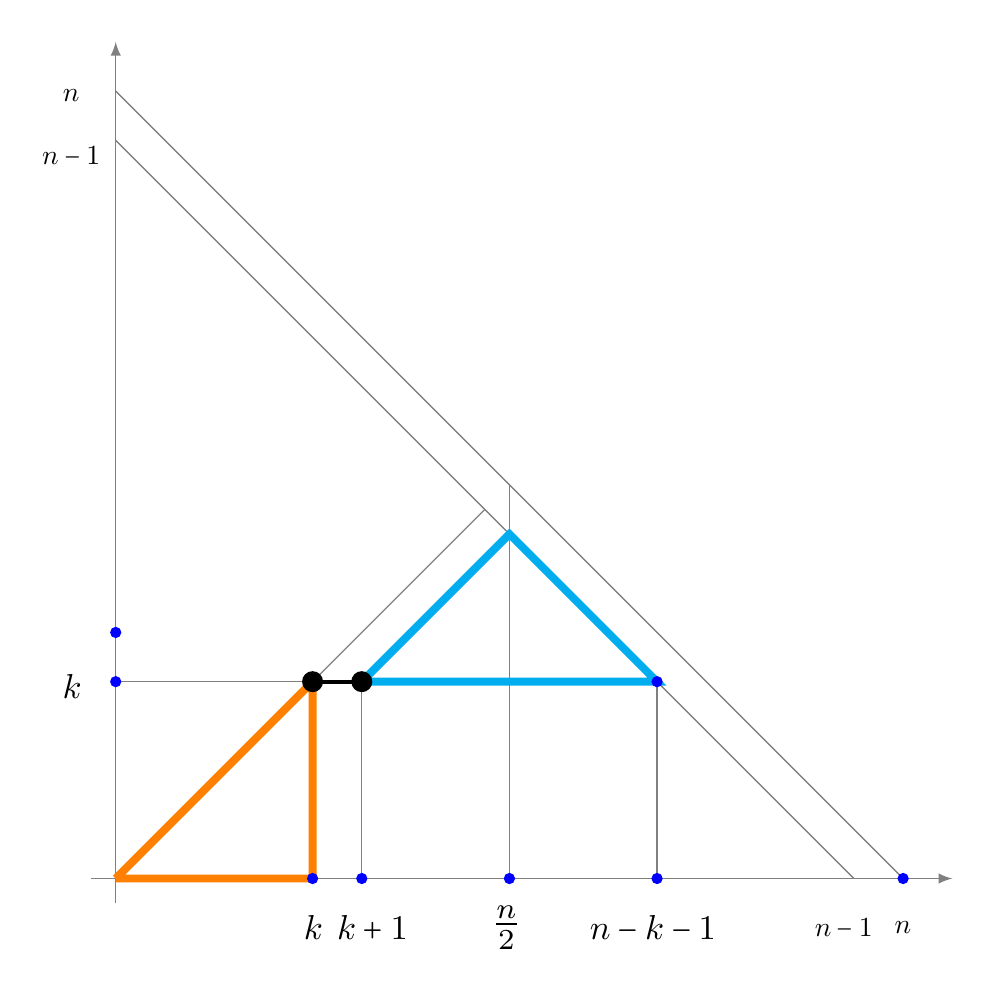} }

}
\caption{ A Negative admissible chain corresponding to a  simplicial chamber. 
  \label{fig:NegativeAdmissibleChain}}
\end{figure}

We compute $a_+(n)$ by using a bijection. A northeastern lattice path  of length $n$ starting at the origin  can be represented by a $n$-sequence of plusses and minuses  $(\epsilon_1, \epsilon_2, \ldots, \epsilon_n)$. 
When $\epsilon_i=+1$ (resp.\ $\epsilon_i=-1$) indicates that the $i$th step is a translation by $(1,0)$ (resp.\ $(0,1)$). 
We call ${\mathcal A}_+(n)$ the set of lattice paths of $\poset_n$ that cross the line $y=x+1$ exactly at one point. They are in bijection with the set of positively admissible $n$-chains. 
We call ${\mathcal A}_0(n)$ the set of lattice paths of $\posetz_n$ that do not cross the line $y=x$, but are allowed to touch it. 
A lattice path does not cross the line $y=x$  if and only if    for all $i$ in $[n]$  the  partial sum $S_i=\epsilon_1+\cdots +\epsilon_i$ is nonnegative. 
A lattice path crosses the line $y=x$ exactly once if and only if there is one and only one index $i\in[n]$ such that the partial sum $S_i=-1$. 
We define a bijection between ${\mathcal A}_+(n)$ and ${\mathcal A}_0(n-1)$. 

The image of a path of ${\mathcal A}_+(n)$ defined by the sequence $(\epsilon_1, \ldots, \epsilon_n)$ is the sequence obtained by removing the unique entry $i$ for which the partial sum $S_i$ is equal to $-1$. 
The inverse image consists of identifying the last entry $j$ of a sequence of ${\mathcal A}_0(n-1)$  at which $S_j$ is zero and introducing the entry $-1$ just before $j$. 
If  the sequence does not have a vanishing partial sum, introduce an entry $-1$ in front of the first entry of the sequence. 

\pref{prop:DyckWithNoRestrictedEndpt} asserts that $$\big|{{\mathcal A}_0(n)}\big|= \binom{n}{\ceil*{n/2}}.$$ 
It follows from the bijection between ${\mathcal A}_+(n)$ and ${\mathcal A}_0(n)$ that  
$$
a_+(n)=\big|{\mathcal A}_+(n)\big|=\big|{\mathcal A}_0(n-1)\big|=\binom{n-1}{\floor*{\frac{n-1}{2}}}.
$$
\end{proof}

\section{Recursions to count chains and ensembles with bounds on signature}
\label{app:Recursions}

To prove \pref{prop:A0} (resp.\ \pref{prop:C0}), we will derive recurrence relations satisfied by the sequence $a(n,k,\sigma)$ in \pref{app:chainRecursions} (resp.\ $c(n,k,\sigma)$ in \pref{app:ensembleRecursions}).
As a result, we will recover in \pref{app:recoverUnCounts} the enumerations of faces and flats of $\glngeo$ first proven in \cite{gl}.
In the next section, we will manipulate generating to solve the recursions.

\subsection{Recursions to count chains with bounds on signature}
\label{app:chainRecursions}
The following lemma is an unsimplified version of the recursion satisfied by the sequence $a(n,k,\sigma)$ that counts the number of elements of the set $\Chn(n,k,\sigma)$.  We will simplify the formula in \pref{prop:simpleRecursionChn}.

\begin{lemma}
\label{lem:uglyRecursionChn}
For all $n \ge 1$ and $\sigma \ge 0$, the number $a(n,k,\sigma)$ of  $k$-chains in $\posetz_n$ all of whose elements have signature at most $\sigma$ satisfy  the following recurrence relation
$$
a(n,k,\sigma)=\delta_{k,0} + \sum_{j=1}^n \sum_{t=0}^{j} a(n-j,k-1,\sigma+2t-j).
$$
\end{lemma}
The strategy of the proof is to sum over the possible minima of an element of $\Chn(n,k,\sigma)$.
\begin{proof}
The lemma is obvious for $k \le 0$  since $\Chn(n,k,\sigma)$ is empty for $k< 0$ and consists of one point when $k=0$.  Assume for the remainder of the proof that $k \ge 1$.
Define $\mu_{n,k,\sigma}: \Chn(n,k,\sigma) \to \posetz_n$ by $\mu_{n,k,\sigma}(S) = \min S,$ which is defined because $\Chn(n,k,\sigma)$ consists of chains.
We will enumerate the fibers $\mu_{n,k,\sigma}^{-1}(p)$.

\begin{figure}
\centering{
\raisebox{0cm}{  \includegraphics[scale=.7]{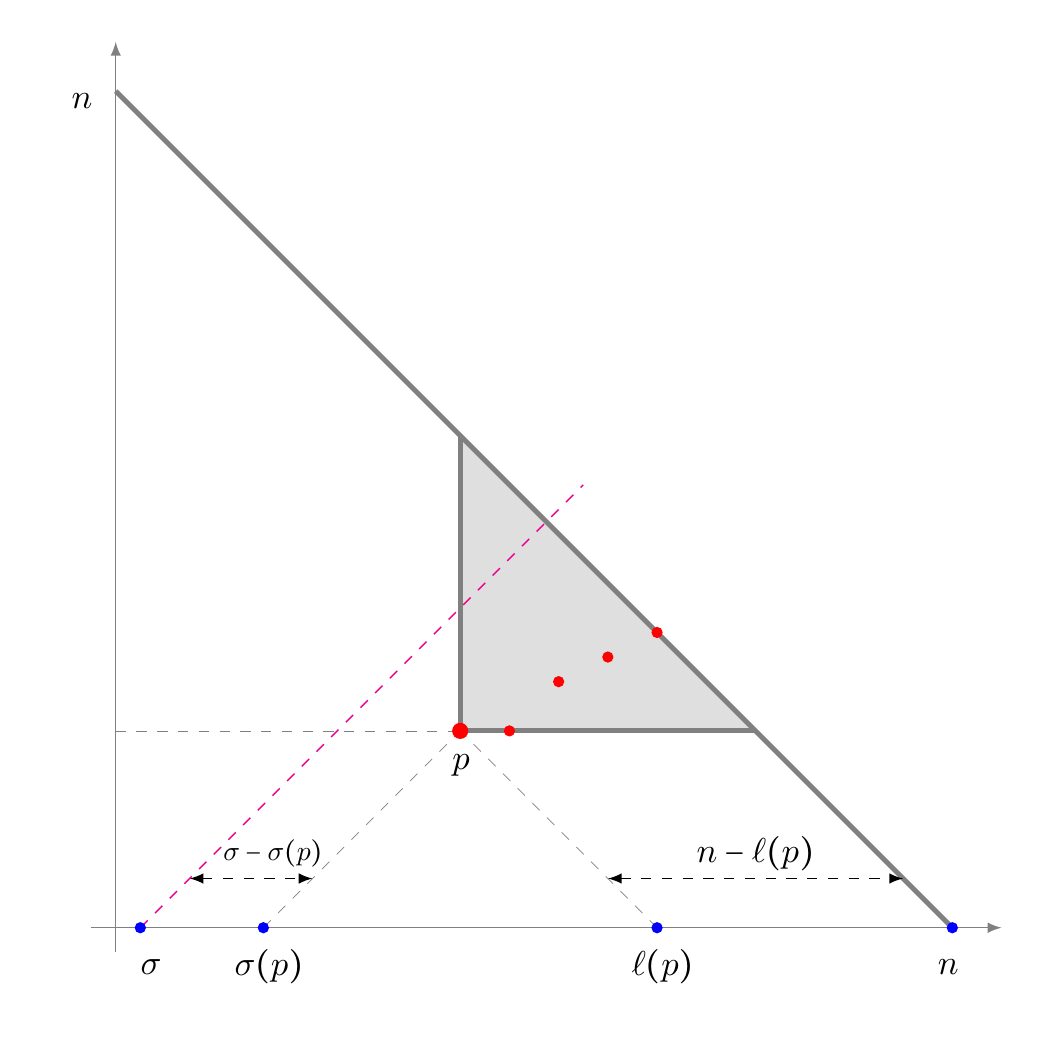} }

}\caption{ {\bf The geometry behind the  recursion formula for $a(n,k,\sigma)$ (Lemma \ref{lem:uglyRecursionChn})}. 
Given a  point $p=(p_x,p_y)\in \posetz_n$. The elements of $\Chn(n,k,\sigma)$ starting at $p$ are in bijection with the set 
$\Chn(n-\ell(p),k-,\sigma-\sigma(p))$.  This is because all the points of the chain are inside the colored triangle which corresponds to the interval $(p,\infty)\cap \posetz_n$. 
Since $p$ is already accounted for, there are $k-1$ points left to determine. They are all in the interval   $(p,\infty)\cap \posetz_n$, which  is in bijection with $\posetz_{n-\ell(p)}$. 
The bijection is a translation $\theta_{n,k,\sigma}$. Under such a translation, the signature is shifted by $-\sigma(p)$. 
  \label{fig:uglyRecursionChn}}
\end{figure}

Suppose that $p \in \posetz_n$.  Define a function $\theta_{n,k,\sigma}: \mu_{n,k,\sigma}^{-1}(p) \to \Chn(n-\ell(p),k-1,\sigma-\sigma(p))$ by 
$$
\theta_{n,k,\sigma}(S)=\left(-p + S\right) \setminus \{(0,0)\}.
$$
The function $\theta_{n,k,\sigma}$ translates a chain $S$ by shifting $p$ to $\hat{0}$, and then removes $\hat{0}$ from the result.
After a shift by $p$, the level (resp.\ signature) is shifted by  $-\ell(p)$ (resp.\  $-\sigma(p)$). 
Thus, translation by $p$ induces a bijection from the set $\Chn(n-\ell(p),k-1,\sigma-\sigma(p))$ to the set of chains of  $\Chn(n,k,\sigma)$ that start at $p$.
Define a function $\omega_{n,k,\sigma}: \Chn(n-\ell(p),k-1,\sigma-\sigma(p)) \to \mu_{n,k,\sigma}^{-1}(p)$ by
$$
\omega_{n,k,\sigma}(S)=\left(p + S\right) \cup \{p\}.
$$
The function $\omega_{n,k,\sigma}$ translates a chain $S$ by shifting $\hat{0}$ to $p$, and then includes $p$ in the result.
We see immediately that $\theta_{n,k,\sigma}$ and $\omega_{n,k,\sigma}$ are well-defined and form a pair of inverse bijections.  It follows that
 $$
\left|\mu_{n,k,\sigma}^{-1}(p)\right| = a(n-\ell(p),k-1,\sigma-\sigma(p)).
$$

We sum this equation over $p \in \posetz_n$ to complete the proof of the lemma.  Note that
$$
\left\{\sigma(p) \mid \ell(p) = j\right\} = \{-j,2-j,\ldots,j-2,j\} = \{j-2t \mid 0 \le t \le j\}.
$$
It follows that
$$
\sum_{\ell(p) = j}a(n-\ell(p),k-1,\sigma-\sigma(p)) = \sum_{t=0}^\ell a(n-\ell,k-1,\sigma+2t-\ell).
$$
Summing over the possible levels $j$ yields that
$$
a(n,k,\sigma)=\sum_{\ell=1}^n \sum_{t=0}^{\ell} a(n-\ell,k-1,\sigma+2t-\ell)
$$
for $k \ge 1$. We have to introduce the correction term $\delta_{k,0}$ in order to deal with the case of $k \le 0$.
\end{proof}

The following proposition simplifies the recursive formula of \pref{lem:uglyRecursionChn} to contain only finitely many terms.

\begin{proposition}
\label{prop:simpleRecursionChn}\quad 
The number $a(n,k,\sigma)$ of  $k$-chains in $\posetz_n$ all of whose elements have signature at most $\sigma$ satisfies  the following recursion relation
\begin{enumerate}[(a)]
\item 
\label{assert:genSimpleRecursionChn}
For all $n \ge 3$ and $\sigma \ge 1$, we have
$$
\begin{aligned}
a(n,k,\sigma) &=  a(n-1,k,\sigma-1)+a(n-1,k,\sigma+1)-a(n-2,k,\sigma)\\
&+a(n-1,k-1,\sigma-1)+a(n-1,k-1,\sigma+1)-a(n-2,k-1,\sigma)
\end{aligned}
$$

\item
\label{assert:zeroSimpleRecursionChn}
For all $n \ge 2$, we have 
$$
a(n,k,0)= a(n-1,k,1)+a(n-1,k-1,1)
$$
\end{enumerate}
\end{proposition}
\begin{proof}
First, we prove \pref{prop:simpleRecursionChn}\pref{assert:genSimpleRecursionChn}.
For $n \ge 2$ and $\sigma \ge 0$, subtracting the results of the applications of \pref{lem:uglyRecursionChn} to $(n,k,\sigma)$ and $(n-1,k,\sigma+1)$ yields that
\begin{equation}
\label{eq:genSimpleStep1Chn}
a(n,k,\sigma)-a(n-1,k,\sigma+1) =
a(n-1,k-1,\sigma+1)+\sum_{j=1}^n a(n-j,k-1,\sigma-j).
\end{equation}
Subtracting the results of the applications of the previous equation  \pref{eq:genSimpleStep1Chn} to $(n,k,\sigma)$ and $(n-1,k,\sigma-1)$ yields that
\begin{multline*}
a(n,k,\sigma)-a(n-1,k,\sigma-1)-a(n-1,k,\sigma+1)+a(n-2,k,\sigma)\\
=a(n-1,k-1,\sigma+1)-a(n-2,k-1,\sigma)+a(n-1,k-1,\sigma-1),
\end{multline*}
and the Proposition \pref{prop:simpleRecursionChn}\pref{assert:genSimpleRecursionChn} follows.

Next, we prove \pref{prop:simpleRecursionChn}\pref{assert:zeroSimpleRecursionChn}.  Applying \pref{eq:genSimpleStep1Chn} to $(n,k,0)$ yields that
$$
a(n,k,0)-a(n-1,k,1) = a(n-1,k-1,1)
$$
because $a(n,k,\sigma) = 0$ for $\sigma < 0$, and \pref{prop:simpleRecursionChn}\pref{assert:zeroSimpleRecursionChn} follows.
\end{proof}

\subsection{Recursions to count ensembles with bounds on signature}
\label{app:ensembleRecursions} We now perform the analogous calculations for ensembles.
The following lemma is an unsimplified version of the recursion satisfied by the sequence $c(n,k,\sigma)$.  We will simplify the formula in \pref{prop:simpleRecursionEns}.

\begin{lemma}
\label{lem:uglyRecursionEns}
For all $n \ge 1$ and $\sigma \ge 0$, we have
\begin{multline*}
c(n,k,\sigma)=c(n-1,k-1,\sigma-1)+c(n-1,k-1,\sigma+1)-c(n-2,k-2,\sigma)\\
+\delta_{k,0} + \sum_{j=2}^n \sum_{t=1}^{j-1} c(n-j,k-1,\sigma+2t-j).
\end{multline*}
\end{lemma}
The strategy of the proof of \pref{lem:uglyRecursionEns} is the same as that of \pref{lem:uglyRecursionChn}, but the analysis is more delicate due to the subtleties involved in the definition of an ensemble.  Indeed, the functions $\mu_{n,k,\sigma}$,$\theta_{n,k,\sigma}$, and $\omega_{n,k,\sigma}$ from the proof of \pref{lem:uglyRecursionChn} reappear in similar contexts, albeit with different domains and codomains.
\begin{proof}
The lemma is obvious for $k \le 0$.  Assume for the remainder of the proof that $k \ge 1$.
Define a function $\mu_{n,k,\sigma}: \Ens(n,k,\sigma) \to \poset_n$ by $\mu_{n,k,\sigma}(S) = \inf S$, where $\inf S$ is defined because $\poset_n$ is a clearly a finite meet-semilattice.  We will enumerate the fibers of $\mu_{n,k,\sigma}$ using several slightly different bijections.  Let $p \in \poset_n$.
\begin{itemize}
\item Case 1: $\ell(p) > 1$ and $\sigma(p) \in [1-\ell(p),\ell(p)-1] \cap \mathbb{Z}$.  Define a function $\theta_{n,k,\sigma}: \mu_{n,k,\sigma}^{-1}(p) \to \Ens(n-\ell(p),k-1,\sigma-\sigma(p))$ by 
$$
\theta_{n,k,\sigma}(S)=\left(-p + S\right) \setminus \{(0,0)\}.
$$
The function $\theta_{n,k,\sigma}$ translates an ensemble $S$ by shifting $p$ to $\hat{0}$, and then removes $\hat{0}$ from the result.
Define a function $\omega_{n,k,\sigma}: \Ens(n-\ell(p),k-1,\sigma-\sigma(p)) \to \mu_{n,k,\sigma}^{-1}(p)$ by
$$
\omega_{n,k,\sigma}(S)=\left(p + S\right) \cup \{p\}.
$$
The function $\omega_{n,k,\sigma}$ translates an ensemble $S$ by shifting $\hat{0}$ to $p$, and then includes $p$ in the result.
A direct calculation shows that $\theta_{n,k,\sigma}$ and $\omega_{n,k,\sigma}$ are well-defined and mutually inverse.  It follows that
$$
\left|\mu_{n,k,\sigma}^{-1}(p)\right| = c(n-\ell(p),k-1,\sigma-\sigma(p)).
$$
\item Case 2: $\ell(p) = 0$ or ($\ell(p)=1$ and $\sigma(p)=1$).  We group the two possibilities together.  Let $\Xi_{n,k,\sigma} = \mu_{n,k,\sigma}^{-1}(0,0) \cup \mu_{n,k,\sigma}^{-1}(1,0)$.  Define a function $\theta'_{n,k,\sigma}: \Xi_{n,k,\sigma} \to \Ens(n-1,k-1,\sigma-1)$ by
$$
\theta'_{n,k,\sigma}(S) = \left(-(1,0)+S\right) \cap \posetz_{n-1}.
$$
The function $\theta'_{n,k,\sigma}$ translates an ensemble $S$ by shifting $(1,0)$ to $\hat{0}$, then removes $\hat{0}$ and any elements of the translate with negative $x$-coordinate.
Define a function $\phi'_{n,k,\sigma}: \Ens(n-1,k-1,\sigma-1) \to \Xi_{n,k,\sigma}$ by
$$
\phi'_{n,k,\sigma}(S) = \left((1,0)+S\right) \cup \{q \in S \mid \sigma(q)=-\ell(q)\} \cup \{(1,0)\}
$$
The function $\phi'_{n,k,\sigma}$ translates an ensemble $S$ by shifting $\hat{0}$ to $(1,0)$, then includes $(1,0)$ in the result and completes the lowest interval of the result to have minimum $\hat{0}$.  The subtlety of completing the lowest interval arises because intervals with minima along the $x$-axis are not permitted in the definition of an ensemble.  Because the additional elements
$$
\{q \in S \mid \sigma(q)=-\ell(q)\}
$$
all have negative signature, their signatures cannot exceed the upper bound $\sigma$ involved in the definition of $\Ens(n,k,\sigma)$.

It follows from the definition of an ensemble that $\theta'_{n,k,\sigma}$ and $\phi$ are well-defined and mutually inverse.  It follows that
$$
\left|\Xi_{n,k,\sigma}\right| = c(n-1,k-1,\sigma-1).
$$
\item Case 3: $\ell(p) = 1$ and $\sigma(p) = -1$.  Define $\theta_{n,k,\sigma}$ as in Case 1.  Note that $(1,1) \notin S$ for all $S \in \mu_{n,k,\sigma}^{-1}(p)$.  Indeed, if $(1,1) \in S$ and $\inf S = (0,1)$, the axioms defining ensembles imply that $(0,1) \in S$ and hence $(1,0) \in S$, which contradicts the hypothesis that $\inf S = (0,1)$.  It follows that $(1,0) \notin \theta_{n,k,\sigma}(S)$ for all $S \in \mu_{n,k,\sigma}^{-1}(p)$.

Therefore, $\theta_{n,k,\sigma}$ gives rise to a function
$$
\theta_{n,k,\sigma}: \mu_{n,k,\sigma}^{-1}(p) \to \left(\Ens(n-1,k-1,\sigma+1) \setminus \Xi_{n-1,k-1,\sigma+1}\right).
$$
Define
$$
\omega_{n,k,\sigma}: \left(\Ens(n-1,k-1,\sigma+1) \setminus \Xi_{n-1,k-1,\sigma+1}\right) \to \mu_{n,k,\sigma}^{-1}(p)
$$
as in Case 1.  The functions $\theta_{n,k,\sigma}$ and $\omega_{n,k,\sigma}$ both clearly behave exactly as they did in Case 1.  We see immediately that $\theta_{n,k,\sigma}$ and $\omega_{n,k,\sigma}$ are mutually inverse.  Case 2 then implies that
$$
\left|\mu_{n,k,\sigma}^{-1}(p)\right| = c(n-1,k-1,\sigma+1)-c(n-2,k-2,\sigma).
$$
\item Case 4: $\ell(p) > 1$ and $\sigma(p) \in \{-\ell(p),\ell(p)\}$.  The definition of an ensemble ensures that $\mu_{n,k,\sigma}^{-1}(p) = \varnothing.$
\end{itemize}
Summing the results of the casework yields the proposition.
We have to introduce the correction term $\delta_{k,0}$ in order to deal with the case of $k \le 0$.
\end{proof}

The following proposition simplifies the recursive formula of \pref{lem:uglyRecursionEns} to contain only finitely many terms.  It is analogous to \pref{prop:simpleRecursionChn}.

\begin{proposition}
\label{prop:simpleRecursionEns}
For all $n \ge 3$ and $\sigma \ge 1$, we have
\begin{multline}
\label{eq:genSimpleRecursionEns}
c(n,k,\sigma)=c(n-1,k,\sigma-1)+c(n-1,k,\sigma+1)-c(n-2,k,\sigma)\\
+c(n-1,k-1,\sigma-1)+c(n-1,k-1,\sigma+1)+c(n-1,k-1,\sigma+1)\\
-c(n-2,k-1,\sigma-2)-c(n-2,k-1,\sigma)-c(n-2,k-1,\sigma+2)\\
+c(n-3,k-1,\sigma-1)+c(n-3,k-1,\sigma+1)-c(n-2,k-2,\sigma)\\
+c(n-3,k-2,\sigma+1)+c(n-3,k-2,\sigma-1)-c(n-4,k-2,\sigma).
\end{multline}
For all $n \ge 2$, we have
\begin{multline}
\label{eq:zeroSimpleRecursionEns}
c(n,k,0)=c(n-1,k,1)+c(n-1,k-1,1)-c(n-2,k-1,2)\\
-c(n-2,k-2,0)+c(n-3,k-2,1).
\end{multline}
\end{proposition}
The proof of \pref{prop:simpleRecursionEns} is very similar to that of \pref{prop:simpleRecursionChn}.
\begin{proof}
First, we prove \pref{eq:genSimpleRecursionEns}.
For $n \ge 2$ and $\sigma \ge 0$, subtracting the results of the applications of \pref{lem:uglyRecursionEns} to $(n,k,\sigma)$ and $(n-1,k,\sigma+1)$ yields that
\begin{multline}
\label{eq:genSimpleStep1Ens}
c(n,k,\sigma)-c(n-1,k,\sigma+1) =
c(n-1,k-1,\sigma-1)+c(n-1,k-1,\sigma+1)\\
-c(n-2,k-1,\sigma+2)-c(n-2,k-2,\sigma)\\
+c(n-3,k-2,\sigma+1)+\sum_{j=3}^n c(n-j,k-1,\sigma+2-j).
\end{multline}
Subtracting the results of the applications of \pref{eq:genSimpleStep1Ens} to $(n,k,\sigma)$ and $(n-1,k,\sigma-1)$ yields that
\begin{multline}
c(n,k,\sigma)-c(n-1,k,\sigma-1)-c(n-1,k,\sigma+1)+c(n-2,k,\sigma)\\
=c(n-1,k-1,\sigma-1)+c(n-1,k-1,\sigma+1)+c(n-1,k-1,\sigma+1)\\
-c(n-2,k-1,\sigma-2)-c(n-2,k-1,\sigma)-c(n-2,k-1,\sigma+2)\\
+c(n-3,k-1,\sigma-1)+c(n-3,k-1,\sigma+1)-c(n-2,k-2,\sigma)\\
+c(n-3,k-2,\sigma+1)+c(n-3,k-2,\sigma-1)-c(n-4,k-2,\sigma),
\end{multline}
and \pref{eq:genSimpleRecursionEns} follows.

Next, we prove \pref{eq:zeroSimpleRecursionEns}.  Applying \pref{eq:genSimpleStep1Ens} to $(n,k,0)$ yields that
\begin{multline*}
c(n,k,0)-c(n-1,k,1) =
c(n-1,k-1,1)-c(n-2,k-1,\sigma)\\
-c(n-2,k-2,0)+c(n-3,k-2,1)
\end{multline*}
because $c(n,k,\sigma) = 0$ for $w < 0,$ and \pref{eq:zeroSimpleRecursionEns} follows.
\end{proof}

\subsection{Recovering the enumerations of faces and flats for   \glngeoForTitle}
\label{app:recoverUnCounts}
We can specialize recursions for $a(n,k,\sigma)$ and $c(n,k,\sigma)$ to recursions for $g(n,k)$ and $h(n,k)$, respectively, due to the following observation.

\begin{lemma}
\label{lem:UnIsHighW}
For all $\sigma \ge n$, we have
\begin{align*}
a(n,k,\sigma) &= g(n,k)\\
c(n,k,\sigma) &= h(n,k).
\end{align*}
\end{lemma}
\begin{proof}
Simply recall that every element of $\posetz_n$ has signature at most $n$.
\end{proof}

In light of \pref{lem:UnIsHighW}, the recursions of Propositions~\ref{prop:simpleRecursionChn} and~\ref{prop:simpleRecursionEns} can be specialized to yield recursive formulas for $g(n,k)$ and $h(n,k),$ respectively.

\begin{proposition}  \label{prop:recursionsForUn}
For all $n \ge 3,$ the number of $k$-flats and $k$-faces in $\glngeo$ satisfy respectively the following recurrence relations
\begin{enumerate}[(a)]
\item \label{assert:UnFaceRecursion}

$$
g(n,k)=2g(n-1,k)-g(n-2,k)+2g(n-1,k-1)-g(n-2,k-1)
$$
\item 
\label{assert:UnFlatRecursion}
\begin{multline*}
h(n,k)=2h(n-1,k)-h(n-2,k)+2h(n-1,k-1)-3h(n-2,k-1)\\
+2h(n-3,k-1)-h(n-2,k-2)+2h(n-3,k-2)-h(n-4,k-2).
\end{multline*}
\end{enumerate}
\end{proposition}
\begin{proof}
\pref{assert:UnFaceRecursion} follows from \pref{prop:simpleRecursionChn} and \pref{lem:UnIsHighW}.
\pref{assert:UnFlatRecursion} follows from \pref{prop:simpleRecursionEns} and \pref{lem:UnIsHighW}.
\end{proof}

We can recover the enumerations of faces and flats of $\glngeo$ from \pref{prop:recursionsForUn} (see \cite[Theorems~\ref*{gln-thm:faceCount} and~\ref*{gln-thm:flatCount}]{gl}).

\begin{proof}[Alternative Proof of \pref{thm:glnFaces}]
Consider the polynomial
$$
P_1 = 1-2x+x^2-2xy+x^2y \in R^\times.
$$
By explicit computation of $g(n,k)$ for $n \ge 3$, one can verify that
$$
G(x,y) = 1+(1+2 y) x+(1+5 y+4 y^2) x^2+(1+9 y+16 y^2+8 y^3) x^3+O(x^4)
$$
so that
$$
P_1 \cdot G = 1-x + O(x^4).
$$
It follows from \pref{prop:recursionsForUn}\pref{assert:UnFaceRecursion} in \pref{prop:recursionsForUn} that  the product $P_1 \cdot G$ has no terms of degree at least 4 in $x$.
Therefore, we have
$$
P_1 \cdot G = 1-x
$$
which yields that
$$
G = \frac{1-x}{P_1} = \frac{1-x}{1-2x+x^2-2xy+x^2y},
$$
as desired.
\end{proof}

\begin{proof}[Alternative Proof of \pref{thm:glnFlats}]
Consider the polynomial
$$
Q_1 = 1 - 2 x + x^2 - 2 x y + 3 x^2 y - 2 x^3 y + x^2 y^2 - 2 x^3 y^2 + 
 x^4 y^2 \in R^\times.
$$
By explicit computation of $h(n,k)$ for $n \le 3$, one can verify that
$$
H(x,y) = 1+(1+y) x+(1+3 y+y^2) x^2+(1+5 y+6 y^2+y^3) x^3+O(x^4)
$$
so that
$$
Q_1 \cdot H = 1 - x - x y + 2 x^2 y - x^3 y + O(x^4).
$$
It follows from  \pref{prop:recursionsForUn}\pref{assert:UnFlatRecursion} in \pref{prop:recursionsForUn} that $Q_1\cdot H$ has no terms of degree at least 4 in $x$.  Therefore, we have
$$
Q_1 \cdot H = 1 - x - x y + 2 x^2 y - x^3 y,
$$
which yields that
$$
H = \frac{1 - x - x y + 2 x^2 y - x^3 y}{Q_1},
$$
as desired.
\end{proof}

\section{Algebraic manipulations to prove Propositions~\ref{prop:A0} and~\ref{prop:C0}}
\label{app:AlgManipulation}

In Section~\ref{app:FuncEqn}, we convert Propositions~\ref{prop:genFuncRelationFaces} and~\ref{prop:genFuncRelationFlats} into equations satisfied by $A,C,$ respectively.  In \pref{app:functionalEqSolveA0C0}, we solve the equations to prove Propositions~\ref{prop:A0} and~\ref{prop:C0}.

\subsection{Functional equations satisfied by the generating functions}
\label{app:FuncEqn}
Translating \pref{prop:simpleRecursionChn} into generating function form and incorporating initial conditions for the recursion yields the following result.

\begin{proposition}
\label{prop:genFuncRelationFaces}
We have
\[
A \cdot P + A_0 (K_0z - J_0z - J_1x) = \rho \cdot z,
\]
in $R\llbracket z \rrbracket,$ where $\rho$ is an auxiliary series defined as
$$
\rho = \rho(x,y,z)= \frac{1 - x z}{1-z} \in R\llbracket z\rrbracket
$$
and $J_0,K_0,P$ are auxiliary polynomials defined as
\begin{align*}
J_0 =1+x^2+x^2y,\quad 
J_1 =-(1+y)\quad
K_0 =1,\quad
P = J_1x+J_0z+J_1xz^2.
\end{align*}
\end{proposition}

In order to prove \pref{prop:genFuncRelationFaces}, we will need to show that the coefficient of $x^ny^kz^\sigma$ in the left-hand-side stabilizes as $\sigma \to \infty$.  One step in the proof of this stabilization is established in the following lemma, which we will also use later in the proof of the analogue of \pref{prop:genFuncRelationFaces} for flats (\pref{prop:genFuncRelationFlats}).

\begin{lemma}
\label{lem:errorStabilization}
For all $\sigma \ge 0$, we have
\begin{align*}
A_\sigma -A_{\sigma-1} &= O(x^\sigma)\\
C_\sigma -C_{\sigma-1} &= O(x^\sigma).
\end{align*}
\end{lemma}
\begin{proof}
Suppose that $n < \sigma$.  Because every element of $\posetz_n$ has signature at most $n$, every $k$-chain (resp.\ $k$-ensemble) in $\posetz_n$ lies in $\Chn(n,k,\sigma-1)$ (resp.\ $\Ens(n,k,\sigma-1)$).  In particular, we have $\Chn(n,k,\sigma-1) = \Chn(n,k,\sigma)$ (resp.\ $\Ens(n,k,\sigma-1) = \Ens(n,k,\sigma)$), so that $a(n,k,\sigma-1) = a(n,k,\sigma)$ (resp.\ $c(n,k,\sigma-1)=c(n,k,\sigma)$).  The lemma follows.
\end{proof}

\begin{proof}[Proof of \pref{prop:genFuncRelationFaces}]
We need the Taylor expansion of $\rho$ in $z$ to prove the proposition.  Define
\begin{equation}
\label{eq:defnOfRhoW}
\rho_\sigma  = \rho_\sigma (x,y) = \begin{cases}
1 & \text{if }  \sigma = 0\\
1-x& \text{if } \sigma  > 0.
\end{cases} \in \mathbb{Z}[x,y]
\end{equation}
so that
$$
\rho = \sum_{\sigma=0}^\infty \rho_\sigma (x,y) z^\sigma \in R\llbracket z\rrbracket.
$$
For $\sigma \in \mathbb{Z}_{\ge 0}$, define  the series $\rho'_\sigma  \in R$ by
$$
\rho'_\sigma  = \begin{cases}
J_1A_{\sigma-1}x+J_0A_\sigma +J_1A_{\sigma+1}x & \text{if }  \sigma   \ge 1\\
K_0A_0+J_1A_1x& \text{if } \sigma  = 0.
\end{cases}
$$
Here, we define $A_{-1} = 0$.  The definitions of $A$ and $P$ ensure that 
$$
\frac{A-A_0}{z} \cdot P + A_0(K_0+J_1z)\\
= \sum_{\sigma=0}^\infty \rho'_\sigma z^\sigma.
$$
Multiplying by $z$ and grouping together terms involving $A_0$ yields the identity
$$
A \cdot P + A_0 (K_0z-J_0z-J_1x) = \sum_{\sigma=0}^\infty \rho'_\sigma z^{\sigma+1}.
$$
Therefore, it suffices to prove that $\rho_\sigma =\rho'_\sigma$ in order to prove the proposition.

It follows from \pref{lem:errorStabilization} that
\begin{align*}
\rho'_\sigma -\rho'_{\sigma-1} &= J_1x(A_{\sigma-1}-A_{\sigma-2}+A_{\sigma+1}-A_\sigma )+J_0(A_\sigma -A_{\sigma-1})\\
&=J_1x \cdot O(x^{\sigma-1}) + J_0 \cdot O(x^\sigma) = O(x^\sigma).
\end{align*}
\pref{prop:simpleRecursionEns} implies that $\rho'_\sigma$ contains no terms with degree at least 4 in $x$ for all $\sigma$.  It follows
that $\rho'_3=\rho'_\sigma$ for all $\sigma \ge 3$.
In light of \pref{eq:defnOfRhoW} and the fact that $\rho'_\sigma$ contains no terms with degree at least 4 in $x$, it suffices to prove the equalities
\begin{equation}
\rho'_0 = 1+O(x^4), \quad 
\rho'_1 = 1-x+O(x^4), \quad 
\rho'_2 = 1-x+O(x^4), 
\quad \rho'_3 = 1-x+O(x^4)
 \label{eq:rho'desire}.
\end{equation}
We can the compute $\rho'_0,\rho'_1,\rho'_2,\rho'_3$ by using an expansion of $A$ to order $x^3$.  Explicit enumeration of chains in $\posetz_n$ for $n \le 3$ yields that
\begin{equation}
\begin{aligned}
A_0 &= 1 + x(1+y) + x^2(1+3y+2y^2) + x^3(1+5y+7y^2+3y^3) + O(x^4)\\
A_1 &= 1 + x(1+2y) + x^2(1+4y+3y^2) + x^3(1+7y+12y^2+6y^3) + O(x^4)\\
A_2 &= 1 + x(1+2y) + x^2(1+5y+4y^2) + x^3(1+8y+14y^2+7y^3) + O(x^4)\\
A_3 &= 1 + x(1+2y) + x^2(1+5y+4y^2) + x^3(1+9y+16y^2+8y^3) + O(x^4)
%A_4 &= 1 + x(1+2y) + x^2(1+5y+4y^2) + x^3(1+9y+16y^2+8y^3) + O(x^4)
\end{aligned}
\end{equation}
It follows from the definition of $\rho'_\sigma$ that 
\begin{align*}
\rho'_0 = 1+O(x^4), \quad
\rho'_1 = 1-x+O(x^4),\quad
\rho'_2 = 1-x+O(x^4),\quad
\rho'_3 = 1-x+O(x^4),
\end{align*}
which are precisely the identities \pref{eq:rho'desire} that we needed to prove.
\end{proof}

The analogue of \pref{prop:genFuncRelationFaces} for flats is the following result.

\begin{proposition}
\label{prop:genFuncRelationFlats}
We have
\begin{equation}
C \cdot Q + C_1 (M_1z^2-L_1xz^2-L_2x^2z)+C_0 (M_0z^2-L_0z^2-L_1xz-L_2x^2) = \xi\cdot z^2
\end{equation}
in $R\llbracket z\rrbracket$, where $\xi$ is an auxiliary series defined as
$$
\xi = \xi(x,y,z) = x - x^2 y + x^3 y + \frac{1 - x - x y + 2 x^2 y - x^3 y}{1 - z} \in R\llbracket z\rrbracket.
$$
and $L_0,L_1,L_2,M_0,M_1,P$ are auxiliary polynomials defined as
\begin{align*}
L_0 &= 1 + x^2 + x^2 y + x^2 y^2 + x^4 y^2\\
L_1 &= -(1 + y + x^2 y + x^2 y^2)\\
L_2 &= y\\
M_0 &= 1 + x^2 y^2\\
M_1 &= -(x + x y + x^3 y^2)\\
Q &= L_2x^2 + L_1 xz + L_0 z^2 + L_1 xz^3 + L_2 x^2z^4.
\end{align*}
\end{proposition}

The proof of \pref{prop:genFuncRelationFlats} is similar to that of \pref{prop:genFuncRelationFaces}.

\begin{proof}
We need the Taylor expansion of $\xi$ in $z$ to prove the proposition.  Define
\begin{equation}
\label{eq:defnOfVarpiW}
\xi_\sigma  = \xi_\sigma (x,y) = \begin{cases}
1-xy+x^2y & \text{if }  \sigma = 0\\
1-x-xy+2x^2y-x^3y & \text{if } \sigma  > 0.
\end{cases} \in \mathbb{Z}[x,y]
\end{equation}
so that
$$
\xi = \sum_{\sigma=0}^\infty \xi_\sigma (x,y) z^\sigma \in R\llbracket z\rrbracket.
$$
For $w \in \mathbb{Z}_{\ge 0}$, define series $\xi'_\sigma  \in R$ by
$$
\xi'_\sigma  = \begin{cases}
L_2C_{\sigma-2}x^2+L_1C_{\sigma-1}x+L_0C_\sigma +L_1C_{\sigma+1}x+L_2C_{\sigma+2}x^2 & \text{if }  \sigma  \ge 1\\
M_0C_0+M_1C_1+M_2C_2 & \text{if }  \sigma = 0.
\end{cases}
$$
Here, we define $C_{-1} = 0$.  The definitions of $C$ and $Q$ ensure that 
\begin{multline*}
\frac{C-C_1z-C_0}{z^2} \cdot Q + C_1 (M_1z^2+L_0z+L_1xz^2+L_2x^2z^3)+C_0 (M_0+L_1xz+L_2x^2z^2)\\
= \sum_{\sigma=0}^\infty \xi'_\sigma z^\sigma.
\end{multline*}
Multiplying by $z^2$ and grouping together terms involving $C_0,C_1$ yields the identity
$$
C \cdot Q + C_1 (M_1z^2-L_1xz^2-L_2x^2z)+C_0 (M_0z^2-L_0z^2-L_1xz-L_2x^2) = \sum_{\sigma=0}^\infty \xi'_\sigma z^{\sigma+2}.
$$
Therefore, it suffices to prove that $\xi_\sigma =\xi'_\sigma$ in order to prove the proposition.

It follows from \pref{lem:errorStabilization} that
\begin{align*}
\xi'_\sigma -\xi'_{\sigma-1} &= \begin{array}{l}
L_2x^2(C_{\sigma-2}-C_{\sigma-3}+C_{\sigma+2}-C_{\sigma+1})\\
+L_1x(C_{\sigma-1}-C_{\sigma-2}+C_{\sigma+1}-C_\sigma )+L_0(C_\sigma -C_{\sigma-1})
\end{array}\\
&=L_2x^2\cdot O(x^{\sigma-2})+L_1x \cdot O(x^{\sigma-1}) + L_0 \cdot O(x^\sigma) = O(x^\sigma).
\end{align*}
\pref{prop:simpleRecursionEns} implies that $\xi'_\sigma$ contains no terms with degree at least 4 in $x$ for all $\sigma$.  It follows
that $\xi'_3=\xi'_\sigma$ for all $\sigma \ge 3$.
In light of \pref{eq:defnOfVarpiW} and the fact that $\xi'_\sigma$ contains no terms with degree at least 4 in $x$, it suffices to prove the equalities
\begin{equation}
\begin{aligned}
\xi'_0 &= 1-xy+x^2y+O(x^4)
&& \xi'_1 = 1-x-xy+2x^2y-x^3y+O(x^4) \\
\xi'_2 &= 1-x-xy+2x^2y-x^3y+O(x^4)
&&\xi'_3 = 1-x-xy+2x^2y-x^3y+O(x^4).
\end{aligned}
 \label{eq:tau'desire}
\end{equation}
We can the compute $\xi'_0,\xi'_1,\xi'_2,\xi'_3$ by using an expansion of $C$ to order $x^3$.  Explicit enumeration of ensembles in $\posetz_n$ for $n \le 3$ yields that
\begin{align*}
C_0 &= 1 + x + x^2 (1 + 2 y) + x^3 (1 + 3 y + 2 y^2) + O(x^4)\\
C_1 &= 1 + x (1 + y) + x^2 (1 + 3 y) + x^3 (1 + 5 y + 5 y^2) + O(x^4)\\
C_2 &= 1 + x (1 + y) + x^2 (1 + 3 y + y^2) + x^3 (1 + 5 y + 6 y^2) + O(x^4)\\
C_3 &= 1 + x (1 + y) + x^2 (1 + 3 y + y^2) + 
  x^3 (1 + 5 y + 6 y^2 + y^3) + O(x^4)\\
C_4 &= 1 + x (1 + y) + x^2 (1 + 3 y + y^2) + 
  x^3 (1 + 5 y + 6 y^2 + y^3) + O(x^4)\\
C_5 &= 1 + x (1 + y) + x^2 (1 + 3 y + y^2) + 
  x^3 (1 + 5 y + 6 y^2 + y^3) + O(x^4)
\end{align*}
It follows from the definition of $\xi'_\sigma$ that
\begin{align*}
\xi'_0 &= 1-xy+x^2y+O(x^4), &&
\xi'_1 = 1-x-xy+2x^2y-x^3y+O(x^4)\\
\xi'_2 &= 1-x-xy+2x^2y-x^3y+O(x^4), 
&& \xi'_3 = 1-x-xy+2x^2y-x^3y+O(x^4),
\end{align*}
which are precisely the identities \pref{eq:tau'desire} that we needed to prove.
\end{proof}

\subsection{Solving the functional equations}
\label{app:functionalEqSolveA0C0}
To compute $A_0$, we reduce the equation of \ref{prop:genFuncRelationFaces} modulo $P$.  In order to work with the quotient ring $R\llbracket z \rrbracket/(P)$ explicitly, we find a zero $z=\kappa$ of $P(x,y,z)$ with $\kappa \in R$.

It is convenient to change coordinates from $z$ to $t$ at various points during the proof of Propositions~\ref{prop:A0} and~\ref{prop:C0} in order to exploit the symmetries $P(x,y,z)=z^2P(x,y,z^{-1})$ and $Q(x,y,z)=z^4Q(x,y,z^{-1})$.
Define the series
$$
t = \frac{z}{z^2+1} \in \mathbb{Q}\llbracket z\rrbracket,
$$
so that
\begin{equation}
\label{eq:zFromt}
z = \frac{1-\sqrt{1-4t^2}}{2t} = \frac{2t}{1+\sqrt{1-4t^2}} \in \mathbb{Q}\llbracket t\rrbracket
\end{equation}
by the formal inverse function theorem.

\begin{proof}[Proof of \pref{prop:A0}]
Work in the notation of \pref{prop:genFuncRelationFaces}.
Note that
$$
P = (z^2+1)J_0\left(t+\frac{J_1}{J_0}x\right) \propto a (z-\kappa),
$$
where
$$
\kappa = \frac{-2J_1x}{J_0+\sqrt{J_0^2-4J_1^2x^2}}
$$
and $u \propto v$ denotes $u \in v\cdot \left(R\llbracket z \rrbracket\right)^\times.$

Because $\kappa$ lies in the maximal ideal of $R,$ the universal property of $R\llbracket z \rrbracket$ guarantees the existence of a unique $R$-algebra homomorphism $\pi: R\llbracket z \rrbracket \to R$ given by evaluation at $z=\kappa$.
Applying $\pi$ to the equation of \pref{prop:genFuncRelationFaces} yields the identity
$$
A_0(K_0\kappa-J_0\kappa-J_1x)=\rho(x,y,\kappa)\kappa
$$
in $R$.  It follows that
$$A_0 = \frac{\kappa \cdot \rho(x,y,\kappa)}{K_0\kappa-J_0\kappa-J_1x}
$$
because $K_0\kappa-J_0\kappa-J_1x \not= 0$ and $R$ is an integral domain.  Simplifying the equation  yields the theorem.
\end{proof}

\begin{remark}
\label{rem:doNotComputeA}
Note that one can apply \pref{prop:A0} and \pref{prop:genFuncRelationFaces} to determine $A$.
The expression for $A$ is complicated and not used in this paper, and so we do not give its explicit form.
\end{remark}

The computation of $C_0$ is similar, but slightly more delicate.
We will reduce the equation of \pref{prop:genFuncRelationFlats} modulo $P$.  In order to work with an explicit basis for the quotient ring $R\llbracket z \rrbracket/(P)$ as an $R$-module, we will find a monic polynomial in $R[z]$ that differs from $P$ only by multiplication by a unit in $\left(R\llbracket t \rrbracket\right)^\times$.  The existence of such a polynomial is guaranteed by the Weierstrass Preparation Theorem, but we find such a polynomial explicitly by exploiting the symmetry $Q(x,y,z)=z^4Q(x,y,z^{-1})$.  The symmetry allow us to avoid having to apply the quartic formula.

\begin{proof}[Proof of \pref{prop:C0}]
Work in the notation of \pref{prop:genFuncRelationFlats}.
The quadratic formula yields that
\begin{align*}
Q &= (z^2+1)^2(L_0-2L_2x^2)\cdot \left(t^2+\frac{L_1x}{L_0-2L_2}t+\frac{L_2x^2}{L_0-2L_2x^2}\right)\\
&=(z^2+1)^2(L_0-2L_2x^2)(t-r_1)(t-r_2)
\end{align*}
where
$$
r_1,r_2=x\cdot \frac{-L_1\pm\sqrt{L_1^2-4L_2(L_0-2L_2x^2)}}{2(L_0-2L_2x^2)}.
$$
Note that $r_1,r_2 \in R$ because
$$
\left.L_1^2-4L_2(L_0-2L_2x^2)\right|_{(x,y)=(0,0)} = 1.
$$
Combined with \pref{eq:zFromt}, it follows that
$$
Q \propto (z-\upsilon_1)(z-\upsilon_2),
$$
where
$$
\{\upsilon_i\} = \frac{x\left(-L_1\pm\sqrt{L_1^2-4L_2(L_0-2L_2x^2)}\right)}{L_0-2L_2x^2+\sqrt{(L_0-2xL_1)^2-x^2\left(-L_1\pm\sqrt{L_1^2-4L_2(L_0-2L_2x^2)}\right)^2}}
$$
and $u \propto v$ denotes $u \in v \cdot \left(R\llbracket z\rrbracket\right)^\times$.
Note that $\upsilon_1,\upsilon_2 \in R$ because all the expressions under the radical signs evaluate to 1 at $(x,y)=(0,0)$ and the denominators of $\upsilon_1,\upsilon_2$ evaluate to 1 at $(x,y)=(0,0)$.  The fact that the numerators of $\upsilon_1,\upsilon_2$ are divisible by $x$ ensures that $\upsilon_i(0,0)=0$ for $i \in \{1,2\}$.

Because $\upsilon_1,\upsilon_2$ lie in the maximal ideal of $R,$ the universal property of $R\llbracket z \rrbracket$ guarantees the existence of unique $R$-algebra homomorphisms $\pi_1,\pi_2: R\llbracket z \rrbracket \to R$ given by evaluation at $z=\upsilon_1$ and $z=\upsilon_2$, respectively.
Applying $\pi_1,\pi_2$ to the equation of \pref{prop:genFuncRelationFlats} yields the identities
$$
C_1(M_1\upsilon_i^2-L_1x\upsilon_i^2-L_2x^2\upsilon_i)+C_0(M_0\upsilon_i^2-L_0\upsilon_i^2-L_1x\upsilon_i-L_2)=\xi(x,y,\upsilon_i)\upsilon_i^2
$$
in $R$ for $i=1,2$.  This is a system of two linear equations in the two unknowns $C_0,C_1$, which we can solve explicitly.  When the sand settles, we obtain the theorem.
\end{proof}

\begin{remark}
The proof of \pref{prop:C0} also explains how to determine $C_1$.  One can then apply \pref{prop:genFuncRelationFlats} to determine $C$.  The expressions for $C_1$ and $C$ are complicated and not used in this paper, and so we do not give their explicit forms.
(See also \pref{rem:doNotComputeA}).
\end{remark}

\section{Basic combinatorics of northeastern lattice paths}
\label{app:NEpaths}
We review some basic results on enumerative combinatorics of northeastern lattice paths that were of use in the paper. 
For a general review, see \cite{LatticePaths} and references within.

\begin{definition}
A northeastern lattice path  is a sequence of lattice points $(p_0, p_1, \ldots, p_m)$ of $\mathbb{N}^2$ such that the difference of two consecutive points  is $(1,0)$ or $(0,1)$. 
The point $p_0$ and $p_m$ are respectively called the {\em start} and the end of the {\em path}. 
\end{definition}

\begin{proposition}\label{prop:NEpaths} For any $a,b \in \mathbb{Z}_{\geq 0}$, 
there are $\binom{a+b}{a}$ northeastern lattice paths connected $(0,0)$ to $(a,b)$.
\end{proposition}

\begin{proof}
Such paths are defined by choosing    $a$   $(1,0)$  steps within a total of $a+b$ steps. Those not selected are  $(0,1)$ steps. 
\end{proof}
\begin{proposition}
Northeastern lattice paths of length $m$ starting at the origin are in bijection with the power-set $\{-1,+1\}^{[m]}$.  
\end{proposition}
\begin{proof}Given a northeastern lattice path $(p_0,\ldots, p_m)$ starting at the origin, we define an element of $\{-1,+1\}^{[m]}$ by $(s_1,\ldots, s_m)$ such that $s_i=+1$ if $p_i-p_{i-1}=(1,0)$ and $s_i=-1$ if $p_i-p_{i-1}=(0,1)$.
The inverse is  given by 
$$
p_0=(0,0), \quad p_{i+1} =\begin{cases} p_i+(1,0)\quad \text{if} \quad s_i=+1 \\  p_i+(0,1)\quad \text{if} \quad s_i=-1\end{cases} 
$$
where $ i\in [m]$ and $(s_1, \ldots, s_n)\in \{-1,1\}^{[m]}$.
\end{proof}

Recall that the  level (resp.\ the signature) of a lattice point $(a,b)$ is $a+b$ (resp.\  $a-b$).
\begin{proposition}\label{prop:weaklyPath}
Given a point $(a,b)$ of nonnegative signature, there are 
$$
\binom{a+b}{a}-\binom{a+b}{a+1}=\frac{a-b+1}{a+b+1}\binom{a+b}{a}
$$
lattice paths connected $(0,0)$ and $(a,b)$ while staying weakly below the diagonal $y= x$.
\end{proposition}
\begin{proof}
The total number of paths from the origin to the point $(a,b)$ is $\binom{a+b}{a}$ as stated in \pref{prop:NEpaths}. We only have to remove those paths that cross the line $y=x$. These are the paths that intersect  the line $y=x+1$. 
Given a path from the origin to $(a,b)$ that touches the line $y=x+1$, we call the initial portion of the path, the path obtained by removing all the steps after it touches the line $y=x+1$ for the first time. 
For any northeastern path that touches the line $y=x+1$, consider the  path obtained by  replacing  the initial portion of the path by its mirror image with respect to the line $y=x+1$ while keeping the rest of the path unchanged. All such paths start at $(-1,1)$, the mirror-image  of $(1,0)$. 
 This shows that 
the paths that cross the diagonal are in bijection with the paths that start at $(-1,1)$ and reach the point $(a,b)$. There are $\binom{a+b}{a+1}$ of them.  Taking the difference gives the result. 
\end{proof}

\begin{remark} The numbers  $\frac{a-b+1}{a+b+1}\binom{a+b}{a}$  $(a\geq b)$ are called {\em ballot numbers} since they are solutions to the classical {\em ballot problem}. 	 
When $a=b=n$, the ballot numbers  become the Catalan numbers  
$$\frac{1}{n+1}\binom{2n}{n}.$$
\end{remark}

The following result was used in the enumeration of bi-signed $n$-chains and simplicial chambers of $\slngeo$.
\begin{proposition}
\label
{prop:DyckWithNoRestrictedEndpt}
There are 
$$\binom{n}{\floor*{\frac{n}{2}}}$$
lattice paths connecting the origin to an arbitrary point of level $n$  while staying weakly below the diagonal $y= x$.
\end{proposition}

\begin{proof}

Let $e(k)$ denote the number of such paths that end on $(k,n-k$).  Furthermore, note that $e(k) = 0$ for $k < \floor*{\frac{n}{2}}$ because $(k,n-k)$ lies above the line $y = x$ for $k <\floor*{\frac{n}{2}}$. 
For the others we have 
$$e(k) = \binom{n}{k} - \binom{n}{k+1}.$$
 It follows that
$$\sum_{k=\floor*{\frac{n}{2}}}^{n} e(k)$$
counts the number of northeastern lattice paths from $(0,0)$ to the line $x+y=n$ that lie in the half-plane $y \le x$.
The formula for $E(k)$ yields that
$$\sum_{k=\floor*{\frac{n}{2}}}^{n} e(k) = \binom{n}{\floor*{\frac{n}{2}}},$$
and the proposition follows.
\end{proof}

\begin{longtable}{rcl}
\caption{{\bf List of notations}}
\label{tab:notations}
\endfirsthead
\multicolumn{3}{c}%
{\tablename\ \thetable\ -- \textit{Continued from previous page}} \\ \\
\endhead
\\ \multicolumn{2}{r}{\textit{Continued on next page}} \\
\endfoot
\endlastfoot

$[n]$ &:& the set $\{1,2,\cdots, n\}$ \\
$\mathbb{R}$ &:& the set of real numbers \\
$\mathbb{R}_{\geq 0}$ &:& the set of non-negative real numbers \\
$\mathbb{N}$ &:& the set of non-negative integers\\
$|A|$ &:& the cardinality of a (finite) set $A$\\
$2^A$ &:& the power set of a set $A$\\
$\chi_S$ &:& the characteristic function of the set $S$
\\
$\epsilon$ &:& an element of $\{0,+,-\}$
\\
$\h$ &:& the split real form of the usual Cartan subalgebra of $\mathfrak{gl}_n$\\
$\h_s$ &:& the split real form of the usual Cartan subalgebra of $\mathfrak{sl}_n$\\
$\h^*$ &:& the linear dual of $\h$\\
$\sigma$ &:& the signature function $x_1+\cdots+x_n \in \h^*$\\
$E_0$ &:& the set of elements of $E$ of signature 0\\
$E_+$ &:& the set of elements of $E$ of positive signature\\
$E_-$ &:& the set of elements of $E$ of negative signature\\
$W$ &:& the  Weyl chamber in $\h$\\
$W_s$ &:& the Weyl chamber in $\h_s$\\
$V$ &:& the defining (vector) representation of $\mathfrak{gl}_n$ and $\mathfrak{sl}_n$\\
$\bigwedge^2$ &:& the antisymmetric representation of $\mathfrak{gl}_n$ and $\mathfrak{sl}_n$\\
& & \\
$\lambda^\perp$ &:& the vanishing locus in $\h$ of $\lambda \in \h^*$\\
$\lambda^+$ &:& the locus in $\h$ where $\lambda$ is non-negative\\
$\lambda^-$ &:& the locus in $\h$ where $\lambda$ is non-positive\\
$\weights$ &:& the set of weights of $V \oplus \bigwedge^2$\\
$\weights_0(S)$ &:& the set of elements of $\weights$ that vanish on $S$\\
$\weights_+(S)$ &:& \hspace{-5pt}\begin{tabular}{l}
the set of elements of $\weights$ that are non-negative\\
on $S$ but do not vanish on $S$
\end{tabular}\\
$\weights_-(S)$ &:& \hspace{-5pt}\begin{tabular}{l}
the set of elements of $\weights$ that are non-positive\\
on $S$ but do not vanish on $S$
\end{tabular}\\
\\
$\mathrm{I}(\mathfrak{g},{\bf R})$ &:& \hspace{-9pt} \begin{tabular}{l}
the hyperplane arrangement in the dual\\
Weyl chamber of $\mathfrak{g}$ given by weights of ${\bf R}$
\end{tabular}\\
$\mathfrak{C}$ &:& the set of chambers of $\glngeo$\\
$\mathfrak{C}_{\text{bi}}$ &:& the set of bi-signed chambers of $\glngeo$\\
$\mathfrak{C}_s$ &:& the set of chambers of $\slngeo$\\
$\mathfrak{F}$ &:& the set of faces of $\glngeo$\\
$\mathfrak{F}_{\text{null}}(n,k)$ &:& the set of null $k$-faces of $\glngeo$\\
$\mathfrak{F}_{\text{bi}}(n,k)$ &:& the set of bi-signed $k$-faces of $\glngeo$\\
$\mathfrak{F}_s$ &:& the set of faces of $\slngeo$\\
$\mathfrak{L}$ &:& the set of flats of $\glngeo$\\
$\mathfrak{L}_{\text{null}}(n,k)$ &:& the set of null $k$-flats of $\glngeo$\\
$\mathfrak{L}_{\text{bi}}(n,k)$ &:& the set of bi-signed $k$-flats of $\glngeo$\\
$\mathfrak{L}_{\text{bi},\text{unbal}}(n,k)$ &:& the set of bi-signed unbalanced $k$-flats of $\glngeo$\\
$\mathfrak{L}_s$ &:& the set of flats of $\slngeo$\\
\\
$\vecspan(\{v_1,\ldots,v_n\})$ &:& $\mathbb{R}_{\geq 0}\ v_1+\ldots+\mathbb{R}_{\geq 0}\  v_n$\\
$\erays(S)$ &:& the set of extreme rays of $\glngeo$ lying in $S$\\
$\ray{a,b}$ &:& the vector $(\underbrace{1,\ldots ,1}_{a}, \underbrace{0,\ldots ,0}_{n-a-b} , \underbrace{-1,\ldots ,-1}_{b}) \in \h$\\
$\ray{(a,b),(c,d)}$ &:& \hspace{-8pt}
\begin{tabular}{l}
the vector
$(\underbrace{\mu,\ldots,\mu}_a,\underbrace{\nu ,\ldots,\nu}_{c-a},\underbrace{0,\ldots,0}_{n-c-d},\underbrace{-\nu,\ldots,-\nu}_{d-b},\underbrace{-\mu,\ldots,-\mu}_b) \in \h_s,$\\
where $\mu=|b-a|+|d-c|$ and $\nu=|b-a|$
\end{tabular}\\
\\
$\ell(p)$ &:& the level $a+b$ of $p=(a,b)$\\
$\sigma(p)$ &:& the signature $a-b$ of $p=(a,b)$ \\
$\posetz_n$ &:& the poset of elements of $\mathbb{Z}_{\ge 0}^2$ with level in $[1,n]$\\
$\poset_n$ &:& the poset $\posetz_n \sqcup \{(0,0)\}$\\
$\poseti_n$ &:& the poset $\poset_n \sqcup \{\infty\}$\\
$\Chn_{\text{null}}(n,k)$ &:& the set of null $k$-chains in $\posetz_n$\\
$\Chn_{\text{bi}}(n,k)$ &:& the set of bi-signed $k$-chains in $\posetz_n$\\
$\Chn'(n,k,\sigma)$ &:& the set of $k$-chains of elements of signature at most $\sigma$ in $\posetz_n$\\
$\Chn(n,k,\sigma)$ &:& the set $\Chn'(n,k,\sigma)$ if $\sigma \ge 0$ and $\varnothing$ otherwise\\
$\Ens_{\text{null}}(n,k)$ &:& the set of null $k$-ensembles in $\posetz_n$\\
$\Ens_{\text{bi}}(n,k)$ &:& the set of bi-signed $k$-ensembles in $\posetz_n$\\
$\Ens_{\text{bi},\text{bal}}(n,k)$ &:& the set of bi-signed balanced $k$-ensembles in $\posetz_n$\\
$\Ens_{\text{bi},\text{unbal}}(n,k)$ &:& the set of bi-signed unbalanced $k$-ensembles in $\posetz_n$\\
$\Ens'(n,k,\sigma)$ &:& the set of $k$-ensembles of elements of signature at most $\sigma$ in $\posetz_n$\\
$\Ens(n,k,\sigma)$ &:& the set $\Ens'(n,k,\sigma)$ if $\sigma \ge 0$ and $\varnothing$ otherwise\\
\\
$\phi$ &:& the bijection from $\mathfrak{C}_s$ to $\mathfrak{C}$ defined in \pref{def:phiForChambers}\\
$\psi$ &:& the injection from $\mathfrak{F}_s$ to $\mathfrak{F}$ defined in \pref{def:psiForFaces}\\
$\tau$ &:& the injection from $\mathfrak{L}_s$ to $\mathfrak{L}$ defined in \pref{def:etaForFlats}\\
\\
$a(n,k,\sigma)$ &:& the size of $\Chn(n,k,\sigma)$\\
$b(n,k)$ &:& the number of $k$-faces of $\slngeo$\\
$c(n,k,\sigma)$ &:& the size of $\Ens(n,k,\sigma)$\\
$d(n,k)$ &:& the number of null $k$-faces of $\glngeo$\\
$f(n,k)$ &:& the number of $k$-flats of $\slngeo$\\
$g(n,k)$ &:& the number of $k$-faces of $\glngeo$\\
$h(n,k)$ &:& the number of $k$-flats of $\glngeo$\\
\\
$R$ &:& the formal power series ring $\mathbb{Q}\llbracket x,y\rrbracket$\\
$A_\sigma(x,y)$ &:& the generating function for $a(n,k,\sigma)$\\
$A(x,y,z)$ &:& the 3-variable generating function for $a(n,k,\sigma)$\\
$B(x,y)$ &:& the generating function for $b(n,k)$\\
$C_\sigma(x,y)$ &:& the generating function for $c(n,k,\sigma)$\\
$C(x,y,z)$ &:& the 3-variable generating function for $c(n,k,\sigma)$\\
$D(x,y)$ &:& the generating function for $d(n,k)$\\
$F(x,y)$ &:& the generating function for $f(n,k)$\\
$G(x,y)$ &:& the generating function for $g(n,k)$\\
$H(x,y)$ &:& the generating function for $h(n,k)$\\
\\
$\alpha$ &:& $1 - 2 x^2 + x^4 - 6 x^2 y + 2 x^4 y - 4 x^2 y^2 + x^4 y^2$ \\
$\beta$ &:& $1 - x + x^2 - x y - x^2 y - x^3 y + x^2 y^2 - x^3 y^2 + x^4 y^2$\\
$\gamma$ &:& $1 - 2 y - 2 x^2 y + y^2 + 8 x^2 y^2 + x^4 y^2 - 2 x^2 y^3 - 2 x^4 y^3 + x^4 y^4$\\
$\zeta$ &:& $1 + x^4 - 2 x^2 y - 2 x^4 y - 7 x^4 y^2 - 2 x^4 y^3 - 2 x^6 y^3 + x^4 y^4 + x^8 y^4$\\
$\eta$ &:& $2(x^2 + x^2 y + x^4 y + x^4 y^2)$
\end{longtable}

%\bibliography{mboyo_Bib_sln}
%\bibliographystyle{utphys}

\end{document}